\documentclass[11pt]{amsart}

\usepackage[T1]{fontenc}
\usepackage[utf8]{inputenc}
\usepackage{geometry}

\usepackage[pdfencoding=auto, psdextra, breaklinks=true]{hyperref}
\usepackage{xcolor}
\hypersetup{
  colorlinks,
  linkcolor={red!50!black},
  citecolor={green!40!black},
  urlcolor={blue!80!black}
}

\usepackage{enumitem}
\setlist{topsep=2mm, itemsep=0.8mm}

\usepackage{amsmath,amssymb,amsthm,bm,mathtools}
\usepackage{stmaryrd}
\numberwithin{equation}{section}

\usepackage{url}
\usepackage[capitalize]{cleveref}
\usepackage{csquotes} 
\usepackage[style=ext-alphabetic, backend=biber, sorting=nyt, doi=false, isbn=false,url=false, hyperref ,backrefstyle=none, maxnames=20, maxalphanames=20]{biblatex}
\DeclareFieldFormat
[article,inbook,incollection,inproceedings,patent,thesis,unpublished]
{titlecase:title}{\MakeSentenceCase*{#1}}
\newbibmacro{string+doiurlisbn}[1]{%
  \iffieldundef{doi}{%
      \iffieldundef{url}{%
          #1
        }{%
      \href{\thefield{url}}{#1}%
    }%
  }{%
    \href{https://dx.doi.org/\thefield{doi}}{#1}%
  }%
}
\DeclareFieldFormat{title}{\usebibmacro{string+doiurlisbn}{\mkbibemph{#1}}}
\DeclareFieldFormat[article,incollection,inproceedings,unpublished,misc,book]{title}%
    {\usebibmacro{string+doiurlisbn}{\mkbibquote{#1}}}
\renewbibmacro{in:}{%
  \ifentrytype{article}{}{\printtext{\bibstring{in}\intitlepunct}}}

\usepackage{xurl}

\newcommand{\from}{\colon}
\newcommand{\st}{\colon}
\newcommand{\embed}{\hookrightarrow}

\newcommand{\overbar}[1]{\mkern
  1.5mu\overline{\mkern-1.5mu#1\mkern-1.5mu}\mkern 1.5mu}
\newcommand{\conj}[1]{\overbar{#1}}
\newcommand{\wconj}[1]{\overline{#1}}
\renewcommand{\mod}{\ \mathrm{mod}\ }

\DeclarePairedDelimiter{\abs}{|}{|}

\DeclarePairedDelimiter{\norm}{\lVert}{\rVert}
\DeclarePairedDelimiter{\paren}{(}{)}
\DeclarePairedDelimiter{\set}{\{}{\}}
\DeclarePairedDelimiter{\floor}{\lfloor}{\rfloor}
\DeclarePairedDelimiter{\ceil}{\lceil}{\rceil}
\DeclarePairedDelimiter{\scal}{\langle}{\rangle}

\newcommand{\R}{\mathbb{R}}
\newcommand{\C}{\mathbb{C}}
\newcommand{\Q}{\mathbb{Q}}
\newcommand{\Z}{\mathbb{Z}}
\newcommand{\F}{\mathbb{F}}

\newcommand{\Proj}{\mathbb{P}}
\newcommand{\Half}{\mathcal{H}}
\newcommand{\Crv}{\mathcal{C}}
\newcommand{\Otilde}{\scalebox{1}[.9]{\ensuremath{\widetilde{O}}}}
\newcommand{\eps}{\varepsilon}
\newcommand{\Fund}{\mathcal{F}}
\newcommand{\Gts}{\Sigma}
\newcommand{\wrongc}{{\rm\texttt{WrongCandidate}}}
\newcommand{\insufp}{{\rm\texttt{InsufficientPrecision}}}
\newcommand{\unknown}{\rm\texttt{Unknown}}
\newcommand{\mf}[1]{\boldsymbol{#1}}

\DeclareMathOperator{\Jac}{Jac}
\DeclareMathOperator{\Sym}{Sym}
\DeclareMathOperator{\Sp}{Sp}
\DeclareMathOperator{\GSp}{GSp}
\DeclareMathOperator{\GL}{GL}
\DeclareMathOperator{\SL}{SL}
\DeclareMathOperator{\Tr}{Tr}
\DeclareMathOperator{\End}{End}
\DeclareMathOperator{\im}{Im}
\DeclareMathOperator{\re}{Re}
\DeclareMathOperator{\logp}{\log^+}
\DeclareMathOperator{\MF}{MF}
\DeclareMathOperator{\logpt}{\log_2^+}

\DeclareMathOperator{\M}{\mathsf{M}}

\newcommand{\dual}{\vee}

\newcommand{\mat}[4]{\left(\begin{matrix}#1&#2\\#3&#4\end{matrix}\right)}
\newcommand{\tmat}[4]{\left(\begin{smallmatrix}#1&#2\\#3&#4\end{smallmatrix}\right)}
\newcommand{\tvec}[2]{\left(\begin{smallmatrix}#1\\#2\end{smallmatrix}\right)}

\newtheorem{thm}[equation]{Theorem}
\crefname{thm}{Theorem}{Theorems}
\newtheorem{lem}[equation]{Lemma}
\crefname{lem}{Lemma}{Lemmas}
\newtheorem{prop}[equation]{Proposition}
\crefname{prop}{Proposition}{Propositions}
\newtheorem{cor}[equation]{Corollary}
\newtheorem{assumption}[equation]{Assumption}
\theoremstyle{definition}
\newtheorem{defn}[equation]{Definition}
\crefname{defn}{Definition}{Definitions}

\newtheorem{rem}[equation]{Remark}
\newtheorem{conv}[equation]{Convention}
\newtheorem{algo}[equation]{Algorithm}
\crefname{algo}{Algorithm}{Algorithms}

\newcommand{\algoinput}[1]{~\\ \noindent\textbf{Input:} #1}
\newcommand{\algooutput}[1]{~\\ \noindent\textbf{Output:} #1}

\title{Evaluating modular equations for abelian surfaces}
\author{Jean Kieffer}
\date{\today}

\begin{document}

\begin{abstract}
  We design efficient algorithms to evaluate modular equations of Siegel and
  Hilbert type for abelian surfaces over number fields or finite fields using
  complex approximations. Their output is provably correct when the associated
  graded ring of modular forms over~$\Z$ is explicitly known; this includes the
  Siegel case, and the Hilbert case for the quadratic fields of
  discriminant~$5$ and~$8$. As part of the proofs, we establish new correctness
  and complexity results for certain key numerical algorithms on period
  matrices in genus~2, namely the reduction algorithm to the fundamental domain, the AGM
  method, and computing big period matrices and RM structures.
\end{abstract}

\maketitle

\section{Introduction}
\label{sec:intro}

\subsection{Modular equations}
\label{subsec:intro-modeq}

Modular equations of Siegel and Hilbert type for abelian
surfaces~\cite{brokerModularPolynomialsGenus2009,
  milioQuasilinearTimeAlgorithm2015, martindaleHilbertModularPolynomials2020,
  milioModularPolynomialsIlbert2020} are higher-dimensional analogues of the
classical modular polynomials~$\Phi_\ell$ for elliptic curves. If~$\ell$ is a
prime, then the Siegel modular equations of level~$\ell$ are three rational
fractions in four variables~$J_1,J_2,J_3,X$, denoted by
\begin{displaymath}
  \Psi_{\ell,1},\Psi_{\ell,2},\Psi_{\ell,3}\in \Q(J_1,J_2,J_3)[X].
\end{displaymath}
They encode the presence of isogenies between principally polarized~(p.p.)
abelian surfaces. More precisely, let~$A$ and~$A'$ be two generic
p.p.~abelian surfaces over a field whose characteristic is prime to~$\ell$,
and denote their Igusa invariants by $(j_1,j_2,j_3)$ and $(j_1',j_2',j_3')$
respectively. Then the equalities
\begin{displaymath}
  \Psi_{\ell,1}(j_1,j_2,j_3,j_1') = 0
  \quad\text{and}\quad
  j_k' = \dfrac{\Psi_{\ell,k}(j_1,j_2,j_3,j_1')}{\partial_X\Psi_{\ell,1}(j_1,j_2,j_3,j_1')}
  \quad \text{for }k=2,3,
\end{displaymath}
where~$\partial_X$ denotes the partial derivative with respect to~$X$, hold if
and only if~$A$ and~$A'$ are $\ell$-isogenous (i.e.~related by an isogeny of
degree~$\ell^2$ with isotropic kernel in~$A[\ell]$) over an algebraic closure
of their base field. Similarly, modular equations of Hilbert type encode the presence of
certain cyclic isogenies between p.p.~abelian surfaces with real
multiplication~(RM) by a fixed real quadratic order.

From a computational point of view, modular equations allow us to detect and
compute isogenies without prior knowledge of their kernels:
see~\cite{elkiesEllipticModularCurves1998, bostanFastAlgorithmsComputing2008}
for elliptic curves, and~\cite{kiefferComputingIsogeniesModular2025} for
abelian surfaces. In contrast, all other available methods to compute
isogenies, to the author's knowledge, use an encoding of the kernel as
input~\cite{veluIsogeniesEntreCourbes1971,
  couveignesComputingFunctionsJacobians2015,
  lubiczComputingSeparableIsogenies2015,
  dudeanuCyclicIsogeniesAbelian2022}. This makes modular equations an essential
tool in the celebrated SEA algorithm for counting points on elliptic curves
over finite fields~\cite{schoofCountingPointsElliptic1995}, specifically in
Elkies's method, which uses isogenies to compute cyclic subgroups of the
elliptic curve.

However, modular equations for abelian surfaces are very large objects. For
instance, the total size of Siegel modular equations of level~$\ell$
is~$O(\ell^{15}\log\ell)$~\cite{kiefferDegreeHeightEstimates2022}, a
prohibitive upper bound that we expect to be accurate.  Already for~$\ell=3$,
their total size is approximately
410~MB~\cite{milioDatabaseModularPolynomials2016}. This has led to concerns
about their usability in practice.

\subsection{Main results}
\label{subsec:intro-main}

In this paper, we argue that precomputing modular equations in full is not an
efficient strategy in higher dimensions. In most contexts, including Elkies's
method, we only need \emph{evaluations} of modular equations, and possibly
their derivatives, at a given point $(j_1,j_2,j_3)\in L^3$, where~$L$ is a
number field; finite fields reduce to this case via lifts. These evaluations of
modular equations are much smaller than the full polynomials~$\Psi_{\ell,k}$,
so we should calculate them directly.

In this paper, we design algorithms to precisely do this. They are inspired by
existing methods to evaluate elliptic modular polynomials via complex
approximations~\cite{engeComputingModularPolynomials2009}, and build on the
certified evaluation of genus~2 theta constants in quasi-linear
time~\cite{kiefferCertifiedEwtonSchemes2022}. In certain key cases, their
complexity is quasi-linear in the output size, as exemplified by the following
result on Siegel modular equations.

\begin{thm}
  \label{thm:intro}
  Given prime numbers~$p$ and~$\ell$ and a generic tuple
  $(j_1,j_2,j_3)\in\F_p^3$, one can evaluate $\Psi_{\ell,k}(j_1,j_2,j_3,X)$ and
  $\partial_{J_n} \Psi_{\ell,k}(j_1,j_2,j_3,X)$ for $1\leq k,n\leq 3$ as
  elements of~$\F_p[X]$ in time $\Otilde(\ell^6\log p)$.
\end{thm}

Here ``generic'' means that~$(j_1,j_2,j_3)$ lies in $\mathcal{U}(\F_p)$ for a
certain nonempty Zariski open subset~$\mathcal{U}$ of affine 3-space
(independent of~$p$.) One can further describe~$\mathcal{U}$ as the nonvanishing locus of
certain explicitly computable modular forms. As usual, the
notation~$\Otilde(N)$ means~$O(N \log^k N)$ for some fixed~$k\in \Z_{\geq 0}$.

Similar results hold for Hilbert modular equations. We focus on the case of RM
by~$\Z_F$, the ring of integers in $F = \Q(\sqrt{5})$ or~$F = \Q(\sqrt{8})$, of
discriminant $\Delta_F\in \{5,8\}$. In these two cases, one can advantageously
replace the three Igusa invariants~$j_1,j_2,j_3$ by two \emph{Gundlach
  invariants} $g_1,g_2$, which are coordinates on the moduli space of
p.p.~abelian surfaces with RM by~$\Z_F$.  There is a notion of
$\beta$-isogenies between such abelian surfaces whenever~$\beta\in \Z_F$ is
totally positive~\cite{dudeanuCyclicIsogeniesAbelian2022}; their degree is
$N_{F/\Q}(\beta)$. The Hilbert modular equations of level~$\beta$ relate the
Gundlach invariants of the source and target of a $\beta$-isogeny.

\begin{thm}
  \label{thm:intro-hilbert}
  Given a totally positive element~$\beta\in \Z_F$ of prime norm~$\ell\in \Z$
  which is prime to~$\Delta_F$, a prime~$p$, and a generic
  pair~$(g_1,g_2)\in \smash{\F_p^2}$, one can evaluate the Hilbert modular equations of
  level~$\beta$ and their derivatives at~$(g_1,g_2)$ as elements of~$\F_p[X]$
  in time~$\Otilde(\ell^2\log p)$.
\end{thm}

In both \cref{thm:intro,thm:intro-hilbert}, we save a factor of~$\log p$ in the
running time when the input values are quotients of small integers: see for
instance \cref{thm:hilbert-nf} in the case $L=\Q$. The complexity estimates are small enough to
make Elkies's method viable for abelian surfaces over finite
fields~\cite{kiefferCountingPointsAbelian2022}.  An implementation of the
algorithms described in this paper, based on the C libraries
FLINT~\cite{theflintteamFLINTFastLibrary2024} and
Arb~\cite{johanssonArbEfficientArbitraryprecision2017}, is publicly
available~\cite{kiefferHDMELibraryEvaluation2022}, though not completely up to
date at the time of writing.

We can also handle other types of modular equations, for instance Hilbert
modular equations in Igusa invariants for general quadratic fields. The price
we pay is that the evaluation algorithm would become heuristic, as we explain
below.

Note that other methods based on isogeny
graphs~\cite{brokerModularPolynomialsIsogeny2012,
  sutherlandEvaluationModularPolynomials2013} allow us to evaluate
the elliptic modular polynomial~$\Phi_\ell$ over finite fields with better
asymptotic complexities than numerical methods, particularly in terms of
memory.  However, they crucially rely on writing~$\Phi_\ell$ in full over small
auxiliary finite fields. It is unclear whether this strategy can still be
competitive when the the dimension of the moduli space grows.

Although it is not the focus of the present paper, our methods would also apply
in genus~1 to compute the classical modular polynomials~$\Phi_\ell$,
recasting~\cite{engeComputingModularPolynomials2009} in a provably correct
setting. It seems plausible that this complex-analytic approach yields the
fastest algorithm to evaluate~$\Phi_\ell$ at a given point over~$\Q$, say, in
the feasible range.

\subsection{Overview of the algorithms}
\label{subsec:intro-overview}

Throughout, we rely on interval arithmetic as implemented in the Arb library
for numerical computations over~$\C$. Consider the Siegel case for
simplicity. Choose a base number field~$L$ of degree $d_L$ over~$\Q$ and Igusa
invariants $(j_1,j_2,j_3)\in L^3$. For each of the $d_L$ embeddings
$\mu:L\embed \C$, the evaluation algorithm for modular equations roughly
proceeds as follows.
\begin{enumerate}
\item \label{step:intro-periods} We compute a period matrix~$\tau$ in the
  Siegel upper half space~$\Half_2$ whose Igusa invariants are
  $\mu(j_1),\mu(j_2),\mu(j_3)$ with the AGM
  method~\cite[Chap.~9]{dupontMoyenneArithmeticogeometriqueSuites2006}.
\item \label{step:intro-enum} We enumerate suitable period matrices
  corresponding to abelian surfaces which are $\ell$- or $\beta$-isogenous to
  the abelian surface attached to~$\tau$.
\item \label{step:intro-theta} We evalute the Igusa invariants at these new
  period matrices, by reducing them to the fundamental domain in~$\Half_2$ and
  applying Dupont's
  algorithm~\cite[Chap.~10]{dupontMoyenneArithmeticogeometriqueSuites2006} (a
  Newton scheme around the AGM), proved to be correct
  in~\cite{kiefferCertifiedEwtonSchemes2022}.
\item \label{step:intro-poly} Finally, we
  recover~$\mu(\Psi_{\ell,k}(j_1,j_2,j_3))\in \C[X]$ from an analytic formula
  involving the Igusa invariants computed in step~\ref{step:intro-theta}.
\end{enumerate}

Combining these steps allows us to numerically evaluate modular equations up to
any desired precision~$N$, i.e.~up to an absolute error at most $2^{-N}$, in
quasi-linear time in~$N$; in practice, step~\ref{step:intro-theta} dominates
the running time.  (Let us mention here that step~\ref{step:intro-enum} is
considerably more involved in the Hilbert case, as we must take the RM
structure into account.) In the cases of \cref{thm:intro,thm:intro-hilbert}, we
actually take advantage of the explicit description of the corresponding graded
ring of modular forms with integral Fourier coefficients to separately evaluate
the numerator and denominator of modular equations in
step~\ref{step:intro-poly}.

After this numerical evaluation step, we must recognize each coefficient~$x$ of
the evaluated modular equations as an element of~$L$ from its complex
approximations. Ensuring the correctness of this step is delicate. When we
separately evaluate the numerator and denominator of modular equations, we can
ensure that the coefficients of~$x$ in a specific~$\Q$-basis of~$L$ are
\emph{integers} rather than rational numbers. Then, running the complex
computations at increasing precisions, for instance repeatedly doubling~$N$,
until these integers can be uniquely recognized is a provably correct
algorithm. (The explicit, but astronomical, upper bounds
of~\cite[Thm.~5.19]{kiefferDegreeHeightEstimates2022} on the height of~$x$ are
of no practical use.) In general, a heuristic approach is to compute tentative
values of~$x\in L$ from approximations at increasing precisions, and stop when
these values stabilize. The complexity of this heuristic algorithm is still
bounded above because the size of modular equations is
bounded~\cite[Thm.~1.1]{kiefferDegreeHeightEstimates2022}, and its result may
often be validated by certifying the existence of isogenies between abelian
surfaces as in~\cite{costaRigorousComputationEndomorphism2019}.

\subsection{Key ingredients of the proofs}
\label{subsec:intro-ingredients}

In the course of proving \cref{thm:intro,thm:intro-hilbert}, we establish new
results on certain key numerical algorithms on the moduli space of p.p.~abelian
surfaces which might be of independent interest.
\begin{itemize}
\item We translate Streng's complexity analysis of the reduction algorithm for
  period matrices in~\cite[§6]{strengComputingIgusaClass2014} into a completely
  numerical setting.
\item We revisit Dupont's formulation of the AGM method
  in~\cite[Chap.~9]{dupontMoyenneArithmeticogeometriqueSuites2006} to turn it
  into a provably correct algorithm with an explicit complexity estimate in
  terms of the input values, not only in terms of the required precision.
\item We describe a new algorithm which, given a (small) period matrix,
  computes the attached big period matrix in quasi-linear time. This addresses
  a disadvantage of the AGM method compared to numerical integration:
  see~\cite[Rem.~2.2.7]{costaRigorousComputationEndomorphism2019}.
\item Finally, we explain how to use big period matrices to compute RM
  structures on complex tori in a provably correct way, without any algebraic
  verification step.
\end{itemize}
This being done, completing the proofs of the main theorems is only a matter of
tracking down the precision losses in numerical computations throughout the
evaluation algorithm.

\subsection{Organization of the paper}
\label{subsec:intro-orga}

In~\cref{sec:modeq}, we recall the definition of Siegel and Hilbert modular
equations, and provide explicit formulas for their numerators and
denominators. In~\cref{sec:structure}, we describe the evaluation algorithms
for these modular equations in more detail. \Cref{sec:agm} is devoted to the key
numerical algorithms on period matrices as described above. We establish the
complexity of evaluating modular equations numerically in \cref{sec:precision},
and complete the proof of the main theorems in \cref{sec:proofs}.

\subsection{Notation}
\label{subsec:intro-notation}

Throughout this text, we use bold letters (such as $\mf{f}$ or $\mf{G}_2$) to
refer to modular functions. The symbols~$\pi$ and~$i$ have their usual values
-- that is, $\exp(\pi i) = -1$. We denote the $n\times n$ identity matrix by
$I_n$. If~$U$ is any matrix, we denote its transpose by $U^t$ and its inverse
transpose by $U^{-t}$.  If we write a $4\times 4$ matrix~$\gamma$ as
\begin{displaymath}
  \gamma =
  \begin{pmatrix}
    a & b \\ c & d
  \end{pmatrix},
\end{displaymath}
we always mean that $a,b,c,d$ are $2\times 2$ blocks.

We denote the absolute value of the largest coefficient in a polynomial
(resp.~vector or matrix)~$P$ by~$\abs{P}$. The notation~$\norm{P}$ refers to
the $L^2$ norm for vectors and the induced $L^2$ norm for matrices. We write
$\log_2$ for the logarithm in base~$2$. For $x\in\R$, we also define
\begin{displaymath}
  \logp x = \log\max\set{1,x}\quad \text{and}\quad
  \logpt x = \log_2\max\set{1,x}.
\end{displaymath}
The notation $\M(N)$ refers to the cost of multiplying two $N$-bit integers. As
above, the notation $\Otilde(N)$ refers to $O(N \log^k N)$ for some value
of~$k\geq 0$; in particular $\M(N) = \Otilde(N)$.

For further notation concerning Siegel modular forms, Hilbert modular forms,
and period matrices, we refer to~§\ref{subsec:siegel-invariants},
§\ref{subsec:hilbert-invariants} and~§\ref{subsec:fund} respectively.

\subsection{Acknowledgements}

This work originates from my 2021 PhD dissertation at the University of Bordeaux,
France. I warmly thank my former supervisors Damien Robert and Aurel Page for
their suggestions and advice. This paper has further benefited from discussions
with Raymond van Bommel, Shiva Chidambaram, Fabien Cléry, and Edgar
Costa. Finally, I thank the anonymous referee for their very helpful comments.

\section{Modular equations for abelian surfaces}
\label{sec:modeq}

In this section, we recall the analytic formulas defining Siegel and
Hilbert modular equations for abelian surfaces. We also study their
denominators using the structures of the corresponding rings of modular
forms over~$\Z$.

\subsection{Invariants on the Siegel moduli space}
\label{subsec:siegel-invariants}

Let~$\Half_2$ be the set of symmetric $2\times 2$ complex
matrices~$\tau$ such that~$\im(\tau)$ is positive definite. Our notation
for the action of the symplectic group~$\GSp_4(\Q)$ on~$\Half_2$ is
as follows: for all $\gamma\in\GSp_4(\Q)$ and $\tau\in\Half_2$,
we write
\begin{displaymath}
  \gamma\tau = (a\tau+b)(c\tau+d)^{-1} \quad\text{and}\quad
  \gamma^*\tau = c\tau+d, \quad \text{where } \gamma=\mat{a}{b}{c}{d}.
\end{displaymath}
We also write $\Sp_4(\Z) = \Gamma(1)$. The
quotient~$\Gamma(1)\backslash\Half_2$ is a coarse moduli space for p.p.~abelian
surfaces over~$\C$~\cite[Thm.~8.2.6]{birkenhakeComplexAbelianVarieties2004}: a
given point~$\tau\in \Half_2$ corresponds to the abelian surface
\begin{displaymath}
  A(\tau) = \C^2/ (\tau\Z^2+\Z^2).
\end{displaymath}
We say that~$\tau$ is a period matrix of $A(\tau)$. Moreover,
$\Gamma(1)\backslash \Half_2$ is the set of complex points of an algebraic
variety defined over~$\Q$.  For a subring $R\subset \C$, we denote by
$\MF(\Gamma(1),R)$ the graded~$R$-algebra of Siegel modular forms of even
weight defined over~$R$, i.e.~whose Fourier coefficients belong
to~$R$~\cite[§4]{vandergeerSiegelModularForms2008}.

The~$\C$-algebra $\MF(\Gamma(1),\C)$ is free over four generators $\mf{h}_4$, $\mf{h}_6$,
$\mf{h}_{10}$, $\mf{h}_{12}$, whose weights appear as
subscripts~\cite{igusaSiegelModularForms1962}.  These modular forms can be
defined in terms of genus~$2$ theta constants (of level $(2,2)$), which we now
introduce. Let $a,b\in \{0,1\}^2$. Then the theta
constant of characteristic~$(a,b)$ is the holomorphic function on~$\Half_2$
defined by
\begin{equation}
  \label{eq:theta}
  \mf{\theta}_{a,b}(\tau) = \sum_{n\in\Z^2} \exp\paren[\Big]{\pi i \left(n+\tfrac a2\right)^t
      \tau\left(n+\tfrac a2\right) + \pi i \left(n+\tfrac a2\right)^t \tfrac b2},
\end{equation}
where $a,b$ are seen as vertical vectors. It is known that $\mf{\theta}_{a,b}$
is not indentically zero exactly when the dot product~$a^t b$ is even. When $a$
or $b$ assume specific values, we omit braces in the notation and write for
instance $\mf{\theta}_{00,10}$. We further abbreviate $\mf{\theta}_{00,00}$
as~$\mf{\theta}_0$.

The expressions of~$\mf{h}_4,\ldots,\mf{h}_{12}$ as polynomials in theta constants can be
found in~\cite[eq.~(7.1)]{strengComputingIgusaClass2014}. In
particular,~$\mf{h}_{10}$ is the product of all squares of even theta constants, and
is therefore a scalar multiple of the traditional cusp
form~$\mf{\chi}_{10}$. Therefore~$\tau\in \Half_2$ is the period matrix of a
genus~$2$ hyperelliptic curve over~$\C$ if and only if~$\mf{h}_{10}(\tau)\neq 0$.

Following \cite[§2.1]{strengComputingIgusaClass2014}, we
define the Igusa invariants as follows:
\begin{equation}
  \label{eq:igusa}
  \mf{j}_1 = \dfrac{\mf{h}_4\mf{h}_6}{\mf{h}_{10}},
  \quad \mf{j}_2 = \dfrac{\mf{h}_4^2\mf{h}_{12}}{\mf{h}_{10}^2},
  \quad \mf{j}_3 = \dfrac{\mf{h}_4^5}{\mf{h}_{10}^2}.
\end{equation}
They realize a birational map between $\Gamma(1)\backslash\Half_2$ and the
projective space~$\Proj^3$. This map is defined over~$\Q$ by the~$q$-expansion
principle~\cite[§V.1.5]{faltingsDegenerationAbelianVarieties1990}. Therefore,
the Igusa invariants~\eqref{eq:igusa} are suitable coordinates in which modular
equations for p.p.~abelian surfaces can be written. Note however that they are
not defined everywhere, and that the matrix
\begin{equation}
  \label{eq:dJ}
  \mf{dJ}(\tau) := \paren[\bigg]{\frac{\partial \mf{j}_k}{\partial \tau_n}}_{1\leq k,n\leq 3} \quad
  \text{where } \tau = \mat{\tau_1}{\tau_3}{\tau_3}{\tau_2}
\end{equation}
is not always invertible, for instance at points in~$\Half_2$ that have a
nontrivial stabilizer in~$\Gamma(1)$. In this paper, we always include in
the genericity conditions (e.g.~in \cref{thm:intro}) the invertibility
of~$\mf{dJ}(\tau)$.\footnote{Note that $\mf{dJ}$ can be viewed as a
  matrix-valued modular function: see~§\ref{subsec:differentials}.}

\begin{rem}
  \label{rem:16-thetas}
  To avoid this perhaps annoying requirement, one possibility would be to use
  theta constants instead, as the map
  $\tau\mapsto \paren[\big]{\mf{\theta}_{a,b}^2(\tau)}_{a,b\in \{0,1\}^2}$
  from~$\Half_2$ to~$\Proj^{15}(\C)$ is locally biholomorphic~\cite[Thm.~4
  p.\,189]{igusaThetaFunctions1972}. We stress however that working out this
  approach in detail would require substantially more work: see
  \cref{rem:dtheta-as-dj}.
\end{rem}

We will also use the structure
of~$\MF(\Gamma(1),\Z)$. Igusa~\cite{igusaRingModularForms1979} computed
a set of fourteen generators of this ring. For our purposes, it is sufficient to note
that $\Z[\mf{h}_4,\mf{h}_6,\mf{h}_{10},\mf{h}_{12}]\subset\MF(\Gamma(1),\Z)$ is a reasonably large
subring.

\begin{lem}
  \label{lem:siegelmfgens} Let $\mf{f}\in\MF(\Gamma(1),\Z)$ be a modular
  form of even weight~$w \in \Z_{\geq 0}$. Then
  \begin{displaymath}
    2^{\floor{7w/4}} 3^{\floor{w/4}}\mf{f}\in \Z[\mf{h}_4,\mf{h}_6,\mf{h}_{10},\mf{h}_{12}].
  \end{displaymath}
\end{lem}

\begin{proof}
  The lowest weight generators of~$\MF(\Gamma(1),\Z)$ are
  \begin{displaymath}
    \mf{X}_4 = 2^{-2}\mf{h}_4,\quad \mf{X}_6 = 2^{-2}\mf{h}_6,\quad
    \mf{X}_{10} = -2^{-12}\mf{h}_{10},\quad \mf{X}_{12} = 2^{-15}\mf{h}_{12}.
  \end{displaymath}
  The result holds for these four generators. Direct computations
  using the formulas from \cite[p.~153]{igusaRingModularForms1979}
  show that it also holds for the ten others.
\end{proof}

\begin{prop}
  \label{prop:mf-poly}
  Let~$\mf{f}\in \MF(\Gamma(1),\Z)$ be a modular form of
  even weight~$w \in \Z_{\geq 0}$. Let~$m\geq 0$ and~$0\leq n\leq 4$ be integers such that
  $4\floor{w/6}+w = 10m+4n$. Then there exists a unique
  polynomial~$Q_{\mf{f}}\in \Z[J_1,J_2,J_3]$ such that the following equality
  of Siegel modular functions holds:
  \begin{displaymath}
    Q_{\mf{f}}(\mf{j}_1,\mf{j}_2,\mf{j}_3)
    = 2^{\floor{7w/4}} 3^{\floor{w/4}} \dfrac{\mf{h}_4^{\floor{w/6}}}{\mf{h}_{10}^m\mf{h}_4^n} \mf{f}.
  \end{displaymath}
  The degree of~$Q_{\mf{f}}$ in~$J_1,J_2,J_3$ is bounded above
  by~$\floor{w/6}, \floor{w/12}, \floor{w/12}$ respectively,
  and the total degree of~$Q_{\mf{f}}$ is bounded above by~$\floor{w/6}$.
\end{prop}

\begin{proof}
  By \cref{lem:siegelmfgens}, we can rewrite the right hand side as
  \begin{displaymath}
    \dfrac{\mf{h}_4^{\floor{w/6}}}{\mf{h}_{10}^m\mf{h}_4^n} \mf{g}
  \end{displaymath}
  for some~$\mf{g}\in \Z[\mf{h}_4,\mf{h}_6,\mf{h}_{10},\mf{h}_{12}]$. Then the equalities
  \begin{displaymath}
    \mf{h}_4\mf{h}_6 = \mf{j}_1 \mf{h}_{10}, \quad \mf{h}_4^2 \mf{h}_{12} = \mf{j}_2 \mf{h}_{10}^2,
    \quad \mf{h}_4^5 = \mf{j}_3 \mf{h}_{10}^2
  \end{displaymath}
  show that this quotient can be rewritten as a polynomial
  in~$\mf{j}_1,\mf{j}_2,\mf{j}_3$ with integer coefficients satisfying the
  required degree bounds, as detailed
  in~\cite[Lem.~4.7]{kiefferDegreeHeightEstimates2022}. The resulting
  polynomial~$Q_{\mf{f}}$ is unique because~$\mf{j}_1,\mf{j}_2,\mf{j}_3$ are
  algebraically independent.
\end{proof}

\subsection{Siegel modular equations}
\label{subsec:siegel-modeq}

Let $\ell\in\Z$ be a prime.  Given~$\tau\in \Half_2$, the $\ell$-isogenies with
domain~$A(\tau)$ can be described as follows.

\begin{lem}
  \label{lem:siegel-representatives}
  For every~$\tau\in \Half_2$, the p.p.~abelian surfaces that are
  $\ell$-isogenous to $A(\tau)$ are the surfaces $A(\gamma\tau)$,
  where~$\gamma$ runs through the $\ell^3+\ell^2+\ell+1$ elements of the coset space
  \begin{displaymath}
    \Gamma(1)\,\backslash\, \Gamma(1) \mat{I_2}{0}{0}{\ell I_2} \Gamma(1).
  \end{displaymath}
\end{lem}

\begin{proof}
  Define the subgroup~$\Gamma^0(\ell)$ of
  $\Gamma(1) = \Sp_4(\Z)$ as follows:
  \begin{displaymath}
    \Gamma^0(\ell) = \left\{\mat{a}{b}{c}{d}\in\Gamma(1)\st b=0\mod\ell\right\}.
  \end{displaymath}
  We have $[\Gamma(1):\Gamma^0(\ell)] = \ell^3+\ell^2+\ell+1$.  For
  each~$\tau\in \Half_2$, the p.p.~abelian surfaces that are $\ell$-isogenous to
  $A(\tau)$ are exactly the abelian surfaces $A(\tfrac1\ell \gamma\tau)$
  where~$\gamma$ runs through the coset space
  $\Gamma^0(\ell)\backslash \Gamma(1)$; the proof of this fact is analogous
  to~\cite[Thm.~3.2]{brokerModularPolynomialsGenus2009}.  Using the action
  of~$\GSp_4(\Q)$ on~$\Half_2$, one can write for every~$\gamma\in \Gamma(1)$:
  \begin{displaymath}
    \tfrac1\ell \gamma\tau = \paren[\bigg]{\mat{I_2}{0}{0}{\ell I_2} \gamma}  \tau. \smallskip
  \end{displaymath}
  Further, we can check that
  $\gamma\mapsto \smash{\mat{I_2}{0}{0}{\ell I_2}} \gamma$ is a bijection from
  $\Gamma^0(\ell)\backslash\Gamma(1)$ to the above coset space.
\end{proof}

An important fact is that one can construct a set $C_\ell$ of representatives
of the coset space in \cref{lem:siegel-representatives} whose lower left
$2\times 2$ blocks are all zero, or in other words, which leave the cusp at
infinity invariant. This will matter in the numerical computations later on, as
it implies that~$\gamma\tau$ is not too far from the fundamental domain
in~$\Half_2$ when~$\tau$ satisfies the same property. Such a set is presented
in~\cite[Prop.~10.1]{cleryConstructingVectorvaluedSiegel2015}. Explicitly,
$C_\ell$ consists of matrices of the following form, where $x, y, z$ run
through $\{0,\ldots,\ell-1\}$:
\begin{displaymath}
  \begin{pmatrix} 1 & 0 & x & y \\ 0 & 1 & y & z \\ 0 & 0 & \ell & 0\\ 0 & 0 & 0 & \ell
  \end{pmatrix},
  \begin{pmatrix}
    \ell & 0 & 0 & 0 \\ -x & 1 & 0 & y \\ 0 & 0 & 1 & x \\ 0 & 0 & 0 & \ell
  \end{pmatrix},
  \begin{pmatrix}
    0 & -\ell & 0 & 0 \\ 1 & 0 & x & 0 \\ 0 & 0 & 0 & -1 \\ 0 & 0 & \ell & 0
  \end{pmatrix}
  \text{ and }
  \begin{pmatrix}
    \ell & 0 & 0 & 0\\ 0 & \ell & 0 & 0\\ 0 & 0 & 1 & 0 \\ 0 & 0 & 0 & 1
  \end{pmatrix}.
\end{displaymath}

The Siegel modular equations of
level~$\ell$~\cite{brokerModularPolynomialsGenus2009,milioQuasilinearTimeAlgorithm2015}
are the three multivariate rational fractions
$\Psi_{\ell,k} \in \Q(J_1,J_2,J_3)[X]$ for $1\leq k\leq 3$ such that for every
$\tau\in\Half_2$ where everything is well defined, we have
\begin{equation} \label{eq:siegelmodeq}
  \begin{aligned}
    \Psi_{\ell,1}\paren[\big]{\mf{j}_1(\tau),\mf{j}_2(\tau),\mf{j}_3(\tau), X}
    &= \prod_{\gamma\in C_\ell} \bigl(X - \mf{j}_1(\gamma\tau)\bigr),\\
    \Psi_{\ell,2}\paren[\big]{\mf{j}_1(\tau),\mf{j}_2(\tau),\mf{j}_3(\tau), X}
    &= \sum_{\gamma\in C_\ell} \mf{j}_2(\gamma\tau)
    \prod_{\gamma'\in C_\ell\backslash\{\gamma\}} \bigl(X - \mf{j}_1(\gamma'\tau)\bigr),\\
    \Psi_{\ell,3}\paren[\big]{\mf{j}_1(\tau),\mf{j}_2(\tau),\mf{j}_3(\tau), X}
    &= \sum_{\gamma\in C_\ell} \mf{j}_3(\gamma\tau)
    \prod_{\gamma'\in C_\ell\backslash\{\gamma\}} \bigl(X - \mf{j}_1(\gamma'\tau)\bigr).
  \end{aligned}
\end{equation}
The degrees of the polynomials $\Psi_{\ell,k}$ in $X$ are at most
$\ell^3+\ell^2+\ell+1$, and their total degrees in $J_1,J_2,J_3$ are
at most $10(\ell^3+\ell^2+\ell+1)/3$ by
\cite[Prop.~4.10]{kiefferDegreeHeightEstimates2022}. The height of
their coefficients in~$\Q$ is~$O(\ell^3\log\ell)$ by
\cite[Thm.~1.1]{kiefferDegreeHeightEstimates2022}.

\subsection{Denominators of Siegel modular equations}
\label{subsec:siegel-denom}

We call $Q_\ell\in \Z[J_1,J_2,J_3]$ a \emph{denominator} of the Siegel modular
equations $\Psi_{\ell,k}$ if for each $1\leq k\leq 3$, we have
\begin{displaymath}
  Q_\ell\Psi_{\ell,k} \in \Z[J_1,J_2,J_3,X].
\end{displaymath}
Our goal is to describe such a denominator~$Q_\ell$ analytically.  For every
$\tau\in\Half_2$, we set
\begin{equation}
  \label{eq:gl}
  \mf{g}_\ell(\tau) = 2^{6}\prod_{\gamma\in C_\ell} 2^{-16} \mf{h}_{10}^2(\gamma\tau).
\end{equation}
One may interpret $\mf{g}_\ell$ as the image of $\smash{\mf{h}_{10}^2}$ under a certain Hecke
operator of level~$\ell$, although these operators are usually defined as sums
rather than products.  By~\eqref{eq:igusa} and~\eqref{eq:siegelmodeq}, the
polynomials
$\mf{g}_\ell(\tau) \Psi_{\ell,k}\paren[\big]{j_1(\tau),j_2(\tau),j_3(\tau)}$ for
$1\leq k\leq 3$ are then equal to, respectively,
\begin{equation}
  \label{eq:siegelnum}
  \begin{aligned}
    2^6 &\prod_{\gamma\in C_\ell} 2^{-16}\paren[\big]{\mf{h}_{10}^2(\gamma\tau) X - \mf{h}_4 \mf{h}_6 \mf{h}_{10}(\gamma\tau)},\\
    &\sum_{\gamma\in C_\ell} 2^{-10} \mf{h}_4^2 \mf{h}_{12}(\gamma\tau) \prod_{\gamma'\in C_\ell\setminus\{\gamma\}}
      2^{-16} \paren[\big]{\mf{h}_{10}(\gamma'\tau)^2 X - \mf{h}_4\mf{h}_6\mf{h}_{10}(\gamma'\tau)},\\
    \text{and}\quad &\sum_{\gamma\in C_\ell} 2^{-10} \mf{h}_4^5(\gamma\tau) \prod_{\gamma'\in C_\ell\setminus\{\gamma\}}
      2^{-16} \paren[\big]{\mf{h}_{10}(\gamma'\tau)^2 X - \mf{h}_4\mf{h}_6\mf{h}_{10}(\gamma'\tau)}.\\
  \end{aligned}
\end{equation}
For every $0\leq m\leq \ell^3+\ell^2+\ell+1$ and~$1\leq k\leq 3$, we also let
$\mf{f}_{\ell,k,m}(\tau)$ be the coefficient of~$\smash{X^m}$ in the polynomial
$\mf{g}_\ell(\tau) \Psi_{\ell,k}\paren[\big]{\mf{j}_1(\tau),\mf{j}_2(\tau),\mf{j}_3(\tau)}$. We further write
\begin{displaymath}
  w_\ell = 20(\ell^3 + \ell^2 + \ell + 1).
\end{displaymath}

\begin{prop}
  \label{prop:gl-ZZ} The functions $\mf{g}_\ell$ and $\mf{f}_{\ell,k,m}$ for
  $1\leq k\leq 3$ and $0\leq m\leq \ell^3+\ell^2+\ell+1$ are Siegel modular
  forms of weight~$w_\ell$ defined over~$\Z$.
\end{prop}

\begin{proof}
  First, we prove that $\mf{g}_\ell$ and $\mf{f}_{\ell,k,m}$ are Siegel modular forms of
  the correct weights. The quotient $\mf{f}_{\ell,k,m}/\mf{g}_{\ell}$, being a
  coefficient of the Siegel modular equations, is a modular function of weight
  zero, and both $\mf{f}_{\ell,k,m}$ and~$\mf{g}_\ell$ are holomorphic on~$\Half_2$
  by~\eqref{eq:gl} and~\eqref{eq:siegelnum}. Thus it is sufficient to show
  that~$\mf{g}_\ell$ satisfies the required transformation rule.

  Let~$\eta\in \Gamma(1)$. For each~$\gamma\in C_\ell$, there
  exists~$\eta_\gamma\in \Gamma(1)$ and~$\gamma'\in C_\ell$ such that
  $\gamma\eta = \eta_\gamma \gamma'$. Further, $\gamma\mapsto \gamma'$ is a
  bijection from~$C_\ell$ to $C_\ell$. We then have
  \begin{displaymath}
    \mf{g}_\ell(\eta \tau) = 2^{6} \prod_{\gamma\in C_\ell} 2^{-16}
    \det(\eta_\gamma^*(\gamma'\tau))^{20} \mf{h}_{10}^2(\gamma'\tau) = \mf{g}_\ell(\tau) \prod_{\gamma\in C_\ell} \det(\eta_\gamma^*(\gamma'\tau))^{20}.
  \end{displaymath}
  Moreover, by the cocycle formula, we have for each $\gamma\in C_\ell$:
  \begin{align*}
    \det(\eta_\gamma^*(\gamma'\tau)) &= \det((\eta_\gamma\gamma')^*\tau) \det(\gamma'^*\tau)^{-1}\\
                                     &= \det(\gamma^*(\eta\tau)) \det(\eta^*\tau) \det(\gamma'^*\tau)^{-1}
  \end{align*}
  In this last expression, $\det(\gamma^*(\eta\tau))$ and $\det(\gamma'^*\tau)$
  are independent of~$\tau$: they are the determinants of the lower right
  $2\times 2$ blocks of~$\gamma$ and~$\gamma'$ respectively. Therefore,
  \begin{displaymath}
    \prod_{\gamma\in C_\ell} \det(\eta_\gamma^*(\gamma'\tau)) = \det(\eta^*\tau)^{\ell^3 + \ell^2 + \ell + 1},
  \end{displaymath}
  so~$\mf{g}_\ell$ is indeed a Siegel modular form of the correct weight.

  Second, we show that the Fourier coefficients of $\mf{g}_{\ell}$ and
  $\mf{f}_{\ell,k,m}$ are algebraic integers. By~\eqref{eq:gl}
  and~\eqref{eq:siegelnum}, both $\mf{g}_\ell$ and $\mf{f}_{\ell,k,i}$ are
  polynomials in holomorphic functions of the form $\mf{h}(\gamma\tau)$,
  where~$\gamma\in C_\ell$ and~$\mf{h}$ is a Siegel modular form defined
  over~$\Z$. More precisely,
  $\mf{h}\in \{2^{-16} \mf{h}_{10}^2, 2^{-16} \mf{h}_4\mf{h}_6\mf{h}_{10},
  2^{-10} \mf{h}_4^2 \mf{h}_{12}, 2^{-10} \mf{h}_4^5\}$. Such a modular
  form~$\mf{h}$ admits a Fourier expansion: letting~$\tau_1,\tau_2,\tau_3$ be
  as in~\eqref{eq:dJ}, we can write
  \begin{displaymath}
    \mf{h}(\tau) = \sum_{n = (n_1,n_2,n_3)\in \Z^3} h_n e^{2 \pi i (n_1 \tau_1 + n_2 \tau_2 + n_3\tau_3)}
  \end{displaymath}
  where~$h_n \in \Z$ for every~$n$, and $h_n = 0$ unless $n_1\geq 0$,
  $n_2\geq 0$ and $n_3^2 \leq 4n_1n_2$.

  Let~$a,b,0,d$ be the $2\times 2$ blocks of~$\gamma$. Then
  $\gamma\tau = (a\tau + b) d^{-1} = \tfrac 1\ell (a \tau a^t + b a^t)$, so the
  entries of $\gamma\tau$ are linear combinations of $\tau_1/\ell$,
  $\tau_2/\ell$, $\tau_3/\ell$, and $1/\ell$ with integral
  coefficients. Substituting these expressions in the Fourier expansion
  of~$\mf{h}$, we see that $\mf{h}(\gamma\tau)$ can be expanded as a power
  series in the ring
  \begin{displaymath}
    \Z \bigl[e^{2\pi i/\ell}, e^{2\pi i \tau_3/\ell}, e^{-2\pi i \tau_3/\ell}\bigr] \bigl[\!\bigl[ e^{2\pi i \tau_1/\ell}, e^{2\pi i \tau_2/\ell} \bigr]\!\bigr].
  \end{displaymath}
  Consequently, the Fourier expansions of $\mf{g}_\ell$ and $\mf{f}_{\ell,k,i}$ also
  belong to this ring, so the Fourier coefficients of $\mf{g}_\ell$ and
  $\mf{f}_{\ell,k,i}$ belong to~$\Z[e^{2\pi i/\ell}]$. In particular, they are
  algebraic integers.

  Finally, we show that the Fourier coefficients of~$\mf{g}_\ell$ are rational. This
  will conclude the proof as $\mf{f}_{\ell,k,i}/\mf{g}_\ell$, being a coefficient of the
  modular equations, is defined over~$\Q$.  Let~$\mf{h}$ be an integral modular form
  as above, and let~$(\Z/\ell\Z)^\times$ act as the Galois group of
  $\Q[e^{2\pi i/\ell}]$, i.e.~each $\lambda\in (\Z/\ell\Z)^\times$ acts as
  $e^{2\pi i/\ell}\mapsto e^{2\pi i \lambda/\ell}$. We extend this action to the
  above power series ring by letting~$(\Z/\ell\Z)^\times$ act on each
  coefficient. Further, we define an action of $(\Z/\ell\Z)^\times$ on~$C_\ell$
  as follows:
  \begin{displaymath}
    \lambda \cdot \gamma = \mat{a}{\lambda b \mod \ell}{0}{d} \in C_\ell
    \quad \text{if } \lambda\in (\Z/\ell\Z)^\times
     \text{ and }  \gamma = \mat{a}{b}{0}{d}.
  \end{displaymath}
  Then we observe that the map $\gamma\mapsto \mf{h}(\gamma\tau)$ is equivariant for
  these actions of $(\Z/\ell\Z)^\times$. As a consequence, the Fourier
  expansion of $\mf{g}_\ell$ is Galois-invariant, so is defined over~$\Q$.
\end{proof}

Let~$Q_{\ell}$ be the polynomial defined in \cref{prop:mf-poly} applied to
$\mf{f} = \mf{g}_\ell$.

\begin{prop}
  \label{prop:Ql}
  The polynomial~$Q_{\ell}\in \Z[J_1,J_2,J_3]$ is a denominator of the
  Siegel modular equations of level~$\ell$.
\end{prop}

\begin{proof}
  For each $1\leq k\leq 3$ and $0\leq m\leq \ell^3+\ell^2+\ell+1$, the
  coefficient of~$X^m$ in the rational
  fraction~$Q_{\ell} \Psi_{\ell,k}\in \Q(J_1,J_2,J_3)[X]$ is the
  polynomial~$Q_{\mf{f}_{\ell,k,m}}$ from \cref{prop:mf-poly}.
\end{proof}

For~$1\leq k\leq 3$, we denote the
polynomial~$Q_\ell \Psi_{\ell,k} \in \Z[J_1,J_2,J_3,X]$ by~$P_{\ell,k}$.

\subsection{Invariants on Hilbert moduli spaces}
\label{subsec:hilbert-invariants}

Fix a real quadratic field~$F$ of discriminant~$\Delta_F$, and let~$\Z_F$ be
its ring of integers. We fix a real embedding of~$F$, and denote the other
embedding by $x\mapsto \conj{x}$.  Then~$\GL_2(F)$ embeds in~$\GL_2(\R)^2$ via
the two real embeddings of~$F$. Via this embedding, $\GL_2(F)$ acts
on~$\Half_1^2$, where~$\Half_1\subset \C$ denotes the usual upper half
plane. We denote this action by~$(\gamma,t)\mapsto \gamma t$. For all
$\gamma\in \GL_2(F)$, all $\lambda\in F$ and every $t = (t_1,t_2)\in\Half_1^2$,
we also write
\begin{displaymath}
  \lambda t = (\lambda t_1, \conj{\lambda} t_2)
  \quad\text{and}\quad
  \gamma^*t = (ct_1+d)(\conj{c}t_2+\conj{d}) \quad\text{if } \gamma=\mat{a}{b}{c}{d}.
\end{displaymath}

Let~$\Z_F^\dual = (1/\sqrt{\Delta_F})\,\Z_F$ be the dual of~$\Z_F$ for the
trace form, and define the group
\begin{displaymath}
  \Gamma_F(1)
  = \left\{\mat{a}{b}{c}{d}\in\SL_2(F)\st
    a,d\in \Z_F,\,b\in \Z_F^\dual,\,c\in (\Z_F^\dual)^{-1}\right\}.
\end{displaymath}
The quotient $\Gamma_F(1)\backslash\Half_1^2$ is a coarse moduli space for
p.p.~abelian surfaces over~$\C$ with RM
by~$\Z_F$~\cite[§9.2]{birkenhakeComplexAbelianVarieties2004}. We denote the
p.p.~abelian surface with RM by~$\Z_F$ attached to $t\in \Half_1^2$ by~$A_F(t)$. The
involution given by
\begin{displaymath}
  \sigma_F\from (t_1,t_2)\mapsto(t_2,t_1)
\end{displaymath}
exchanges the RM embedding with its Galois conjugate, so we find an action of
the semidirect product $\Gamma_F(1)\rtimes\langle\sigma_F\rangle$
on~$\Half_1^2$. Both~$\Gamma_F(1)\backslash\Half_1^2$ and its quotient
by~$\sigma_F$ have canonical algebraic models that are defined
over~$\Q$~\cite[Chap.~X, Rem.~4.1]{vandergeerHilbertModularSurfaces1988}.

For a subring~$R\subset\C$, we denote by~$\MF(\Gamma_F(1), R)$ the
graded $R$-algebra of Hilbert modular forms of even weight defined
over~$R$ that are symmetric, i.e.~invariant under~$\sigma_F$. There is
no known general description of these graded algebras, although
$\MF(\Gamma_F(1),\Z)$ is known for discriminants~$5$
and~$8$~\cite{nagaokaRingIlbertModular1983}, $\MF(\Gamma_F(1),\Q)$
is known for the additional discriminants~$12, 13, 17, 24, 29, 37$
(see~\cite{williamsRingsIlbertModular2020} and the references
therein), and in general they are amenable to computation for a
fixed~$F$~\cite{dembeleExplicitMethodsHilbert2013}.

One way of defining coordinates on Hilbert moduli spaces consists in
pulling back invariants from the Siegel moduli space by the forgetful
map. Let $(e_1,e_2)$ be the $\Z$-basis of~$\Z_F$; for the sake of
fixing definitions, we take $(e_1,e_2) = (1,\tfrac12\sqrt{\Delta_F})$
if~$\Delta_F$ is even, and~$(1,\tfrac12 + \tfrac12\sqrt{\Delta_F})$
if~$\Delta_F$ is odd. Write
\begin{displaymath}
  M_F = \mat{e_1}{e_2}{\wconj{e}_1}{\wconj{e}_2}\in \GL_2(\R).
\end{displaymath}
Then the \emph{Hilbert
  embedding}~\cite[p.\,209]{vandergeerHilbertModularSurfaces1988} is
the map
\begin{equation}
  \label{eq:Hilb}
  \begin{array}{cccc}
    \mathrm{Hilb}_F \from & \Half_1^2& \to &\Half_2 \\[2mm]
    & (t_1,t_2) & \mapsto& M_F^t \mat{t_1}{0}{0}{t_2} M_F.
  \end{array}
\end{equation}
It is equivariant for the actions of~$\GL_2(F)$ and~$\GSp_4(\Q)$ under
the following map, also denoted by~$\mathrm{Hilb}_F$:
\begin{equation}
  \label{eq:Hilb-mat}
  \mat{a}{b}{c}{d}\mapsto \mat{M_F^t}{0}{0}{M_F^{-1}}
  \mat{a^*}{b^*}{c^*}{d^*} \mat{M_F^{-t}}{0}{0}{M_F}
\end{equation}
where $x^* = \tmat{x}{0}{0}{\conj{x}}$ for~$x\in F$. Further, the action
of~$\sigma_F$ on~$\Half_1^2$ corresponds via~$\mathrm{Hilb}_F$ to the action of
a certain involution~$\mathrm{Hilb}_F(\sigma_F)\in \Sp_4(\Z)$.

One can check that~$\mathrm{Hilb}_F$ induces a map
$\conj{\mathrm{Hilb}}_F\from (\Gamma_F(1)\rtimes\sigma_F)\backslash \Half_1^2\to
\Gamma(1)\backslash\Half_2$; in the modular interpretation, $\conj{\mathrm{Hilb}}_F$
forgets the RM structure. Generically, the map~$\conj{\mathrm{Hilb}}_F$ is
injective. Therefore, the pullbacks of the Igusa invariants~$\mf{j}_1,\mf{j}_2,\mf{j}_3$
under~$\conj{\mathrm{Hilb}}_F$ can be used to write down modular equations of Hilbert
type symmetrized under~$\sigma_F$.

When the structure of $\MF(\Gamma_F(1),\Z)$ is known, it is better to design
other coordinates in terms of a generating set of modular forms. When
$F = \Q(\sqrt{5})$, the ring $\MF(\Gamma_F(1),\Z)$ is generated by four
elements $\mf{G}_2,\mf{F}_6,\mf{F}_{10}$, and~$\mf{F}_{12}= \tfrac{1}{4}(\mf{F}_6^2 -\mf{G}_2\mf{F}_{10})$, whose
weights appear as subscripts~\cite{nagaokaRingIlbertModular1983}.
Following~\cite{milioModularPolynomialsIlbert2020}, we define the Gundlach
invariants $\mf{g}_1, \mf{g}_2$ as
\begin{equation}
  \label{eq:gundlach}
  \mf{g}_1 = \dfrac{\mf{G}_2^5}{\mf{F}_{10}} \quad\text{and}\quad
  \mf{g}_2 = \dfrac{\mf{G}_2^2\mf{F}_6}{\mf{F}_{10}}.
\end{equation}
The expressions of the pullbacks of~$\mf{h}_4,\ldots,\mf{h}_{12}$ by the Hilbert
embedding~$\mathrm{Hilb}_F$ in terms of~$\mf{G}_2,\mf{F}_6,\mf{F}_{10}$ can be found
in~\cite{resnikoffGradedRingIlbert1974}. In particular,
$\mathrm{Hilb}_F^*(\mf{h}_{10}) = 2^{12} \mf{F}_{10}$. As a byproduct, Gundlach and Igusa
invariants are related by explicit polynomial
formulas~\cite[Thm.~1]{resnikoffGradedRingIlbert1974},~\cite[Prop.~4.5]{lauterComputingGenusCurves2011}.
From the description of~$\MF(\Gamma_F(1),\Z)$, we immediately obtain the
following analogue of \cref{prop:mf-poly}.

\begin{prop}
  \label{prop:mf-poly-hilbert}
  Let $F = \Q(\sqrt{5})$. For every $\mf{f}\in\MF(\Gamma_F(1),\Z)$ of even
  weight~$w\in \Z_{\geq 0}$, we have
  $2^{\floor{w/3}}\mf{f}\in\Z[\mf{G}_2,\mf{G}_6,\mf{F}_{10}]$. Further
  let~$m\geq 0$ and~$0\leq n\leq 4$ be the unique integers such that
  $4\floor{w/6} + w = 10m+2n$. Then there is a unique
  polynomial~$Q_{\mf{f}}\in \Z[G_1,G_2]$ (where $G_1,G_2$ are variables) such
  that the following equality of Hilbert modular functions holds:
  \begin{displaymath}
    Q_{\mf{f}}(\mf{g}_1,\mf{g}_2) = 2^{\floor{w/3}} \dfrac{\mf{G}_2^{2\floor{w/6}}}{\mf{F}_{10}^m \mf{G}_2^n} \mf{f}.
  \end{displaymath}
  The total degree of~$Q_{\mf{f}}$ is bounded above by~$\floor{w/6}$.
\end{prop}

If~$F = \Q(\sqrt{8})$, then~$\MF(\Gamma_F(1),\Z)$ is generated by three
algebraically independent modular forms
$\mf{G}_2, \mf{F}_4, \mf{F}_6$~\cite{nagaokaRingIlbertModular1983}. In this case, we the
Gundlach invariants are $\mf{g}_1 = \mf{G}_2^2/\mf{F}_4$ and $\mf{g}_2 = \mf{G}_2 \mf{F}_6/\mf{F}_4^2$. If
$\mf{f}\in \MF(\Gamma_F(1),\Z)$ is of weight~$w$, and $m\geq 0$ and $0\leq n\leq 1$
are the unique integers such that $\floor{w/6} = 2m + n$, there is a unique
polynomial $Q_{\mf{f}}\in \Z[J_1,J_2]$ such that
\begin{displaymath}
  Q_{\mf{f}}(\mf{g}_1,\mf{g}_2) = \frac{\mf{G}_2^{\floor{w/6}}}{\mf{F}_4^m\mf{G}_2^n} \mf{f}.
\end{displaymath}
Thus, the case $F = \Q(\sqrt{8})$ is similar to, or even simpler than,
$F = \Q(\sqrt{5})$. We focus on the latter case in the rest of the paper for brevity.

\subsection{Hilbert modular equations}
\label{subsec:hilbert-modeq}

We fix~$F = \Q(\sqrt{5})$ in the rest of this section, and consider modular
equations in Gundlach invariants.  Let~$\ell$ be a prime not
dividing~$\Delta_F$ that splits in~$F$ in two principal ideals generated by
totally positive elements $\beta, \conj{\beta}\in\Z_F$. We have the following
analogue of \cref{lem:siegel-representatives}; we omit its proof.

\begin{lem}
  \label{lem:hilbert-representatives}
  For every~$t\in \Half_1^2$, the p.p.~abelian surfaces that are
  $\beta$-isogenous to $A_F(t)$ are the surfaces $A_F(\gamma t)$,
  where~$\gamma$ runs through the $\ell+1$ elements of the coset space
  \begin{displaymath}
    \Gamma_F(1) \backslash \Gamma_F(1) \mat{1}{0}{0}{\beta}\Gamma_F(1).
  \end{displaymath}
\end{lem}

We can similarly write down a set of representatives whose lower left entry is
zero. It is the set $C_\beta$ consisting of the following matrices for
$x\in \{0,\ldots,\ell-1\}$:
\begin{displaymath}
  \mat{\beta}{0}{0}{1} \quad\text{and}\quad \mat{1}{x/\sqrt{\Delta_F}}{0}{\beta}.
\end{displaymath}
We let~$\smash{C_\beta^{\mathrm{sym}}}$ be the (formal) disjoint union of $C_\beta$ and
$C_\beta\,\sigma_F$. We can then make sense of $\gamma t$ for each
$\gamma\in C_\beta^{\mathrm{sym}}$ and $t\in \Half_1^2$: if $\gamma\in C_\beta\,\sigma_F$,
we act first by $\sigma_F$, then by the element of $C_\beta$.

The Hilbert modular equations of level~$\beta$ in Gundlach
invariants~\cite{martindaleHilbertModularPolynomials2020,
  milioModularPolynomialsIlbert2020} are the two multivariate rational
fractions in 3 variables $G_1,G_2,X$ denoted by
\begin{displaymath}
  \Psi_{\beta,k} \in \Q(G_1,G_2)[X] \quad \text{for } 1\leq k\leq 2
\end{displaymath}
such that for every $t\in \Half_1^2$ where everything is well defined, we have
\begin{equation}
  \label{eq:modeq-hilbert}
  \begin{aligned}
    \Psi_{\beta,1}\paren[\big]{\mf{g}_1(t),\mf{g}_2(t),X}
    &= \prod_{\gamma\in C_\beta^{\mathrm{sym}}} \bigl(X - \mf{g}_1(\gamma t)\bigr),\\
    \Psi_{\beta,2}\paren[\big]{\mf{g}_1(t),\mf{g}_2(t),X}
    &= \sum_{\gamma\in C_\beta^{\mathrm{sym}}} \mf{g}_2(\gamma t)
      \prod_{\gamma'\in C_\beta^{\mathrm{sym}}\backslash\{\gamma\}} \bigl(X - \mf{g}_1(\gamma' t)\bigr).
  \end{aligned}
\end{equation}
The formulas in the general case of Hilbert modular equations in Igusa
invariants are similar: we refer to~\cite[Lemma
3.9]{milioModularPolynomialsIlbert2020}.

The degrees of the polynomials~$\Psi_{\beta,k}$ in~$X$ are at most $2\ell+2$,
and their total degrees in $G_1,G_2$ are at most $10(\ell+1)/3$ by
\cite[Prop.~4.11]{kiefferDegreeHeightEstimates2022}. The height of their
coefficients is~$O(\ell\log\ell)$ by
\cite[Thm.~1.1]{kiefferDegreeHeightEstimates2022}.

\begin{rem}
  \label{rem:nonsymmetric}
  Considering non-symmetric coordinates on the Hilbert surface instead of
  Gundlach invariants would allow us to further divide the degrees and heights
  of modular equations by an approximate factor of~2 as we would only consider
  $C_\beta$, not~$C_\beta^{\mathrm{sym}}$. However, this would also introduce further
  technicalities because the relation with theta constants and Igusa invariants
  will not be as direct. We leave this improvement to future work.
\end{rem}

\subsection{Denominators of Hilbert modular equations}
\label{subsec:hilbert-denom}
As above, we call $Q_\beta\in \Z[J_1,J_2]$ a \emph{denominator} of the
Hilbert modular equations~$\Psi_{\beta,k}$ if for each
$1\leq k\leq 2$, we have
\begin{displaymath}
  Q_\beta\Psi_{\beta,k} \in \Z[J_1,J_2,X].
\end{displaymath}
For every $t\in \Half_1^2$, we write
\begin{equation}
  \label{eq:gbeta}
  \mf{g}_\beta(t) = \prod_{\gamma\in C_\beta^{\mathrm{sym}}} 2^{-12} \mf{h}_{10}\paren[\big]{\mathrm{Hilb}_F(\gamma t)},
\end{equation}
and let~$w_\beta = 20(\ell + 1)$.

\begin{prop}
  \label{prop:gbeta-ZZ}
  The holomorphic function $\mf{g}_\beta$ is a symmetric and integral Hilbert
  modular form of weight $w_\beta$ and level~$\Gamma_F(1)$
  where~$F = \Q(\sqrt{5})$. The polynomial $Q_{\beta}$ obtained from
  \cref{prop:mf-poly-hilbert} with~$\mf{f} = \mf{g}_\beta$ is a denominator of the
  Hilbert modular equations.
\end{prop}

\begin{proof}
  By~\cite[Thm.~1]{resnikoffGradedRingIlbert1974}, we have
  $2^{-12}\mf{h}_{10}\paren[\big]{\mathrm{Hilb}_F(\gamma t)} = -\mf{F}_{10}(t)$ for each
  $t\in \Half_1^2$, so the polynomials
  $\mf{g}_\beta(t)\Psi_{\beta,k}\paren[\big]{\mf{g}_1(t),\mf{g}_2(t), X}$ for $k\in \{1,2\}$ have
  holomorphic coefficients. One can prove that $\mf{g}_\beta$ is a Hilbert modular
  form for~$\Gamma_F(1)$ of the correct weight by applying the cocycle formula
  as in the Siegel case. Moreover, $\mf{g}_\beta$ is plainly invariant
  under~$\sigma_F$.

  We now show that $\mf{g}_\beta$, as well as any coefficient of
  $\mf{g}_\beta(t)\Psi_{\beta,k}\paren[\big]{\mf{g}_1(t),\mf{g}_2(t)}$ for $k\in \{1,2\}$, is
  defined over~$\Z$. As a Hilbert modular form of level~$\Gamma_F(1)$,
  $\mf{g}_\beta$ admits a Fourier expansion of the form
  \begin{displaymath}
    \mf{g}_\beta(t) = \sum_{n \in (\Z_F^\dual)^{-1}} a_n e^{2\pi i (n t_1 + \conj{n} t_2)},
  \end{displaymath}
  We want to show that the coefficients~$a_n$ are integral. This is
  sufficient because the Hilbert surface in question has only one cusp, namely
  the cusp at infinity~\cite[§I.1]{vandergeerHilbertModularSurfaces1988}.

  In order to show that the coefficients~$a_n$ are algebraic integers, we
  proceed as in \cref{prop:gl-ZZ}.  Consider a function of the form
  $\mf{h}(\gamma t)$, where~$\mf{h}$ is a Hilbert modular form defined
  over~$\Z$ and~$\gamma\in C_\beta$. Using the description of $C_\beta$ and the
  fact that $\beta\overline{\beta}=\ell$, we observe that
  $t\mapsto \mf{h}(\gamma t)$ admits a Fourier expansion of the form
  \begin{displaymath}
    \mf{h}(\gamma t) = \sum_{n\in (\Z_F^\dual)^{-1}} b_n e^{2\pi  i (n t_1 + \conj{n} t_2)/\ell}
  \end{displaymath}
  with~$b_n \in \Z[e^{2\pi i /\ell}]$ for all~$n$. This shows
  that~$a_n\in \Z[e^{2\pi i/\ell}]$ for all~$n$.

  Finally, we argue that the coefficients~$a_n$ are actually integers: if we
  consider the Galois action of~$(\Z/\ell\Z)^\times$ on the coefficients and
  let $\lambda\in (\Z/\ell\Z)^\times$ act on $C_\beta$ by
  \begin{displaymath}
    \lambda \mat{\beta}{0}{0}{1} = \mat{\beta}{0}{0}{1} \text{ and }
    \lambda \mat{1}{a/\sqrt{\Delta_F}}{0}{\beta} = \mat{1}{(\lambda a \text{ mod } \ell)/\sqrt{\Delta_F}}{0}{\beta},
  \end{displaymath}
  then $\gamma\mapsto \mf{h}(\gamma t)$ is equivariant, so the expansion of $\mf{g}_\beta$ is invariant.
\end{proof}

We denote the polynomial $Q_\beta \Psi_{\beta,k}$ by $P_{\beta,k}$ for
$1\leq k\leq 2$. In the case $F = \Q(\sqrt{8})$, we can replace $2^{-12}\mf{h}_{10}$
by $2^{-24} \mf{h}_{10}^2$ to define the denominator $\mf{g}_\beta$, as
\cite[Thm.~A.14]{milioModularPolynomialsIlbert2020} shows.

\section{Structure of the evaluation algorithms}
\label{sec:structure}

In this section, we describe the evaluation algorithms for modular equations of
Siegel and Hilbert type in detail. The description involves certain fundamental
algorithms on period matrices, for instance the AGM method, that are further
discussed in~§\ref{sec:agm}.

\subsection{Computational model}
\label{subsec:model}

All our numerical computations are performed in interval arithmetic. We now
review some key notions and results that we use throughout.

Given $z\in \C$ and~$N\in \Z_{\geq 0}$, we define an \emph{approximation of~$z$
  to precision~$N$} to be a complex ball centered at a dyadic point of radius
at most~$2^{-N}$ containing~$z$. An approximation of a polynomial (or a vector,
or a matrix) to precision~$N$ is by definition an approximation to
precision~$N$ of each of its coefficients. When we say that an operation is
performed \emph{at working precision~$N$}, this means that we truncate its
output to precision~$N$, even if we might be able to compute it to a higher
precision.

We say that the \emph{precision loss} in a given algorithm is (at most)~$B$
bits if the following property holds: if the input is given to
precision~$N\geq B$, then the algorithm succeeds and its output is specified to
precision at least~$\floor{N-B}$. We sometimes allow~$B$ to depend on~$N$, and often
use the Landau $O$ notation.  For example, if we say that \emph{given an input
  $x$ to precision~$N$, Algorithm~$\mathcal{A}$ computes $f(x)$ with a
  precision loss of $O(\log N)$ bits}, we really mean: \emph{there exist an
  absolute constant $K$ such that, given any $N\in \Z_{\geq 0}$ such that
  $N\geq K\log N$ and any~$x$ specified to precision~$N$,
  Algorithm~$\mathcal{A}$ succeeds and outputs~$f(x)$ to precision at
  least~$\floor{N - K\log N}$.} The precision losses in an algorithm, like its
complexity, can then be estimated by summing up the precision losses in each
individual step.

Our whole analysis is based on estimations on the precision losses and
complexities of certain elementary operations that we summarize in the next two
lemmas. Recall the notation introduced in~§\ref{subsec:intro-notation}.

\begin{lem}
  \label{lem:elementary-sqrt}
  Given~$z\in \C^\times$ to precision~$N\in \Z_{\geq 0}$, one can compute~$1/z$ and a square
  root of~$z$ in time $O(\M(N))$ with a precision loss
  of~$-2\log_2\abs{z} + O(1)$ and~$-\tfrac 12\log_2\abs{z} + O(1)$ bits
  respectively.
\end{lem}

\begin{lem}
  \label{lem:elementary}
  Let~$P_1,P_2\in \C[X]$ and~$N,N_1,N_2\in\Z_{\geq 0}$. Assume that $P_1,P_2$ and
  their approximations all have degree at most~$d\in \Z_{\geq 0}$. Then:
  \begin{enumerate}
  \item Given approximations of~$P_1,P_2$ to precision~$N$, the sum $P_1+P_2$
    can be computed in time
    $O\paren[\big]{(d+1)(N + \log\max\{1,\abs{P_1},\abs{P_2})}$ with a
    precision loss of~$O(1)$ bits.
  \item \label{it:elem-poly} Given approximations of~$P_k$ to precision~$N_k$
    for $k = 1,2$, the product $P_1P_2$ can be computed in time
    $O\paren[\big]{\M \paren[\big]{(d+1) \log(d+1)\max\{N_1 + \logp\abs{P_1}, N_2 +
        \logp\abs{P_2}\}}}$ to precision
    $\min\{N_1 - \logpt \abs{P_2}, N_2 - \logpt \abs{P_1}\} - \log_2(1 + d) -
    O(1)$.
  \end{enumerate}
\end{lem}

Note that in \cref{lem:elementary-sqrt}, the assumption on~$N$ when
computing~$1/z$ ensures that the input ball does not contain zero, and that
\cref{lem:elementary} includes the case of the sum and product of two complex
numbers ($d=0$).

\begin{proof}
  Let us only prove \cref{lem:elementary} as an illustration. Write
  \begin{displaymath}
    P_1 = \sum_{k=0}^d (x_k + \delta_k) X^k \quad\text{and}\quad
    P_2 = \sum_{k=0}^d (y_k + \eps_k) X^k
  \end{displaymath}
  where the $x_k,y_k$ for $0\leq k\leq d$ are known dyadic numbers (the
  midpoints of each coefficient of $P_1$ and $P_2$ respectively), and
  $\delta_k,\eps_k$ are unknown but satisfy $\abs{\delta_k} \leq 2^{-N_1}$
  and~$\abs{\eps_k}\leq 2^{-N_2}$ for all~$k$. Up to decreasing~$N_1$
  and~$N_2$ by~$O(1)$, we can assume that the coefficients~$x_k$ (resp.~$y_k$)
  are (Gaussian) integers divided by~$2^{-N_1-4}$ (resp.~$2^{-N_2-4}$). Let~$Q$
  be the exact product of the midpoint polynomials. After expanding the
  product~$P_1 P_2$, we obtain
  \begin{align*}
    \abs{P_1 P_2 - Q} &\leq (d + 1) \paren[\big]{2^{-N_1}\abs{P_2} + 2^{-N_2}\abs{P_1} + 2^{-N_1N_2}}\\
    &\leq 4 (d+1) \max\bigl\{2^{-N_1}, 2^{-N_1}\abs{P_1}, 2^{-N_2}, 2^{-N_2}\abs{P_2}\bigr\}.
  \end{align*}
  Thus~$Q$ is an approximation of~$P_1P_2$ to the required precision.

  After rescaling and taking real and imaginary parts, computing~$Q$ amounts to
  computing $O(1)$ product of polynomials in~$\Z[X]$ of degree~$d$ and whose
  coefficients have size
  $O\paren[\big]{\max\{N_1 + \logp\abs{P_1}, N_2 + \logp\abs{P_2}\}}$. Thus~$Q$
  can be computed within the required time bound using Kronecker substitutions.
\end{proof}

Any implementation of our algorithms using an interval arithmetic package in
which \cref{lem:elementary-sqrt,lem:elementary} hold will satisfy the
complexity bounds stated in the main theorems. In our implementation, we rely
on the Arb library~\cite{johanssonArbEfficientArbitraryprecision2017}. This
implementation has some additional features, e.g.~approximations of the real
and imaginary parts of~$z$ are stored separately as real intervals, but this
makes no difference in our theoretical analysis.

\subsection{Input, output, and certification}
\label{subsec:io}

The input to the evaluation algorithm consists in a tuple of invariants defined
over a number field~$L$. We consider two possible situations:
\begin{enumerate}
\item \label{it:input-1} In the first description, the number field
  is~$L = \Q(\alpha)$ where~$\alpha$ is a root of a monic polynomial
  $P\in\Z[X]$, and the input invariants belong to~$\Z[\alpha]$. This situation
  arises for instance when lifting from a finite field; not much is known
  about~$L$ itself.
\item \label{it:input-2} In the second description, we assume that a suitable
  $\Z$-basis of the ring of integers~$\Z_L$ is known (in fact a reduced basis
  in the sense of lattice reduction), and we assume that the input invariants
  are presented as quotients of elements in~$\Z_L$.
\end{enumerate}

The complexity estimates we will provide in~§\ref{sec:proofs} in these two
situations are different. In case~\ref{it:input-1}, we express them in terms of
the degree $d_L$ of~$L$ and bounds on the coefficients of~$P$ and the input
invariants, and deduce complexity estimates for the evaluation of modular
equations over finite fields. In case~\ref{it:input-2}, we express them in
terms of $d_L$, the discriminant~$\Delta_L$ of~$L$, and the absolute
logarithmic height of the input invariants, as defined
in~\cite[§B.2]{hindryDiophantineGeometry2000}. Of course, the distinction is
moot when~$L=\Q$, as in \cref{thm:intro}.

Recall the notations~$Q_\ell$ and~$P_{\ell,k}$ from~§\ref{subsec:siegel-denom}
(we consider the Siegel case here, as the Hilbert case is completely
analogous.) We have seen that they are polynomials in~$J_1,J_2,J_3$ with
integral coefficients and total degree at most
$10(\ell^3+\ell^2+\ell+1)/3$. Consequently, we immediately have:

\begin{lem}
  \label{lem:adapted-basis}
  \begin{enumerate}
  \item In case~\ref{it:input-1}, let $j_1,j_2,j_3\in \Z[\alpha]$. Then
    $Q_\ell(j_1,j_2,j_3)$ and all the coefficients of
    $P_{\ell,k}(j_1,j_2,j_3,X)$ for $1\leq k\leq 3$ lie in~$\Z[\alpha]$.
  \item In case~\ref{it:input-2}, let~$j_1,j_2,j_3\in L$, and let~$D\in \Z_L$
    be a common denominator such that $Dj_k\in \Z_L$ for each $1\leq k\leq 3$.
    Let~$D' = D^{\floor{10(\ell^3+\ell^2+\ell+1)/3}}$. Then
    $D' Q_\ell(j_1,j_2,j_3)$ and all the coefficients of
    $D' P_{\ell,k}(j_1,j_2,j_3,X)$ for $1\leq k\leq 3$ lie in~$\Z_L$.
  \end{enumerate}
\end{lem}

In the Hilbert case for $F = \Q(\sqrt{5})$, we would replace
$10(\ell^3+\ell^2+\ell+1)/3$ by $10(\ell+1)/3$.

In case~\ref{it:input-1}, let~$\mathcal{B}_L$ denote the power basis
$(1,\ldots,\alpha^{d_L-1})$ of~$\Z[\alpha]$; in case~\ref{it:input-2},
let~$\mathcal{B}_L$ be any $\Z$-basis of~$\Z_L$. Then $\mathcal{B}_L$ is a
$\Q$-basis of~$L$ in which the quantities of \cref{lem:adapted-basis} have
integral coordinates.  The skeleton of the evaluation algorithm outlined
in~§\ref{subsec:intro-overview} is then as follows.

\begin{algo}
  \label{algo:skeleton}
  \algoinput{Igusa invariants $j_1,j_2,j_3\in L$ which either lie in
    $\Z[\alpha]$ (case~\ref{it:input-1}) or which are integers in~$L$ divided
    by some~$D\in \Z_L$ (case~\ref{it:input-2}); the prime~$\ell$; a
    $\Q$-basis $\mathcal{B}_L$ of~$L$ as above.}
  \algooutput{$\Psi_{\ell,k}(j_1,j_2,j_3,X)\in L[X]$ for $1\leq k\leq 3$}.
  \begin{enumerate}
  \item Let $D' = 1$ (case~\ref{it:input-1}) or
    $D' = D^{\floor{10(\ell^3+\ell^2+\ell+1)/3}}$ (case~\ref{it:input-2}).
  \item Pick a starting precision $N\in \Z_{\geq 0}$.
  \item \label{step:skeleton-start} For each embedding~$\mu: L\to\C$,
    \begin{enumerate}
    \item \label{step:skeleton-embed} Compute $\mu(D')$ and $\mu(j_k)$ for
      $1\leq k\leq 3$ at working precision~$N$.
    \item \label{step:skeleton-eval} Compute
      $Q_\ell\paren[\big]{\mu(j_1),\mu(j_2),\mu(j_3)}$ and
      $P_{\ell,k}\paren[\big]{\mu(j_1),\mu(j_2),\mu(j_3), X}$ for
      $1\leq k\leq 3$ at working precision~$N$.
    \end{enumerate}
    In case of failure due to insufficient precision, double~$N$ and restart
    step~\ref{step:skeleton-start}.
  \item \label{step:skeleton-reconstruct} Recover $D' Q_\ell(j_1,j_2,j_3)$ and
    each coefficient of $D' P_{\ell,k}(j_1,j_2,j_3)$ for $1\leq k\leq 3$ as
    $\C$-linear combinations of the elements of~$\mathcal{B}_L$, viewing
    $L\otimes\C$ as a subspace of~$\C^{d_L}$ via the $d_L$ embeddings
    $\mu:L\embed\C$ at working precision~$N$.  In case of failure due to
    insufficient precision, double~$N$ and go back to
    step~\ref{step:skeleton-start}.
  \item \label{step:skeleton-reconstruct-ZZ} Replace each coefficient of the
    linear combination computed in step~\ref{step:skeleton-reconstruct}, seen
    as a complex ball, by the unique integers it contains. This yields
    $D' Q_\ell(j_1,j_2,j_3)$ and the coefficients of
    $D' P_{\ell,k}(j_1,j_2,j_3, X)$ for $1\leq k\leq 3$ as elements of~$L$.  If
    any of these balls contains several integers, double~$N$ and go back to
    step~\ref{step:skeleton-start}.
  \item Output
    $\displaystyle \Psi_{\ell,k}(j_1,j_2,j_3,X) = \frac{D'
      P_{\ell,k}(j_1,j_2,j_3,X)}{D' Q_\ell(j_1,j_2,j_3)}$ for $1\leq k\leq 3$.
  \end{enumerate}
\end{algo}

Of course, in step~\ref{step:skeleton-start}, it is enough to consider only one
embedding of~$L$ in each conjugate pair, and conjugate the results of
step~\ref{step:skeleton-eval} to obtain the other embedding.

Assuming that we have a provably correct numerical algorithm for
steps~\ref{step:skeleton-embed}, \ref{step:skeleton-eval}
and~\ref{step:skeleton-reconstruct}, \Cref{algo:skeleton} is also provably
correct by \cref{lem:adapted-basis}. To prove \cref{thm:intro}, we will
describe these algorithms and analyze their precision losses and complexity at
a given working precision~$N$. \Cref{algo:skeleton} terminates when~$N$ becomes
greater than these precision losses. In our theoretical analysis, big-$O$
estimate on these precision losses are sufficient; in practice, Arb keeps track
of error bounds throughout the algorithm, so it is easy to decide when to
stop. The rest of this section focuses on step~\ref{step:skeleton-eval}; for
steps~\ref{step:skeleton-embed} and~\ref{step:skeleton-reconstruct},
see~§\ref{sec:proofs}.

\begin{rem}
  \label{rem:tight-precision}
  In this paper, we always consider nonnegative absolute precisions
  (see~§\ref{subsec:model}), but Arb is capable of manipulating complex numbers
  with nonnegative relative precisions as well. In that case, it is also
  possible to attempt a first run through \cref{algo:skeleton} at a moderate,
  but probably insufficient precision (say~$N = 2000$ bits), observe the
  magnitude~$M$ of the error bounds we obtain in
  step~\ref{step:skeleton-reconstruct}, then increase~$N$ by $\log_2(M)$ plus a
  small margin. The algorithm is then likely to ``just'' succeed at the next
  iteration, leading to improved practical performances.
\end{rem}

In order to evaluate derivatives of modular equations with respect to one of
the invariants, we modify \cref{algo:skeleton} to use $Q^2$ as a denominator
instead, and replace step~\ref{step:skeleton-eval} by a numerical evaluation of
the partial derivatives. If~$Q$ is already a square (this happens for instance
in the Siegel case, as $\mf{g}_\ell$ is the square of a modular form with
integral coefficients), then~$Q^{3/2}$ is a valid denominator for these
derivatives.

\subsection{Numerical evaluation: the Siegel case}
\label{subsec:numerical-siegel}

In the Siegel case, step~\ref{step:skeleton-eval} of \cref{algo:skeleton}
relies on the formulas~\eqref{eq:siegelnum} to evaluate the numerator and
denominator of modular equations. To apply these formulas, we must compute a
period matrix~$\tau\in \Half_2$ whose Igusa invariants are
$\mu(j_1),\mu(j_2),\mu(j_3)$, then evaluate the modular forms $\mf{h}_k$ for
$k\in \{4,6,10,12\}$ at the matrices~$\gamma\tau\in \Half_2$
for~$\gamma\in C_\ell$.

Both computations crucially involve the usual fundamental domain~$\Fund_2$ for
the action of~$\Sp_4(\Z)$ on~$\Half_2$ defined
in~\cite[§1]{klingenIntroductoryLecturesIegel1990}, as well as certain
neighborhoods~$\Fund_2^\eps$ (for~$\eps>0$) of~$\Fund_2$, to be defined
below. Informally, period matrices~$\tau\in \Fund_2$ are computationally
convenient because the eigenvalues of~$\im(\tau)$ are bounded below, so
evaluating theta functions (hence modular forms) at~$\tau$ is fast.

\begin{defn}
  \label{def:fund}
  Fix $\eps\in \R_{\geq  0}$, and let
  \begin{displaymath}
    Y = \mat{y_1}{y_3}{y_3}{y_2}
  \end{displaymath}
  be a symmetric $2\times 2$ real matrix. Assume that $Y$ is positive
  definite. We say that $Y$ is \emph{$\eps$-Minkowski reduced} if
  \begin{displaymath}
    y_1 \leq (1+\eps) y_2 \quad\text{and}\quad
    -\eps y_1 \leq 2y_3 \leq (1+\eps)y_1.
  \end{displaymath}
  Let $\Gts\subset \Gamma(1)$ be the set of 19 matrices defining the boundary
  of~$\Fund_2$ as described
  in~\cite{gottschlingExpliziteBestimmungRandflaechen1959}; see
  also~\cite[§6.2]{strengComputingIgusaClass2014}. We define
  $\Fund_2^\eps\subset \Half_2$ as the set of all matrices $\tau\in\Half_2$
  such that
  \begin{enumerate}
  \item $\im(\tau)$ is $\eps$-Minkowski reduced,
  \item $\abs{\re(\tau)} \leq 1/2+\eps$,
  \item $\abs{\det(\sigma^*\tau)} \geq 1-\eps$ for every $\sigma\in\Gts$.
  \end{enumerate}
  The fundamental domain~$\Fund_2$ corresponds to the case
  $\eps=0$.
\end{defn}

In step~\ref{step:skeleton-eval} of \cref{algo:skeleton}, we proceed as
follows.

\begin{algo}
  \label{algo:numerical-siegel}
  \algoinput{The prime~$\ell$; a complex embedding~$\mu$ of~$L$; the Igusa
    invariants $\mu(j_1),\mu(j_2),\mu(j_3)$ to precision~$N$.
    Fix~$\eps = 2^{-10}$, say.}
  \algooutput{$Q_\ell\paren[\big]{\mu(j_1),\mu(j_2),\mu(j_3)}$
    and~$P_{\ell,k}\paren[\big]{\mu(j_1),\mu(j_2),\mu(j_3),X}$ for
    $1\leq k\leq 3$.}
  \begin{enumerate}
  \item \label{step:siegel-tau} Compute a period matrix~$\tau\in \Fund_2^\eps$
    whose Igusa invariants are $\mu(j_1)$, $\mu(j_2)$, $\mu(j_3)$.
  \item  For each~$\gamma\in C_\ell$,
    \begin{enumerate}
    \item \label{step:siegel-fd} Reduce~$\gamma\tau$ to~$\Fund_2^\eps$,
      i.e.~compute~$\eta_\gamma\in \Sp_4(\Z)$ such that
      $\eta_\gamma\gamma\tau \in \Fund_2^\eps$.
    \item \label{step:siegel-theta} Evaluate the
      squared theta constants~$\mf{\theta}_{a,b}^2(\eta_\gamma\gamma\tau)$ for
      $a,b\in \{0,1\}^2$. Deduce the values
      of~$\mf{h}_4,\mf{h}_6,\mf{h}_{10},\mf{h}_{12}$ at $\eta_\gamma\gamma\tau$
      using~\cite[eq.~(7.1)]{strengComputingIgusaClass2014}.
    \item \label{step:siegel-transf} Compute $\mf{h}_k(\gamma\tau)$ for
      $k\in \{4,6,10,12\}$ as
      $\det\paren[\big]{\eta_\gamma^*(\gamma\tau)}^{-k}\,
      \mf{h}_k(\eta_\gamma\gamma\tau)$.
    \end{enumerate}
  \item Evaluate~$\mf{h}_4,\ldots, \mf{h}_{12}$ at~$\tau$ via theta constants as in step~\ref{step:siegel-theta}.
  \item \label{step:siegel-prodtree} Evaluate $\mf{g}_\ell(\tau)$ and
    $\mf{g}_\ell(\tau) \Psi_{\ell,k}\paren[\big]{\mf{j}_1(\tau),\mf{j}_2(\tau),\mf{j}_3(\tau)}$ for $1\leq k\leq 3$
    using~\eqref{eq:gl},~\eqref{eq:siegelnum}.
  \item \label{step:siegel-lambda} Let~$w = 20(\ell^3+\ell^2+\ell+1)$ and
    evaluate
    $\lambda = 2^{\floor{7w/4}} 3^{\floor{w/4}} \mf{h}_4(\tau)^{\floor{w/6}- n}
    \mf{h}_{10}(\tau)^{-m}$ where~$m,n$ are as in \cref{prop:mf-poly}.
  \item \label{step:siegel-output} Output the results of
    step~\ref{step:siegel-prodtree} multiplied by~$\lambda$.
  \end{enumerate}
\end{algo}

\noindent
Let us now detail how computations are performed in each step.
\begin{itemize}
\item In step~\ref{step:siegel-tau}, we rely on the AGM
  method~\cite[Chap.~9]{dupontMoyenneArithmeticogeometriqueSuites2006}, which
  runs in quasi-linear time in~$N$. We describe this step and
  explain how to make it provably correct in~§\ref{sec:agm}.
\item In step~\ref{step:siegel-fd}, we adapt the usual reduction algorithm
  presented in~\cite[§6]{strengComputingIgusaClass2014}. A new contribution is
  that we analyze its complexity and precision losses in a completely numerical
  setting. We also discuss this in~§\ref{sec:agm}.
\item In step~\ref{step:siegel-theta}, we apply a quasi-linear time algorithm
  based on the AGM method and a Newton scheme to evaluate theta constants at a
  given point
  of~$\Fund_2^\eps$~\cite[Chap.~10]{dupontMoyenneArithmeticogeometriqueSuites2006}. This
  algorithm is shown to be certified and of uniform complexity on~$\Fund_2$
  in~\cite{kiefferCertifiedEwtonSchemes2022}. In the present work, we claim
  that \textbf{this algorithm is also quasi-linear and uniform
    on~$\Fund_2^\eps$ for~$\eps = 2^{-10}$}. This assertion could conceivably
  be checked by repeating the proofs
  in~\cite{kiefferCertifiedEwtonSchemes2022}. Alternatively, the author is
  currently working with Noam D.~Elkies on a new algorithm to evaluate theta
  functions which satisfies this claim. Yet another possibility would be to
  set~$\eps\simeq 2^{-N}$, so that a matrix in~$\Fund_2^\eps$ is always
  sufficiently close to a matrix in~$\Fund_2$
  where~\cite{kiefferCertifiedEwtonSchemes2022} applies out of the box.
\item Step~\ref{step:siegel-transf} is valid because the~$\mf{h}_k$ are
  modular forms for~$\Sp_4(\Z)$.
\item In step~\ref{step:siegel-prodtree}, we use product trees for efficiency
  reasons and to better manage precision losses. We postpone the detailed
  description of this step to~§\ref{sec:precision}.
\end{itemize}

\Cref{algo:numerical-siegel} can be generalized to also evaluate the
derivatives of modular equations of Siegel type with respect to the variables
$J_1,J_2,J_3$. If the $3\times 3$ matrix $\mf{dJ}(\tau)$ defined
in~\eqref{eq:dJ} is invertible, then we can compute
\begin{displaymath}
  \frac{\partial \Psi_{\ell,k}}{\partial J_n}
  \paren[\big]{\mu(j_1),\mu(j_2),\mu(j_3),X}
  \quad \text{for } 1\leq k, n\leq 3
\end{displaymath}
by modifying steps~\ref{step:siegel-theta}--\ref{step:siegel-prodtree}
of~\cref{algo:numerical-siegel} as follows.
\begin{itemize}
\item In step~\ref{step:siegel-theta}, we also compute the partial derivatives
  of theta constants with respect to the entries of~$\tau$
  using~\cite{kiefferCertifiedEwtonSchemes2022} (either using the fact that
  Newton iterations intrinsically compute derivatives, or using finite
  differences with an explicit upper bound on derivatives to the second
  order). We obtain the values of partial derivatives of the modular
  functions~$\mf{h}_k$ by differentiating the
  expressions~\cite[eq.~(7.1)]{strengComputingIgusaClass2014}.
\item In step~\ref{step:siegel-transf}, we differentiate the
  transformation formula. This yields the partial derivatives
  of~$\mf{j}_k(\gamma\tau)$ with respect to the entries of~$\tau$ for
  all~$1\leq k\leq 3$ and~$\gamma\in C_\ell$.
\item In step~\ref{step:siegel-prodtree}, we evaluate
  $\displaystyle\frac{\partial \Psi_{\ell,k}}{\partial \tau_n}(\tau,X)$ for
  $1\leq k,n\leq 3 $ by differentiating~\eqref{eq:siegelmodeq}.
 \item In addition, we evaluate derivatives of theta constants at~$\tau$ to
   compute~$\mf{dJ}(\tau)$. Then, we compute the inverse of this matrix
   numerically. If we cannot be certain that~$\det(\mf{dJ}(\tau))\neq 0$ at the
   current working precision, then the algorithm fails.
 \item Finally, we have
   \begin{displaymath}
     \paren[\bigg]{\frac{\partial \Psi_{\ell,k}}{\partial J_n}
       \paren[\big]{\mu(j_1),\mu(j_2),\mu(j_3),X}}_{1\leq k,n\leq 3} =  \paren[\bigg]{\frac{\partial \Psi_{\ell,k}}{\partial \tau_n}(\tau)}_{1\leq k,n\leq 3} \cdot  \mf{dJ}(\tau)^{-1}.
   \end{displaymath}
   We can then multiply this quantity by the correct denominator and continue
   as in step~\ref{step:skeleton-reconstruct} of \cref{algo:skeleton}.
\end{itemize}

With this strategy, evaluating the derivatives of modular equations is not much
more expensive than evaluating the modular equations themselves using
\cref{algo:numerical-siegel}: the cost and precision losses grow by a
(sizeable) constant factor. For this reason, evaluating derivatives of modular
equations will not be our main focus in the rest of the paper.

\subsection{Numerical evaluation: the Hilbert case}
\label{subsec:numerical-hilbert}

In the case of modular equations of Hilbert type of level~$\beta$ in Gundlach
invariants for~$F = \Q(\sqrt{5})$ or~$F = \Q(\sqrt{8})$, the numerical
evaluation algorithm (step~\ref{step:skeleton-eval} of \cref{algo:skeleton}) is
based on the formulas~\eqref{eq:modeq-hilbert} and~\eqref{eq:gbeta}. To apply
them, we must compute a period point~$t\in \Half_1^2$ associated with the given
pair of Gundlach invariants $\mu(g_1), \mu(g_2)$, and conversely, be able to
evaluate the Gundlach invariants at a given point of~$\Half_1^2$.

Since our knowledge of efficient numerical methods for Hilbert surfaces is less
advanced than in the Siegel case, we use the Siegel half-space~$\Half_2$
as much as possible, and transition to~$\Half_1^2$ by inverting the Hilbert
embedding~$\mathrm{Hilb}_F$ when necessary to enumerate~$\beta$-isogenous
abelian surfaces. However, inverting the Hilbert embedding is not a trivial
step. Even if~$\tau\in \Half_2$ is a period matrix of a p.p.~abelian
surface~$A$ with RM by~$\Z_F$ (say~$\tau\in\Fund_2^\eps$), there does not
always exist~$t\in \Half_1^2$ such that~$\tau = \mathrm{Hilb}_F(t)$: rather,
there exists~$t\in \Half_1^2$ and a matrix~$\eta\in \Sp_4(\Z)$ such that
$\eta\tau = \mathrm{Hilb}_F(t)$, as $\mathrm{Hilb_F}:\Gamma_F(1)\to\Sp_4(\Z)$
is not surjective. Inverting~$\mathrm{Hilb}_F$ then amounts to
computing~$\eta$.

In turn, computing~$\eta$ is essentially equivalent to computing the action
of~$\Z_F$ on the lattice~$\tau\Z^2+\Z^2$, as explained
in~\cite[§4]{birkenhakeHumbertSurfacesKummer2003}. More precisely, we consider
the endomorphism~$\xi = \sqrt{\Delta_F}$ or~$\xi = (1 + \sqrt{\Delta_F})/2$
of~$A(\tau) = \C^2/(\tau\Z^2+\Z^2)$ depending on whether~$\Delta_F = 0$
or~$\Delta_F = 1 \mod 4$. We can view~$\xi$ in two ways: either as a linear
map~$\C^2\to\C^2$ leaving~$\tau\Z^2 + \Z^2$ stable, represented by a
$2\times 2$ complex matrix, the \emph{analytic representation}~$\rho_A(\xi)$
of~$\xi$; or as an endomorphism of the lattice~$\tau\Z^2+\Z^2$, represented by
a $4\times 4$ integral matrix (in the basis obtained from the canonical basis
of each~$\Z^2$), the \emph{rational representation}~$\rho_R(\xi)$
of~$\xi$. These matrices satisfy the following
relation~\cite[eq.\,(6)]{birkenhakeHumbertSurfacesKummer2003}:
\begin{equation}
  \label{eq:endo}
  \rho_A(\xi) (\tau \ I_2) = (\tau \ I_2) \rho_R(\xi).
\end{equation}
Moreover, there exist~$a,b,c,d,e,n\in \Z$ such that
\begin{displaymath}
  \rho_R(\xi) =
  \begin{pmatrix}
    n & a & 0 & d\\
    -c & n + b & -d & 0 \\
    0 & e & n & -c \\
    -e & 0 & a & n + b
  \end{pmatrix}
  \quad\text{and}\quad
  \begin{cases}
    b^2 - 4ac - 4de = \Delta_F,\\
    2n + b = \Tr_{F/\Q}(\xi) \in \{0,1\}.
  \end{cases}
\end{displaymath}
The matrix~$\tau = \mat{\tau_1}{\tau_3}{\tau_3}{\tau_2} $ then satisfies the
``singular relation'' \cite[eq.\,(8)]{birkenhakeHumbertSurfacesKummer2003}
\begin{equation}
  \label{eq:singular}
  a\tau_1 + b\tau_3 + c\tau_2 + d\det(\tau) + e = 0.
\end{equation}

\begin{prop}
  \label{prop:birkenhake}
  Given integers~$(a,b,c,d,e)$ such that $b^2 - 4ac - 4de = \Delta_F$, one can
  compute a matrix~$\eta\in \Sp_4(\Z)$ with the following property: for
  every~$\tau\in \Half_2$ satisfying~\eqref{eq:singular}, the matrix~$\eta\tau$
  lies in the image of the Hilbert embedding~$\mathrm{Hilb}_F:\Half_1^2\to
  \Half_2$. This algorithm runs in quasi-linear time in
  $\log\max\{\abs{a},\abs{b},\abs{c},\abs{d},\abs{e}\}$.
\end{prop}

\begin{proof}
  Given our choice of matrix~$M_F$ in~\eqref{eq:Hilb}, there
  exists~$t\in \Half_1^2$ such that~$\tau = \mathrm{Hilb}_F(t)$ if and only if
  \begin{displaymath}
    (a,b,c,d,e) =
    \begin{cases}
      \paren[\big]{-\Delta_F/4, 0, 1, 0, 0} &\text{if } \Delta_F = 0\mod 4,\\
      \paren[\big]{(1 - \Delta_F)/4, 1, 1, 0, 0} & \text{if } \Delta_F = 1 \mod 4.
    \end{cases}
  \end{displaymath}
  Consequently, the algorithm we are looking for is
  exactly~\cite[Prop.~4.5]{birkenhakeHumbertSurfacesKummer2003}. Note that in
  Step~II of that proof, one can avoid using Dirichlet's theorem on primes in
  arithmetic progressions: we only need an integer~$n$ such that
  $p = (e_1 + na_1)/g_1$ is prime to~$c_1$. Such an~$n$ exists with
  $\abs{n}< c_1$. Then each integer~$m$ appearing in the algorithm satisfies
  \begin{displaymath}
    \logp \abs{m} = O\paren[\big]{\log\max\{1,\abs{a},\abs{b},\abs{c},\abs{d},\abs{e}\}}.
  \end{displaymath}
  Further, the algorithm only involves elementary arithmetic operations and
  gcd's which can be computed in quasi-linear time.
\end{proof}

Given~$\tau$, one could imagine solving for~$a,b,c,d,e$ numerically
using~\eqref{eq:singular} and the LLL algorithm, as in the usual strategy to
recognize an algebraic number of a known degree from a complex
approximation. This leads however to multiple issues, for instance when the
endomorphism ring of~$A$ is a quaternion algebra and thus contains multiple
quadratic fields. In~§\ref{sec:agm}, we propose an alternative strategy based
on computing the big period matrix, attached to~$\tau$, deducing~$\rho_A(\xi)$
from this data, and finally $\rho_R(\xi)$ using~\eqref{eq:endo}.

Besides inverting the Hilbert embedding, another issue is that
for~$\gamma\in C_\beta^{\mathrm{sym}}$ and~$t$ as above, the
matrices~$\mathrm{Hilb}_F(\gamma t)$ that appear in
equation~\eqref{eq:modeq-hilbert} may be very far from the fundamental
domain~$\Fund_2$. Running the reduction algorithm on each of these matrices
independently would be too expensive in general. Luckily, from the data
of~$\eta\in \Sp_4(\Z)$ such that~$\eta\tau = \mathrm{Hilb}_F(t)$ (where~$\tau$
is reduced) and~$\gamma\in C_\beta^{\mathrm{sym}}$, one can directly compute a
matrix~$\eta'\in \Sp_4(\Z)$ such that $\eta'\mathrm{Hilb}_F(\gamma t)$ is also
reasonably close to the cusp at infinity. Indeed, we have
\begin{displaymath}
  \mathrm{Hilb}_F(\gamma t) = \paren[\big]{\mathrm{Hilb}_F(\gamma) \eta} \tau,
\end{displaymath}
and we can run the following algorithm on the
matrix~$\mathrm{Hilb}_F(\gamma)\eta\in \GSp_4(\Q)$.

\begin{algo}
  \label{algo:change-rep}
  \algoinput{A matrix $\gamma\in\GSp_4(\Q)$ with integral coefficients.}
  \algooutput{A matrix $\eta\in \Sp_4(\Z)$ such that the lower left $2\times 2$
    block of~$\eta\gamma$ is zero.}
  \begin{enumerate}
  \item Let~$a,b,c,d$ be the $2\times 2$ blocks of~$\gamma$.
  \item \label{step:V-basis} Compute a $\Z$-basis~$(l_3,l_4)$ of $V\cap \Z^4$,
    where~$V\subset \Q^4$ is the $\Q$-vector space spanned by the two lines of
    the matrix $(-c^t\ \ a^t)$, using the classical algorithm for elementary
    divisors over~$\Z$.
  \item \label{step:sp-basis} Complete~$(l_3,l_4)$ into a symplectic basis
    $(l_1,l_2,l_3,l_4)$ of~$\Z^4$ for the standard symplectic pairing
    $\scal{u,v} =  u^t J v$, as follows.
    \begin{enumerate}
    \item Complete~$l_3$ into a basis $(x,y,l_3)$ of $W\cap \Z^4$,
      where~$W = \{v\in \Q^4: \scal{v,l_3} = 0\}$. Let~$l_2$ be an integral
      linear combination of~$x,y$ such that $\scal{l_2,l_4} = 1$, computed
      using an extended gcd.
    \item Compute a basis $(x,y)$ of $W\cap \Z^4$,
      where~$W = \{v\in \Q^4: \scal{v,l_2} = \scal{v,l_4} = 0\}$. Let~$l_1$ be
      an integral linear combination of~$x,y$ such that $\scal{l_1,l_3} = 1$,
      computed using an extended gcd.
    \end{enumerate}
  \item Output the matrix~$\eta$ whose lines are $l_1,\ldots,l_4$.
  \end{enumerate}
\end{algo}

\begin{prop}
  \label{prop:change-rep}
  \Cref{algo:change-rep} is correct and outputs~$\eta$ such that
  $\log\abs{\eta} = O\paren[\big]{\log\abs{\gamma}}$. It runs in quasi-linear
  time in $\log\abs{\gamma}$.
\end{prop}

\begin{proof}
  Step~\ref{step:V-basis} can be performed in
  time~$\Otilde\paren{\log\abs{\gamma}}$, and its output satisfies
  $\log\abs{l_k} = O(\log\abs{\gamma})$ for~$k = 3,4$. Since~$\gamma$ is
  symplectic,~$V$ is isotropic for the standard
  symplectic pairing, so step~\ref{step:sp-basis} succeeds. Extended
  gcd's can also be computed in quasi-linear time.
\end{proof}

With these ingredients in hand, the algorithm to numerically evaluate modular
equations of Hilbert type can be stated as follows. Recall the notations
$Q_\beta$ and $P_{\beta,k}$ from~§\ref{subsec:hilbert-denom}.

\begin{algo}
  \label{algo:numerical-hilbert}

  \algoinput{A totally positive~$\beta\in \Z_F$ for~$F = \Q(\sqrt{5})$ of prime
    norm~$\ell$; a complex embedding~$\mu$ of~$L$; the Gundlach invariants
    $\mu(g_1)$ and~$\mu(g_2)$ to precision~$N\in \Z_{\geq 0}$. Fix~$\eps = 2^{-10}$.}
  \algooutput{$Q_\beta\paren[\big]{\mu(g_1),\mu(g_2)}$
    and~$P_{\beta,k}\paren[\big]{\mu(g_1),\mu(g_2),X}$ for $1\leq k\leq 2$.}
  \begin{enumerate}
  \item \label{step:hilbert-j} Compute the Igusa
    invariants~$\mu(j_1),\mu(j_2),\mu(j_3)$ in~$L$ associated with the Gundlach
    invariants~$\mu(g_1),\mu(g_2)$
    using~\cite[Prop.~4.5]{lauterComputingGenusCurves2011}
    or~\cite[Cor.~A.15]{milioModularPolynomialsIlbert2020}. (This computation
    will actually be done in~$L$.)
  \item \label{step:hilbert-tau} Compute a period matrix~$\tau\in \Fund_2^\eps$
    whose Igusa invariants are~$\mu(j_1),\mu(j_2),\mu(j_3)$ using the AGM
    method.
  \item \label{step:hilbert-rm} Compute the analytic representation~$\rho_A(\xi)$ of the
    endomorphism~$\xi$ defined above.
  \item \label{step:hilbert-abcde} Compute the integers~$a,b,c,d,e,n$ as follows.
    \begin{enumerate}
    \item After taking imaginary parts in~\eqref{eq:endo}, we know
      $d\im(\tau_1)$. Use this to compute~$d$.
    \item Similarly, we know the quantities $n\im(\tau_1) - c\im(\tau_3)$ and
      $n\im(\tau_3) - c \im(\tau_2)$. Since~$\det\im(\tau)>0$, we can use this
      to compute~$n$ and~$c$.
    \item Similarly, we know $a\im(\tau_1) + (n+b)\im(\tau_3)$ and
      $a\im(\tau_3) + (n+b)\im(\tau_2)$. Use this to compute~$a$ and~$n+b$,
      hence~$b$.
    \item Finally, we know $n\tau_1 - c\tau_3 - e$. Use this to compute~$e$.
    \end{enumerate}
  \item \label{step:hilbert-eta} Use \cref{prop:birkenhake} to compute~$\eta\in \Sp_4(\Z)$ such
    that~$\eta\tau$ lies in the image of~$\mathrm{Hilb}_F$.
  \item \label{step:hilbert-loop} For each~$\gamma\in C_\beta^{\mathrm{sym}}$,
    \begin{enumerate}
    \item Run \cref{algo:change-rep} on $\mathrm{Hilb}_F(\gamma)\eta$ to
      obtain~$\zeta_\gamma\in \Sp_4(\Z)$ such that the lower left block
      of~$\zeta_\gamma \mathrm{Hilb}_F(\gamma)\eta$ is zero.
    \item \label{step:hilbert-fd} Compute~$\eta_\gamma\in \Sp_4(\Z)$ such that
      $\eta_\gamma\zeta_\gamma \mathrm{Hilb}_F(\gamma)\eta\tau \in \Fund_2^\eps$.
    \item \label{step:hilbert-theta} Evaluate the modular
      forms~$\mf{h}_4,\ldots,\mf{h}_{12}$ at this point of~$\Fund_2^\eps$ via theta
      constants. Deduce the values of~$\mf{h}_4,\ldots,\mf{h}_{12}$
      at~$\mathrm{Hilb}_F(\gamma)\eta \tau = \mathrm{Hilb}_F(\gamma t)$, where~$t\in \Half_1^2$
      is the period point such that~$\mathrm{Hilb}_F(t) = \eta\tau$, using the
      transformation law.
    \item \label{step:numerical-hmf} Deduce the values of the modular
      forms $\mf{G}_2, \mf{F}_6, \mf{F}_{10}, \mf{F}_{12}$ at~$\gamma t$
      using~\cite{resnikoffGradedRingIlbert1974}.
    \end{enumerate}
  \item \label{step:Gtau} Similarly evaluate~$\mf{h}_4,\ldots,\mf{h}_{12}$ at~$\tau$ and
    deduce the values of $\mf{G}_2,\mf{F}_6,\mf{F}_{10},\mf{F}_{12}$ at~$t$.
  \item \label{step:hilbert-prodtree} Evaluate $\mf{g}_\beta(t)$ and
    $\mf{g}_\beta(t) \Psi_{\beta,k}\paren[\big]{\mf{g}_1(t),\mf{g}_2(t),X}$ for $k = 1,2$. For
    the latter, we use a division-free formula analogous
    to~\eqref{eq:siegelnum}.
  \item Let~$w = 20(\ell + 1)$ and
    valuate~$\lambda = 2^{\floor{w/3}} \mf{G}_2(t)^{2\floor{w/6}- n}
      \mf{F}_{10}(t)^{-m}$ where~$m,n$ are defined as in
      \cref{prop:mf-poly-hilbert}.
  \item Output the results of step~\ref{step:hilbert-prodtree} multiplied by~$\lambda$.
  \end{enumerate}
\end{algo}

If~$F = \Q(\sqrt{8})$, we only have to replace the fundamental modular
forms~$\mf{G}_2, \mf{F}_6, \mf{F}_{10}, \mf{F}_{12}$ by the modular forms $\mf{G}_2, \mf{F}_4, \mf{F}_6$ at
steps~\ref{step:numerical-hmf} and~\ref{step:Gtau}, and adapt how
\cref{prop:mf-poly-hilbert} is used given that the denominator $\mf{g}_\beta$ has a
higher weight.

In order to evaluate the derivatives of Hilbert modular equations with respect
to the variables $G_1$ and~$G_2$, we proceed as
in~§\ref{subsec:numerical-siegel}: we evaluate the derivatives
of~$\Psi_{\beta,k}\paren[\big]{\mf{g}_1(t),\mf{g}_2(t),X}$ with respect to the
two entries of~$t\in \Half_1^2$, evaluate the derivatives of Gundlach
invariants at~$t$, and apply the chain rule. To do so, we must assume that the
$2\times 2$ matrix $\mf{dG}(t)$, defined in an analogous way to $\mf{dJ}(\tau)$
in~§\ref{subsec:siegel-invariants}, is invertible. This is the case as soon
as~$(g_1,g_2)$ lies outside a certain Zariski-closed set of codimension~1 in
affine 2-space.

\section{Fundamental algorithms for period matrices}
\label{sec:agm}

In this section, we focus on the fundamental algorithms on genus~$2$ period
matrices that appear in
\cref{algo:numerical-siegel,algo:numerical-hilbert}. These are the reduction
algorithm to the fundamental domain~$\Fund_2$ and its variants~$\Fund_2^\eps$
(§\ref{subsec:fund}); the AGM method (§\ref{subsec:agm}) and its complexity
analysis (§\ref{subsec:agm-cost}); gaining information on bases of differential
forms after the AGM method (§\ref{subsec:differentials}); and finally,
computing the analytic representation of a real
endomorphism~(§\ref{subsec:analytic-rep}).

\subsection{The reduction algorithm}
\label{subsec:fund}

The algorithm to reduce a given period matrix to~$\Fund_2$, and its complexity
analysis, is presented by Streng
in~\cite[§6]{strengComputingIgusaClass2014}. Streng considers period matrices
attached to abelian surfaces with complex multiplication whose entries are
algebraic numbers, and uses exact computations. Here, we adapt the reduction
algorithm to the setting of interval arithmetic and provide a new analysis of
complexity and precision losses. Recall that~$\Sigma\subset\Sp_4(\Z)$ denotes
the finite set of matrices defining the boundary of~$\Fund_2$ introduced in
\cref{def:fund}.

\begin{algo}
  \label{algo:reduction}
  \algoinput{A matrix~$\tau_0\in \Fund_2$ to precision~$N\in \Z_{\geq 0}$.}
  \algooutput{A matrix~$\gamma\in \Sp_4(\Z)$ such that~$\gamma\tau_0$ lies
    in~$\Fund_2$ or is very close to~$\Fund_2$.}
  \begin{enumerate}
  \item Let~$\gamma = I_4$ and $\tau = \tau_0$.
  \item \label{step:red-imag} Compute a matrix~$U\in \GL_2(\Z)$ such that $U \im(\tau) U^t$ is
    (very close to being) Minkowski-reduced as follows.
    \begin{enumerate}
    \item Compute the Cholesky decomposition~$\im(\tau) = M^t M$
      of~$\im(\tau)$, where~$M$ is upper-triangular.
    \item \label{step:red-M} Let~$B\in \Z_{\geq 0}$ be maximal such that the error
      radius of each entry of~$M$ is at most~$2^{-B}$, and let~$M'$ be the
      matrix obtained by reducing the coefficients of~$2^{B}M$ to the nearest
      integers. If~$\det(M')=0$, then the algorithm fails.
    \item Apply a quasi-linear version of Gauss's algorithm
      to~$M'$~\cite{yapFastUnimodularReduction1992} to obtain a reduced basis
      of the lattice~$M'\Z^2$, and let~$U$ be the base-change matrix.
    \end{enumerate}
  \item Multiply~$\gamma$ on the left by $\mat{U}{0}{0}{U^{-t}}$ and
    replace~$\tau$ by $U\tau U^{t}$.
  \item \label{step:red-real} Compute a symmetric matrix~$S\in \mathrm{Mat}_{2\times 2}(\Z)$ such
    that $\abs{\re(\tau) - S}$ is (very close to being) smaller than $1/2$,
    by rounding the entries of $\re(\tau)$.
  \item Multiply~$\gamma$ on the left by $\mat{I_2}{-S}{0}{I_2}$ and
    replace~$\tau$ by~$\tau - S$.
  \item \label{step:red-sigma} Find~$\sigma\in \Sigma$ such that (an exact
    upper bound we compute on) $\det(\sigma^*\tau)$ is minimal. If this upper
    bound is not strictly less than~$1$, then exit the loop and output~$\gamma$, otherwise continue.
  \item \label{step:red-recompute} Multiply~$\gamma$ on the left
    by~$\sigma$, recompute~$\tau = \gamma\tau_0$, and go back to step~\ref{step:red-imag}.
\end{enumerate}
\end{algo}

In order to round a real number, encoded as an interval, to the nearest integer
as in steps~\ref{step:red-M} and~\ref{step:red-real}, we simply reduce its
midpoint to the nearest integer. In most cases, even when we work in interval
arithmetic, we have that $U^t\im(\tau) U$ is Minkowski-reduced in
step~\ref{step:red-imag} and $\abs{\re(\tau) - S}\leq 1/2$ in
step~\ref{step:red-real}, but adding the words ``very close to'' is necessary
for the algorithm to work on every input.

In order to study the behavior and cost of \cref{algo:reduction}, we introduce
the following notation. For $\tau\in\Half_2$, we define
\begin{displaymath}
  \Lambda(\tau) = \log\max\{2, \abs{\tau}, \det\im (\tau)^{-1}\}.
\end{displaymath}
Denote by $\lambda_1(\tau)\leq \lambda_2(\tau)$ the two eigenvalues of
$\im(\tau)$, and by $\nu_1(\tau)\leq \nu_2(\tau)$ the successive minima of
$\im(\tau)$ on the lattice $\Z^2$. We also write
\begin{displaymath}
  \Xi(\tau) = \log \max\{2, \nu_1(\tau)^{-1}, \nu_2(\tau)\}.
\end{displaymath}
By \cite[eq.~(6.4)]{strengComputingIgusaClass2014}, we always have
\begin{equation}
  \label{eq:minima}
  \dfrac{3}{4} \nu_1(\tau) \nu_2(\tau) \leq \det\im(\tau) \leq \nu_1(\tau) \nu_2(\tau),
\end{equation}
so that
\begin{displaymath}
  \log\max\{\lambda_1(\tau)^{-1},\lambda_2(\tau),\nu_1(\tau)^{-1},
  \nu_2(\tau)\} = O(\Lambda(\tau)) \quad \text{and} \quad \Xi(\tau) = O(\Lambda(\tau)).
\end{displaymath}

\begin{thm}
  \label{thm:reduction} There exists an absolute constant~$K>0$ such that the
  following holds: for every~$\tau_0\in \Half_2$, \cref{algo:reduction} is
  given as input an approximation of~$\tau_0$ at absolute
  precision~$N\geq 2K\Lambda(\tau_0)$ and the working precision~$N$, then the
  algorithm succeeds, outputs a matrix~$\gamma\in \Sp_4(\Z)$ such that
  $\gamma\tau_0\in \Fund_2^\eps$ where~$\eps = 2^{-N + K\Lambda(\tau_0)}$
  and~$\abs{\gamma}\leq K\Lambda(\tau_0)$, and runs in time
  $O\paren[\big]{\Xi(\tau_0)\M(N)\log N}$.
\end{thm}

The proof relies on the following lemmas.

\begin{lem}
  \label{lem:star}
  Let $\tau,\tau'\in \Half_2$, and assume that there exists
  $\gamma\in \Gamma(1)$ such that $\tau'=\gamma\tau$. Then we have
  \begin{align*}
    \logp \max\{\abs[\big]{\gamma^*\tau},\abs{(\gamma^*\tau)^{-1}}\}
    &= O\paren[\big]{\max\{\Lambda(\tau),\Lambda(\tau')\}},\\
    \log\abs{\gamma}
    &= O\paren[\big]{\max\{\Lambda(\tau),\Lambda(\tau')\}}.
  \end{align*}
\end{lem}

\begin{proof}
  Let $M$ be a real matrix such
  that $M^t M = \im(\tau)$. Then we have
  \begin{displaymath}
    \im(\tau') = (\gamma^*\tau)^{-t} \im(\tau) (\gamma^*\conj{\tau})^{-1}
    = M'^t \conj{M'}
  \end{displaymath}
  where $M' = M (\gamma^*\tau)^{-1}$. Since
  $\abs{M}\leq \abs{\im(\tau)}^{1/2}$ and
  $\abs{M'}\leq \abs{\im(\tau')}^{1/2}$, we obtain
  \begin{displaymath}
    \abs{\gamma^*\tau} = \abs{M'^{-1} M}
    \leq 2 \dfrac{\abs{M'}}{\det(M')} \abs{M}
  \end{displaymath}
  so
  $\logp\abs{\gamma^*\tau} = O(\max\{\Lambda(\tau),\Lambda(\tau')\})$,
  and similarly for $(\gamma^*\tau)^{-1}$.

  We now bound~$\abs{\gamma}$. Let~$a,b,c,d$ be the $2\times 2$ blocks
  of~$\gamma$. Then $\im(\gamma^*\tau) = c\im(\tau)$, so
  $\logp\abs{c} = O(\max\{\Lambda(\tau),\Lambda(\tau')\})$, and in turn
  \begin{displaymath}
    \logp\abs{d} \leq \logp(\abs{c\tau} + \abs{\gamma^*\tau}) =
  O(\max\{\Lambda(\tau),\Lambda(\tau')\}).
  \end{displaymath}
  Finally, we obtain upper bounds on~$\abs{a}$ and~$\abs{b}$ using the relation
  $a\tau+b = \tau'(c\tau+d)$.
\end{proof}

\begin{lem}
  \label{lem:minkred}
  There exists an absolute constant~$K>0$ such that the following holds. Let
  $\tau\in \Half_2$ and $\eps>0$. Then, given~$\tau$ to precision
  $N \geq K\Lambda(\tau) + \abs{\log_2\eps}$, step~\ref{step:red-imag} of
  \cref{algo:reduction} computes $U\in\GL_2(\Z)$ such that
  $U \im(\tau) U^t$ is $\eps$-Minkowski reduced
  and~$\log\abs{U}\leq K\Lambda(\tau)$ in time~$O(\M(N)\log N)$.
\end{lem}

\begin{proof}[Proof (sketch)]
  Informally, if~$K$ is correctly chosen, then all the numerical errors we
  encounter in step~\ref{step:red-imag} are extremely small compared to
  $\det\im(\tau)$. In particular, the Cholesky decomposition succeeds and the
  matrix~$M'$ is still invertible. The base change matrix $U\in\GL_2(\Z)$
  satisfies $ \log\abs{U} = O(\Lambda(\tau))$ by
  \cite[Lem.~6.6]{strengComputingIgusaClass2014}. As a consequence, the matrix
  $U \im(\tau) U^t$ is $\eps$-Minkowski reduced.
\end{proof}

\begin{lem}
  \label{lem:reduction}
  Assume that the absolute errors on the entries of~$\tau$ during
  \cref{algo:reduction} remains smaller than $2^{-10}$. Then the number of
  loops in the algorithm is~$O(\Xi(\tau_0))$. Moreover, during the execution of
  the algorithm, the quantities
  $\abs[\big]{\log \paren[\big]{\abs{\det (\gamma^*\tau_0)}}}$, $\Lambda(\tau)$
  and~$\log \abs{\gamma}$ remain in~$O(\Lambda(\tau_0))$.
\end{lem}

\begin{proof}
  The number of iterations is $O(\Xi(\tau_0))$ by
  \cite[Prop.~6.16]{strengComputingIgusaClass2014}: observe that
  \cite[Lem.~6.11 and 6.12]{strengComputingIgusaClass2014} still apply, because
  $\det \im(\tau)$ increases strictly each time through
  step~\ref{step:red-sigma}. The proof of
  \cite[Lem.~6.16]{strengComputingIgusaClass2014} also remains valid, with
  slightly worse constants. This shows that~$\log\abs{\tau}$ and
  $\log\abs{\det(\gamma^*\tau_0)}$ remain in~$O(\Lambda(\tau_0))$.

  During the execution of the algorithm, we have
  $\logp \nu_2(\tau) = O(\Lambda(\tau_0))$
  \cite[Lem.~6.13]{strengComputingIgusaClass2014}. Moreover
  $\det\im(\tau)\geq \det\im(\tau_0)$, so
  \begin{displaymath}
    \nu_1(\tau)^{-1} \leq \dfrac{\nu_2(\tau)}{\det\im(\tau)}
    \leq \dfrac{\nu_2(\tau)}{\det\im(\tau_0)}
    \leq \dfrac{4\nu_2(\tau)}{3\nu_1(\tau_0)^2}
  \end{displaymath}
  by~\eqref{eq:minima}. Thus $\Lambda(\tau) = O(\Lambda(\tau_0))$. The
  remaining upper bounds follow from \cref{lem:star}.
\end{proof}

\begin{proof}[Proof of \cref{thm:reduction}.]
  By \cref{lem:reduction}, there exists a constant~$K'$ such that
  $\log\abs{\gamma}\leq K' \Lambda(\tau)$ during \cref{algo:reduction} as long
  as the absolute errors remain smaller than~$2^{-10}$. Given
  step~\ref{step:red-recompute} of the algorithm and \cref{lem:minkred}, by
  choosing~$K$ appropriately, we can ensure that the absolute errors remain
  smaller than $2^{-10}$ at every step. For such a choice of~$K$, the upper
  bounds on~$\log\abs{\gamma}$ and~$\Lambda(\tau)$ remain valid until the end
  of the algorithm. By \cref{lem:reduction}, there are $O(\Xi(\tau_0))$
  iterations of the loop in \cref{algo:reduction}.

  We now consider the complexity of each iteration. By \cref{lem:minkred}, the
  Minkowski reduction step can be performed in time $O(\M(N)\log N)$. Rounding
  the real part and choosing~$\sigma$ can be done in time~$O(N)$ and~$O(\M(N))$
  respectively. Therefore, the whole algorithm runs in time
  $O\paren[\big]{\Xi(\tau_0)\M(N)\log N}$.

  Finally, we show that $\tau = \gamma\tau_0$ lies in~$\Fund_2^\eps$ at the end
  of the algorithm. Up to increasing~$K$, we can ensure that the error radii on
  the entries of~$\tau$ are at most $\eps \cdot 2^{-K'\Lambda(\tau)}$, for any
  fixed choice of constant~$K'$. For a good choice of~$K'$, the fact that we
  exited the main loop in step~\ref{step:red-sigma} implies that
  $\det(\sigma^*\tau) \geq 1-\eps$ for each~$\sigma\in \Sigma$, one of the
  defining conditions of~$\Fund_2^\eps$ in \cref{def:fund}. Moreover,
  step~\ref{step:red-real} ensures that $\abs{\re(\tau)}\leq 1/2 + \eps$, and
  $\im(\tau)$ is $\eps$-Minkowski reduced by \cref{lem:minkred}.
\end{proof}

\begin{rem}
  \label{rem:streng-complexity}
  The complexity bound of \cref{thm:reduction} appears to be sharper than
  \cite[Thm.~6.17]{strengComputingIgusaClass2014}, where the large height of
  the algebraic numbers appearing in the reduction algorithm adversely impacts
  the complexity estimate.
\end{rem}

\subsection{The AGM method}
\label{subsec:agm}

The AGM method to compute a period matrix $\tau\in \Fund_2$ from a triple of
Igusa invariants
over~$\C$~\cite[Chap.~9]{dupontMoyenneArithmeticogeometriqueSuites2006}
consists of two main steps:
\begin{enumerate}
\item First, we recover the set of squared theta quotients
  $\mf{\theta}_{a,b}^2(\tau)/\mf{\theta}_0^2(\tau)$ for $a,b\in \{0,1\}^2$, which is
  slightly finer information than the triple of Igusa invariants. We do this by
  first reconstructing a genus~2 curve over the complex numbers that has the
  required invariants using Mestre's
  algorithm~\cite{mestreConstructionCourbesGenre1991}, then apply Thomae's
  formula~\cite[Thm.~IIIa.8.1]{mumfordTataLecturesTheta1984} to recover the
  associated squared theta quotients.
\item \label{step:sieveout} Second, we compute the limits of certain AGM
  sequences constructed from these theta values to obtain the period
  matrix~$\tau\in \Fund_2$ as in \cite[Algorithm
  13]{dupontMoyenneArithmeticogeometriqueSuites2006}.
\end{enumerate}

We actually have to be more careful. Thomae's formula involves a choice
of ordering of the Weierstrass points of the genus~2 curve as well as certain
sign choices, and only certain choices will provide the correct theta quotients
at~$\tau\in \Fund_2$. In the first step, we apply the strategy of~\cite[§9.2.3,
``Variantes'']{dupontMoyenneArithmeticogeometriqueSuites2006} and compute
several possible vectors of theta quotients. In step~\ref{step:sieveout}, we
are then able to sieve out most of the phony values. At the end of the sieving
procedure, we might still obtain several reasonable-looking
matrices~$\tau\in \Fund_2$ (though we don't expect this to happen generically.)
We then proceed to a validation step where we attempt to prove that the Igusa
invariants of the given complex ball~$\tau$ provably contain the input
values. We can then safely output any~$\tau$ that passes this test.

While the whole algorithm might seem intricate, it is a very favorable way of
computing period matrices in genus~2. Its complexity is quasi-linear in the required
precision as AGM sequences converge rapidly. Moreover, it only involves
elementary operations on complex numbers, so we are able analyze its complexity
and precision losses explicitly.

\begin{rem}
  \label{rem:complexity-periods}
  To our knowledge, an analysis of the complexity of computing periods of
  genus~2 curves in terms of the curve or its invariants has never appeared in
  the literature. We note that computing periods to low precision
  using~\cite{molinComputingPeriodMatrices2019} would provide information on
  the correct sign choices in Thomae's formula. It would be interesting to
  analyze the complexity of such a computation in terms of the input curve.
\end{rem}

In order to explain how the sieving-out process works, we provide more details
on the AGM method, starting with the definition of an AGM sequence.

\begin{defn}
  \label{def:agm}
  We say that a sequence $(S_n)_{n\geq 0}$, where each $S_n$ is a quadruple of
  complex numbers~$s_{n,b}$ indexed by $b\in \{0,1\}^2$, is an \emph{AGM sequence with
    good sign choices} (in genus~$2$) if the following conditions hold:
  \begin{itemize}
  \item For every~$n\geq 0$, there exists a tuple of square roots
    $(u_{n,b})_{b\in \{0,1\}^2}$ of the entries of~$S_n$ that lie in a common
    open quarter plane centered at the origin in~$\C$, i.e.~$u_{n,b}\neq 0$ and
    the angle between any two square roots $u_{n,b}$ and~$u_{n,b'}$ is less
    than $\pi/2$.  These square roots~$u_{n,b}$ are said to be a \emph{good
      choice} of (signs of the) square roots.
  \item For every~$n\geq 0$, if $(u_{n,b})_{b\in \{0,1\}^2}$ is a good choice of
    square roots of the entries of~$S_n$, then
    $S_{n+1} = (s_{n+1,b})_{b\in \{0,1\}^2}$, where
    \begin{displaymath}
      s_{n+1,b} = \frac 14 \sum_{b'\in \{0,1\}^2} u_{n,b'} u_{n,b + b'}
    \end{displaymath}
    for each~$b\in \{0,1\}^2$. The addition $b+b'$ is understood as in
    $(\Z/2\Z)^2$.
  \end{itemize}
  The AGM sequence with good sign choices starting from a given~$S_0$ is unique
  if it exists. In that case, the complex numbers $s_{n,b}$ as $n\to\infty$
  converge quadratically to a common $s_\infty\in \C^\times$ called the
  \emph{limit} of the AGM sequence.
\end{defn}

The method also involves the action of certain matrices of~$\Sp_4(\Z)$. Define
the matrices $J, M_1, M_2, M_3\in \Sp_4(\Z)$ whose actions on $\tau\in\Half_2$
are given by
\begin{displaymath}
  \begin{matrix}
  &J\tau = -\tau^{-1},&\qquad M_1\tau = \tau+\mat{1}{0}{0}{0},\\[1em]
  &M_2\tau = \tau+\mat{0}{0}{0}{1},&\quad\text{and}\quad M_3\tau = \tau+\mat{0}{1}{1}{0}.
  \end{matrix}
\end{displaymath}
Let $\gamma_k = (JM_k)^2$ for $1\leq k\leq 3$, and further
let~$\gamma_0 = I_4$. Given theta quotients at $\tau$, we can compute the
values of the modular functions
\begin{displaymath}
  \mf{f}_b = \mf{\theta}_{00,b}^2/\mf{\theta}_0^2 \quad \text{for } b\in \{0,1\}^2
\end{displaymath}
at $\gamma_k\tau$ for $0\leq k\leq 3$ using the transformation
formula~\cite[§II.5]{mumfordTataLecturesTheta1983}. We then have the following
result, stated
as~\cite[Conj.~9.1]{dupontMoyenneArithmeticogeometriqueSuites2006} and proved
in~\cite{kiefferSignChoicesAGM2022}.

\begin{prop}
  \label{prop:sign-choices}
  For every~$\tau\in \Half_2$ and every~$0\leq k\leq 3$, the AGM sequence with
  good sign choices starting from the quadruple
  $\paren[\big]{\mf{f}_b(\gamma_k\tau)}_{b\in \{0,1\}^2}$
  converges to $1/\mf{\theta}_0^2(\gamma_k\tau)$. For every~$n\geq 0$, its $n$-th
  term is the quadruple
  \begin{displaymath}
    \frac{1}{\mf{\theta}_0^2(\gamma_k\tau)}\paren[\big]{\mf{\theta}_{00,b}^2(2^n\gamma_k\tau)}_{b\in \{0,1\}^2}.
  \end{displaymath}
\end{prop}

Computing these limits (in interval arithmetic) gives us access to the
matrix~$\tau$ using the relations stated
after~\cite[Conj.~9.1]{dupontMoyenneArithmeticogeometriqueSuites2006}:
\begin{equation}
  \label{eq:dupont}
  \begin{aligned}
    \mf{\theta}_0^2(\gamma_1\tau) &= - i \tau_1 \mf{\theta}_{01,00}^2(\tau),\\
    \mf{\theta}_0^2(\gamma_2\tau) &= - i \tau_2 \mf{\theta}_{10,00}^2(\tau),\\
    \mf{\theta}_0^2(\gamma_2\tau) &= (\tau_3^2 - \tau_1\tau_2)\mf{\theta}_0^2(\tau).
  \end{aligned}
\end{equation}
The theta values on the right hand side can be computed from the known values
of~$\mf{\theta}_0^2(\tau)$ and the square theta
quotients~$\mf{\theta}_{a,b}^2(\tau)/\mf{\theta}_0^2(\tau)$. We arrive at the
following algorithm, which
recasts~\cite[§9.2.3]{dupontMoyenneArithmeticogeometriqueSuites2006} in a
provably correct setting.

\begin{algo}
  \label{algo:agm}
  \algoinput{A candidate vector
    $\paren[\big]{\mf{\theta}_{a,b}^2(\tau)/\mf{\theta}_0^2(\tau)}_{a,b\in
      \{0,1\}^2}$ to precision~$N$ for an unknown~$\tau\in \Fund_2$.}
  \algooutput{Either an approximation of~$\tau$, or one of the two special
    values \wrongc\ or \insufp.}
  \begin{enumerate}
  \item Compute $\paren[\big]{\mf{f}_b(\gamma_k\tau)}_{b\in \{0,1\}^2}$ for each
    $0\leq k\leq 3$ with the transformation
    law~\cite[§II.5]{mumfordTataLecturesTheta1983}.
  \item \label{step:agm-aux} For each $0\leq k\leq 3$, attempt to compute
    $1/\mf{\theta}_0^2(\gamma_k\tau)$ as follows.
    \begin{enumerate}
    \item \label{step:agm-1} Compute successive terms~$(s_{n,b})$ of the AGM
      sequence with good sign choices starting from the quadruple
      $\paren[\big]{\mf{f}_b(\gamma_k\tau)}_{b\in \{0,1\}^2}$ until
      $\abs{s_{n,b} - s_{n,00}} \leq 2^{-10} \abs{s_{n,00}}$ for each~$b$. If
      at any point a good choice of square roots does not exist (resp.~cannot
      be determined but might exist), stop and output \wrongc\ (resp. \insufp.)
      Let~$n_0 = n$.
    \item \label{step:agm-2} Compute further terms~$S'_n$ of the AGM sequence
      with good sign choices starting from
      $ (s_{n_0,b}/s_{n_0,00})_{b\in \{0,1\}^2}$ until we obtain an
      approximation of its limit~$s'_\infty$ up to an error of~$2^{-N}$. This
      is the case when the four entries of~$S'_n$ are within a distance
      $2^{-N-6}$ of each other
      by~\cite[Prop.~7.1]{dupontMoyenneArithmeticogeometriqueSuites2006}. The
      product $s_{n_0,00}s'_\infty$ is then a complex ball containing
      $1/\mf{\theta}_0^2(\gamma_i\tau)$.
    \end{enumerate}
  \item \label{step:agm-check} Compute~$\tau$
    using~\eqref{eq:dupont}. If~$\tau$ certainly lies outside~$\Fund_2$
    (resp.~does not certainly lie inside $\Fund_2^\eps$ for~$\eps = 2^{-10}$),
    output \wrongc\ (resp.~\insufp.)  Otherwise, output~$\tau$.
  \end{enumerate}
\end{algo}

Given \cref{prop:sign-choices}, \cref{algo:agm} terminates and is
correct in the following sense:
\begin{itemize}
\item If its output is \wrongc, then the input does not correspond to squared theta
  quotients of any~$\tau\in \Fund_2$.
\item Conversely, if the input consists of theta quotients at a
  valid~$\tau\in \Fund_2$ and if the input precision is large enough, then its
  output is an approximation of~$\tau$.
\end{itemize}
We make the required input precision in \cref{algo:agm} explicit
in~§\ref{subsec:agm-cost}: see \cref{thm:agm-cost}.

After this sieving-out process, we use a validation step to certify that a
matrix~$\tau$ output by \cref{algo:agm} indeed has the requested Igusa
invariants. To this end, we use an explicit version of the inverse function
theorem. Recall the notation~$\norm{\cdot}$ from~§\ref{subsec:intro-notation}.

\begin{lem}
  \label{lem:inverse-fct} Let~$n\geq 1$ and fix constants~$\delta>0$
  and~$B\geq 0$.  Let~$x_0\in \C^n$, and let~$\mathcal{V}\subset\C^n$ denote
  the open ball centered in~$x_0$ of radius~$\delta$ for~$\norm{\cdot}$.  Let
  $f: \mathcal{V}\to \C^n$ be a holomorphic function such that~$df(x_0)$ is
  invertible and~$df$ is $B$-Lipschitz for~$\norm{\cdot}$
  on~$\mathcal{V}$. Let~$y\in \C^n$. If
  \begin{displaymath}
    \norm{f(x_0) - y} \norm[\big]{df(x_0)^{-1}}< \delta \quad\text{and}\quad
    \norm{f(x_0) - y} \norm[\big]{df(x_0)^{-1}}^2 B < \frac 12,
  \end{displaymath}
  then the equation~$f(x)= y$ has a unique solution in~$\mathcal{V}$.
\end{lem}

\begin{proof}
  In fact, one can prove that Newton's method starting from~$x_0$ converges to
  a solution of~$f(x) = 0$. The case~$\mathcal{V}\subset \R^n$ is known as (a consequence
  of) Kantorovich's theorem, and the case~$\mathcal{V}\subset\C^n$ follows by
  considering~$\C$ as~$\R^2$.
\end{proof}

We apply this lemma to the function~$f = (\mf{j}_1,\mf{j}_2,\mf{j}_3)$ (using
$\tau_1,\tau_2,\tau_3$ as coordinates) under the assumption that the
matrix~$\mf{dJ}(\tau)$ from~\eqref{eq:dJ} is invertible.

\begin{algo}
  \label{algo:agm-verification}

  \algoinput{A dyadic matrix~$\tau_0\in \Fund_2^\eps$ for~$\eps = 2^{-10}$; the
    putative theta values $\mf{\theta}_{a,b}^2(\tau_0)/\mf{\theta}_0^2(\tau_0)$ for
    all characteristics $a,b\in \{0,1\}^2$ to precision~$N$.}  \algooutput{A
    radius~$\delta>0$ such that there certainly exists~$\tau\in \Half_2$ such that
    $\norm{\tau - \tau_0} < \delta$ and which realizes the corresponding Igusa
    invariants, or \unknown.}
  \begin{enumerate}
  \item \label{step:verif-theta} Evaluate~$f(\tau_0)$ and~$\mf{dJ}(\tau_0)$ to
    precision~$N$
    using~\cite{kiefferCertifiedEwtonSchemes2022}. If~$\mf{dJ}(\tau_0)$ is not
    invertible, then stop and output \unknown. Otherwise, compute
    $\norm[\big]{\mf{dJ}(\tau_0)^{-1}}$.
  \item Compute~$f(\tau)$ from the input data,
    then~$\norm{f(\tau_0)-f(\tau)}$. Compute~$\delta > 0$ such that
    $\norm{f(\tau_0)-f(\tau)} \norm[\big]{\mf{dJ}(\tau_0)^{-1}} < \delta$, and
    compute~$B\geq 0$ such that $df$ is $B$-Lipschitz on the ball~$\mathcal{V}$
    centered in~$\tau_0$ of radius~$\delta$. This can be done using uniform
    bounds for theta functions and their derivatives on~$\Fund_2^\eps$, as
    in~\cite[§9]{klingenIntroductoryLecturesIegel1990}
    or~\cite[§6.2]{dupontMoyenneArithmeticogeometriqueSuites2006}.
  \item If~$\norm{f(\tau_0)-f(\tau)} \norm[\big]{\mf{dJ}(\tau_0)^{-1}}^2 B < 1/2$,
    then output~$\delta$, otherwise output \unknown.
  \end{enumerate}
\end{algo}

\Cref{algo:agm-verification} is correct by \cref{lem:inverse-fct}. We prove
in~§\ref{subsec:agm-cost} that its output is a positive (and small)
radius~$\delta$ and not \unknown\ if the input~$\tau_0$ is close enough to the
actual~$\tau\in \Fund_2$ corresponding to the specified theta values.

We put the previous buiding blocks together in the following algorithm for
computing period matrices using the AGM in step~\ref{step:siegel-tau} of
\cref{algo:numerical-siegel} and step~\ref{step:hilbert-tau} of
\cref{algo:numerical-hilbert}. To ease the complexity analysis
of~§\ref{subsec:agm-cost}, we do not make Mestre's algorithm or Thomae's
formula part of the algorithm: we assume that we have computed a genus~$2$
curve realizing the required Igusa invariants and its Weierstrass points, and
return to this part in~§\ref{sec:precision}.

\begin{algo}
  \label{algo:periods}

  \algoinput{A monic polynomial~$f\in \C[X]$ of degree~$6$ with simple
    roots~$\rho_1,\ldots,\rho_6 \in \C$; a working
    precision~$N\in \Z_{\geq 0}$. Fix~$\eps=2^{-10}$.}  \algooutput{A period
    matrix~$\tau\in \Fund_2^\eps$ for the Jacobian~$\Jac(\Crv_f)$ of the curve
    $\Crv_f: y^2 = f(x)$.}

  \begin{enumerate}
  \item \label{step:periods-choices} Let~$\mathcal{Z}$ denote the set of all
    $720\times 16 = 11520$ possible choices of orderings of the roots of~$f$ as
    well as subsequent sign choices in Thomae's formula. Pick a starting
    precision~$N'=400$, say. If at any point~$\mathcal{Z}$ is empty, then the
    algorithm fails.
  \item \label{step:periods-agm} For each~$z\in \mathcal{Z}$, compute the associated
    candidate vector~$\mathcal{Q}(z)$ of squared theta quotients
    $\mf{\theta}_{a,b}^2(\tau)/\mf{\theta}_0^2(\tau)$ by applying Thomae's
    formula at precision~$N'$. Apply \cref{algo:agm} to~$\mathcal{Q}(z)$. If
    the output is \wrongc, then remove~$z$ from~$\mathcal{Z}$.
  \item \label{step:periods-loop} If~$\mathcal{Z}$ has only one element~$z$,
    run~\cref{algo:agm} again on~$\mathcal{Q}(z)$ at the full precision~$N$ and
    output the result (if it is \insufp, then the algorithm fails.) Otherwise,
    if~$2N' \leq N$, then replace~$N'$ by~$2N'$ and go back to
    step~\ref{step:periods-agm}; if $2N' > N$, then go
    to~step~\ref{step:periods-sieve}.
  \item \label{step:periods-sieve} (At this step, the set~$\mathcal{Z}$
    contains at least two, but hopefully much less than 11520 elements.) For
    each~$z\in \mathcal{Z}$, apply \cref{algo:agm} at the working
    precision~$N$. If the output is not a period matrix, then remove~$z$
    from~$\mathcal{Z}$, otherwise denote the output
    by~$\tau(z)$. If~$\mathcal{Z}$ contains only one element~$z$ and all the
    other outputs were \wrongc, then stop and output~$\tau(z)$.
  \item \label{step:periods-theta} (At this step, we expect that~$\mathcal{Z}$ is even
    smaller than in step~\ref{step:periods-sieve}.) For each~$z\in \mathcal{Z}$, evaluate the
    theta quotients~$\mf{\theta}_{a,b}^2/\mf{\theta}_0^2$ at~$\tau(z)$ to
    precision~$N$. If their intersection as complex balls with~$\mathcal{Q}(z)$ is
    empty, then remove~$z$ from~$\mathcal{Z}$.
  \item \label{step:periods-verif} Run \cref{algo:agm-verification} on the
    midpoint~$\tau_0(z)$ of~$\tau(z)$ for each~$z\in \mathcal{Z}$. If it not
    successful or its output does not lie in~$\Fund_2^\eps$, then remove~$z$
    from~$\mathcal{Z}$.  Otherwise let~$\delta(z)$ be the computed
    radius. Otherwise, find an element~$z\in \mathcal{Z}$ for which~$\delta(z)$
    is minimal, and output~$\tau_0(z)$ with an error radius of~$\delta(z)$.
  \end{enumerate}
\end{algo}

\begin{rem}
  \label{rem:typical-agm}
  One can further tweak \cref{algo:periods} to enhance its performance in
  practice. For instance, the initial value of~$N'$ can be chosen such that the
  very last value of~$N'$ in step~\ref{step:periods-loop} is
  approximately~$N/2$. One could also consider quadrupling instead of
  doubling~$N'$ to shorten the loop at the cost of potentially more expensive
  computations. Yet another possibility would be to compute theta values to
  sieve out period matrices as in step~\ref{step:periods-theta} earlier in the
  algorithm, i.e.~in step~\ref{step:periods-agm}. In
  step~\ref{step:periods-verif}, one might consider applying
  \cref{algo:agm-verification} only once if certain matrices~$\tau(z)$ are
  numerically indistinguishable.

  On a typical instance, we observe that on the first run, the
  list~$\mathcal{Z}$ contains only one element in step~\ref{step:periods-loop}.
  Steps~\ref{step:periods-sieve}--\ref{step:periods-verif} are necessary in
  certain special cases, e.g.~when the AGM sequences appearing for the
  legitimate matrix~$\tau\in \Fund_2$ all consist of positive real numbers: for
  other sign choices, we obtain both positive and negative real numbers, and we
  can never certify that a good choice of square roots does not exist in
  step~\ref{step:agm-1} of \cref{algo:agm}.
\end{rem}

\begin{rem}
  \label{rem:dtheta-as-dj}
  The assumption that~$\mf{dJ}(\tau)$ is invertible in \cref{algo:periods} is
  not an essential one: we could apply \cref{lem:inverse-fct} to the map
  $\tau\mapsto \paren[\big]{\mf{\theta}_{a,b}^2(\tau)}_{a,b}$ as in
  \cref{rem:16-thetas}. However, the image has~16 coordinates, so we would need
  to choose a subset of~$3$ coordinates (depending on~$\tau$) to project to, and
  compute an explicit neighborhood where this projection map is injective. The
  above algorithm avoids these technicalities.
\end{rem}

\subsection{Complexity and precision losses in the AGM method}
\label{subsec:agm-cost}

We analyze the complexity of \cref{algo:periods} in terms of a
quantity~$\Pi(f)$ defined as follows.

\begin{defn}
  \label{def:Pi}
  Let~$f\in \C[X]$ be a monic degree~$6$ polynomial with simple
  roots~$\rho_1,\ldots,\rho_6$. Let~$\Crv_f$ be the genus~$2$ curve with
  equation $y^2 = f(x)$, let~$j_1,j_2,j_3$ be the Igusa invariants of
  $\Jac(\Crv_f)$, and let~$\tau\in \Half_2$ be any period matrix
  of~$\Jac(\Crv_f)$. We then define
  \begin{displaymath}
    \Pi(f) = \log\max\Bigl\{2, \max_{1\leq i\leq 6} \abs{\rho_i}, \max_{1\leq i < j\leq 6}\abs{\rho_i-\rho_j}^{-1},
    \frac{\mf{h}_{10}(\tau)^3}{\det(\mf{dJ}(\tau))^{10}}\Bigr\}.
  \end{displaymath}
  Note that $\mf{h}_{10}(\tau)^3/ \det(\mf{dJ}(\tau))^{10}$ does not depend on the
  choice of~$\tau$ because~$\det(\mf{dJ})$ is a modular function of
  weight~$3$.\footnote{ If we adapt \cref{algo:periods} to not rely on the
    invertibility of~$\mf{dJ}(\tau)$, then one could presumably
    omit~$\mf{h}_{10}(\tau)^3/\det(\mf{dJ}(\tau))^{10}$ from the definition of~$\Pi(f)$.}
\end{defn}

\begin{thm}
  \label{thm:agm-cost}
  Let~$f\in \C[X]$ be a monic complex polynomial with simple roots. Then
  \cref{algo:periods}, on the input of~$f$ and its complex roots to
  precision~$N$, runs in time of~$O(\M(N)\log N)$ with a precision loss
  of~$O(\log N + \Pi(f) \log \Pi(f))$ bits, and outputs a period
  matrix~$\tau\in \Fund_2^\eps$ such that~$\abs{\tau} = O(\Pi(f))$.
\end{thm}

We divide the proof into several lemmas. First, we prove that $\abs{\tau}$ is
bounded above in terms of~$\Pi(f)$. Then, we analyze the precision losses in
\cref{algo:agm} in terms of~$\abs{\tau}$ and conclude on the cost of the whole
algorithm.

\begin{lem}
  \label{lem:size-periods} Let $\tau\in\Fund_2^\eps$ be any period matrix
  of~$\Jac(\Crv_f)$ as above. Then the quantities $\abs{\tau}$,
  $\log\max\{\mf{h}_{10}(\tau), \mf{h}_{10}(\tau)^{-1}\}$, and
  $\logp\abs{\mf{dJ}(\tau)}$ are all in $O(\Pi(f))$.
\end{lem}

\begin{proof} By \cite[Prop.~7.6]{strengComputingIgusaClass2014}, we have
  for~$\tau\in \Fund_2$:
  \begin{align*}
    \abs{\mf{\theta}_0(\tau) - 1} &< 0.405,\\
    \abs{\mf{\theta}_{01,00}(\tau)/2\exp(i\pi\tau_1/4) - 1} &< 0.348,\\
    \abs{\mf{\theta}_{10,00}(\tau)/2\exp(i\pi\tau_2/4) - 1} &< 0.348.
  \end{align*}
  Moreover, the absolute values of all theta constants are uniformly bounded
  above on~$\Fund_2$. Similar inequalities hold for~$\tau\in \Fund_2^\eps$ with
  slightly larger upper bounds. This already shows
  that~$\abs{\mf{h}_{10}(\tau)} = O(1)$. On the other hand, in Thomae's
  formula~\cite[Thm.~IIIa.8.1]{mumfordTataLecturesTheta1984}, the quotients
  $\mf{\theta}_{a,b}^4(\tau)/\mf{\theta}_0^4(\tau)$ are expressed as quotients of
  differences of roots of~$f$. Therefore,
  \begin{displaymath}
    \abs[\big]{\log \paren[\big]{\abs{\mf{\theta}_{a,00}(\tau)/\mf{\theta}_0(\tau)}}} = O(\Pi(f))
    \quad \text{for } a\in\{01,10\}.
  \end{displaymath}
  Thus~$\abs{\im(\tau_1)}$ and~$\abs{\im(\tau_2)}$ are both
  in~$O(\Pi(f))$. Since~$\tau\in \Fund_2^\eps$, we also
  have~$\det\im(\tau)\geq 0$ and~$\abs{\re(\tau)}\leq 1$, so
  overall~$\abs{\tau} = O(\Pi(f))$.

  Next, we show that $\log \abs{\mf{h}_{10}(\tau)^{-1}} = O(\Pi(f))$. Given the
  correspondence between Siegel modular forms and covariants of binary forms
  (see e.g.~\cite{igusaSiegelModularForms1962} or~§\ref{subsec:differentials}
  below), there exists a scalar~$\lambda\in \C^\times$ such
  that~$\mf{h}_k(\tau) = \lambda^k H_k(f)$ for $k\in \{4,6,10,12\}$,
  where~$H_4,\ldots,H_{12}$ are certain fixed polynomials in the coefficients
  of~$f$. Further,~$H_{10}(f)$ is the discriminant of~$f$. Since the modular
  forms $\mf{h}_k$ do not have common zeroes on~$\Fund_2^\eps$, and moreover~$\mf{h}_4$
  and~$\mf{h}_6$ are not cusp forms, there exists an absolute~$\eps>0$ such that
  $\max_{k}\abs{\mf{h}_k(\tau)} \geq \eps$. On the other
  hand~$\abs{H_k(f)} = O(\Pi(f))$ for each~$k$, so
  $\log \abs{\lambda^{-1}} = O(\Pi(f))$. Since
  $\log\abs{H_{10}(f)^{-1}} = O(\Pi(f))$ by definition of this quantity, we
  conclude that $\log\abs{\mf{h}_{10}(\tau)}^{-1} = O(\Pi(f))$.

  Finally, the entries of~$\mf{dJ}(\tau)$ are all of the
  form~$\mf{g}(\tau)/\mf{h}_{10}(\tau)^k$ where~$\mf{g}$ is a polynomial in
  (partial derivatives of) Siegel modular forms and~$k\in
  \Z_{\geq 0}$. Since~$\abs{\mf{g}(\tau)} = O(1)$, we have
  $\logp\abs{\mf{dJ}(\tau)} = O(\Pi(f))$.
\end{proof}

Next, we need explicit bounds on the convergence of the AGM sequences appearing
in step~\ref{step:agm-aux} of \cref{algo:agm}. To this end, we
use~\cref{prop:sign-choices} and the fact that theta constants converge quickly
as the smallest eigenvalue~$\lambda_1(\tau)$ of~$\im(\tau)$ tends to
infinity. Moreover, $\lambda_1(\gamma_k\tau)$ for each $1\leq k\leq 3$ is
bounded below when $\tau\in\Fund_2^\eps$ and $\abs{\tau}$ is not too large.

\begin{lem}
  \label{lem:theta-bounds}
  For every $\tau\in\Half_2$ such that $\lambda_1(\tau)\geq 1$, we
  have
  \begin{displaymath}
    \abs{\mf{\theta}_{00,b}(\tau)-1}
    < 4.18 \exp({-\pi\lambda_1(\tau)})
    \qquad \text{for } b\in \{0,1\}^2.\\
  \end{displaymath}
\end{lem}

\begin{proof} This is similar
  to~\cite[Prop.~6.1]{dupontMoyenneArithmeticogeometriqueSuites2006}. Using the
  series expansion of $\mf{\theta}_{00,b}$, we obtain
  \begin{align*}
    |\mf{\theta}_{00,b}(\tau)-1|
    &\leq \sum_{n\in\Z^2\backslash\{0\}}\exp({-\pi n^t\im(\tau)n})
      \leq \sum_{n\in\Z^2\backslash\{0\}}\exp({-\pi \lambda_1(\tau) \norm{n}^2}).
  \end{align*}
  By splitting the plane into quadrants, we see that this last sum is
  equal to $4S^2 + 4S$, with
  \begin{displaymath}
    S = \sum_{n\geq 1} \exp({-\pi \lambda_1(\tau)n^2})
     \leq \dfrac{\exp({-\pi\lambda_1(\tau)})}{1- \exp({-3\pi\lambda_1(\tau)})}.
  \end{displaymath}
  Since $\lambda_1(\tau)\geq 1$, the conclusion follows.
\end{proof}

\begin{lem}
  \label{lem:gammatau}
  Let $\tau\in\Half_2$ and $\gamma\in\Sp_4(\Z)$. Then
  \begin{displaymath}
    \lambda_1(\gamma\tau) \geq  \dfrac{\det\im(\tau)}{8 \abs{\gamma}^2 \abs{\tau}(2\abs{\tau}+1)^2}.
  \end{displaymath}
\end{lem}

\begin{proof}
  We have
  \begin{displaymath}
    \lambda_1(\gamma\tau)
    \geq \dfrac{\det\im(\gamma\tau)}{\Tr\im(\gamma\tau)}.
  \end{displaymath}
  By \cite[eq.~(6.10)]{strengComputingIgusaClass2014}, we also have
  \begin{displaymath}
    \im(\gamma\tau) = (\gamma^* \tau)^{-t} \im(\tau) (\gamma^* \bar{\tau})^{-1}
  \end{displaymath}
  so that
  \begin{align*}
    \det\im(\gamma\tau)
    &= \dfrac{\det\im(\tau)}{\abs{\det(\gamma^*\tau)}^2},
    \quad \text{and}\\
    \Tr\im(\gamma\tau)
    &\leq 8 \abs{(\gamma^* \tau)^{-1}}^2 \abs{\im(\tau)}
      \leq 8 \dfrac{\abs{\gamma^* \tau}^2 \abs{\tau}}{\abs{\det(\gamma^*\tau)}^2}
      \leq 8\dfrac{\abs{\gamma}^2(2\abs{\tau}+1)^2\abs{\tau}}{\abs{\det(\gamma^*\tau)}^2}.
      \qedhere
  \end{align*}
\end{proof}

\begin{prop}
  \label{prop:dupont}
  Let~$\tau\in \Fund_2$. Then \cref{algo:agm}, given theta
  quotients at~$\tau$ to precision~$N$, runs in time~$O(\M(N)\log N)$ with
  a precision loss of~$O(\log N + \abs{\tau}\log\abs{\tau})$ bits.
\end{prop}

\begin{proof}
  By \cref{lem:gammatau},
  $\abs{\log \lambda_1(\gamma_k\tau)} = O(\log\abs{\tau})$ for each
  $1\leq k\leq 3$.  By \cref{lem:theta-bounds}, we deduce that the number of
  AGM steps to be done in step~\ref{step:agm-1} of \cref{algo:agm} is
  $O(\log\abs{\tau})$.  Up to this point, we performed $O(\log\abs{\tau})$
  elementary operations with complex numbers $x$ such that
  $\log\abs{x} = O(\abs{\tau})$. Therefore the total running time is
  \begin{displaymath}
    O(\M(N+\abs{\tau})\log\abs{\tau}),
  \end{displaymath}
  and the precision loss is $O(\abs{\tau}\log\abs{\tau})$ bits. If $N$ is
  larger than this precision loss, then everything we compute is given at a
  sufficiently high precision so that the good choice of square roots can
  indeed be determined at each step.

  Then, by~\cite[Prop.~7.2]{dupontMoyenneArithmeticogeometriqueSuites2006}, the
  number of AGM steps in step~\ref{step:agm-2} is~$O(\log N)$. The complex
  numbers~$x$ we manipulate satisfy $\log\abs{x} = O(1)$, so this step
  costs~$O(\M(N)\log N)$, and the precision loss is $O(\log N)$ bits.

  Finally, step~\ref{step:agm-check} of \cref{algo:agm} runs in
  time~$O(\M(N+\abs{\tau})$ involves a precision loss of~$O(\abs{\tau})$
  bits. Up to increasing the $O$ constant, the approximation of~$\tau$ we
  obtain is precise enough to land in~$\Fund_2^\eps$, so the algorithm
  succeeds.
\end{proof}

\begin{proof}[Proof of \cref{thm:agm-cost}.]
  In step~\ref{step:periods-agm} of \cref{algo:periods}, applying Thomae's
  formula involves $O(1)$ elementary operations (including square roots) on
  complex numbers~$x$ such that $\logp\abs{x} = O(\Pi(f))$ at precision~$N'$.
  By \cref{lem:elementary}, this is done in $O(\M(N'))$ binary operations with
  a precision loss of~$O(\Pi(f))$ bits. By \cref{prop:dupont}, \cref{algo:agm}
  in step~\ref{step:periods-loop} runs in time~$O(\M(N')\log N')$ with a
  precision loss of~$O(\log N' + \Pi(f)\log\Pi(f))$ bits. Since we keep
  doubling~$N'$, the total complexity is dominated by the very last loop where
  $N' = N$, so is $O(\M(N)\log N)$. Then, provided that~$N$ is larger than the
  precision loss, we will have computed the period matrix for the correct
  candidate~$z\in \mathcal{Z}$ with a precision loss
  of~$O(\log N + \Pi(f)\log\Pi(f))$ bits.

  The complexity of steps~\ref{step:periods-sieve}--\ref{step:periods-theta} is
  bounded above by the same quantity. Moreover, the list~$\mathcal{Z}$ after
  step~\ref{step:periods-theta} still contains the period matrix of the correct
  candidate $z\in \mathcal{Z}$.

  Finally, in step~\ref{step:periods-verif}, we have (using the notations from
  \cref{algo:agm-verification})
  \begin{displaymath}
    \log_2 \norm{f(\tau) - f(\tau_0)} \leq -N + O(\log N + \Pi(f)\log \Pi(f))
  \end{displaymath}
  given the checks performed in step~\ref{step:periods-theta}. By definition
  of~$\Pi(f)$, the same upper bound holds on~$\log_2\delta$. Thus the
  verification step succeeds if~$N \gg \log N + \Pi(f)\log\Pi(f)$. The output
  matrix~$\tau$ satisfies $\abs{\tau} = O(\Pi(f))$ by \cref{lem:size-periods}.
\end{proof}

As indicated above, in order to fully analyze step~\ref{step:siegel-tau} of
\cref{algo:numerical-siegel}, we also need to evaluate the complexity of
computing~$f$ and to provide an upper bound on~$\Pi(f)$. We defer these
questions to~§\ref{sec:precision}.

\subsection{Gaining information on differential forms}
\label{subsec:differentials}

The goal of this subsection is to explain how one can recover the information
contained in the so-called ``big'' period matrix of a p.p.~abelian surface
over~$\C$ from the data of a ``small'' period matrix (an element of~$\Half_2$),
for instance obtained from the AGM method. We first review the notion of a big
period matrix and its algorithmic uses.

Let~$A$ be a p.p.~abelian surface defined over~$\C$. Then $H_1(A,\Z)$ is a
rank~4 lattice inside $H_1(A,\C)\simeq \C^2$, and~$A$ is isomorphic
to~$\C^2/H_1(A,\Z)$. We can view the polarization on~$A$ as a symplectic
pairing~$E$ on~$H_1(A,\Z)$ (arising as the imaginary part of a hermitian form
on~$\C^2$) with elementary divisors $(1,1)$
\cite[§3.1]{birkenhakeComplexAbelianVarieties2004}. This means that there exist
bases of~$H_1(A,\Z)$, called \emph{symplectic bases}, in which the matrix
of~$E$ is
\begin{displaymath}
  \begin{pmatrix}
    0 & I_2 \\ -I_2 & 0
  \end{pmatrix}.
\end{displaymath}

\begin{defn}
  \label{def:big-periods}
  Let~$\phi = (\phi_1,\ldots,\phi_4)$ be a symplectic basis of~$H_1(A,\Z)$, and
  let $\omega = (\omega_1,\omega_2)$ be a basis of~$H^0(A,\Omega^1_A)$, the
  vector space of holomorphic differential forms on~$A$. The \emph{big period
    matrix} of~$A$ with respect to the bases~$\phi$ and~$\omega$ is
  \begin{displaymath}
    (\Omega_0 \ \Omega_1) = \paren[\Big]{\int_{\phi_k}{\omega_i}}_{\substack{1\leq i\leq 2\\1 \leq k\leq 4}}.
  \end{displaymath}
  Both~$\Omega_0$ and~$\Omega_1$ are complex $2\times 2$ matrices. The
  \emph{small period matrix} of~$A$ with respect to~$\phi$ is
  $\tau = \Omega_1^{-1}\Omega_0$. The matrix~$\tau$ is independent of the
  choice of~$\omega$ and lies in~$\Half_2$
  \cite[Prop.~8.1.1]{birkenhakeComplexAbelianVarieties2004}. Moreover,~$A$ is
  isomorphic to~$A(\tau) = \C^2/(\tau\Z^2+\Z^2)$.
\end{defn}

In other words, when we consider a (small) period matrix~$\tau\in \Half_2$ for
a p.p.~abelian surface~$A$, we implicitly make two choices: we choose a
symplectic basis~$\phi$ of $H_1(A,\Z)$, and fix the basis~$\omega$
of~$H^0(A,\Omega^1_A)$ that is dual to the first half of~$\phi$ with respect
to the integration pairing. The big period matrix of~$A$ in the bases~$\phi$
and~$\omega$ is then $(\tau\ I_2)$.

Providing a big period matrix is therefore equivalent to providing a small
period matrix \emph{and} the attached basis of differential forms on~$A$. This
input is required, for instance, to analytically compute isogenies between
Jacobians of genus~2 curves over number
fields~\cite[§3]{vanwamelenComputingAnalyticJacobian2006},~\cite[§6.2]{bookerDatabaseGenusCurves2016},
or when inverting the Hilbert embedding in the next subsection. One obtains a
big period matrix directly when computing periods of~$A$ using numerical
integration~\cite{molinComputingPeriodMatrices2019}, but it seems to be lost
with the AGM
method~\cite[Rem.~2.2.7]{costaRigorousComputationEndomorphism2019}.

We now explain how to recover this information a posteriori in quasi-linear
time.  We rely on the fundamental Siegel modular function~$\mf{\chi}_{-2,6}$
studied in \cite[§5]{cleryCovariantsBinarySextics2017}
and~\cite[Ex.~16.4]{clerySiegelModularForms2015}. This modular function is
vector-valued of weight~$\det^{-2}\otimes\Sym^6$: for background on
vector-valued Siegel modular forms, we refer
to~\cite{vandergeerSiegelModularForms2008}. For any~$n\in \Z_{\geq 0}$, we view~$\Sym^n$
as the following representation of~$\GL_2(\C)$ on the vector space~$\C_n[X]$ of
polynomials of degree at most~$n$: for all
$r = \paren[\big]{\begin{smallmatrix} a & b \\c & d
\end{smallmatrix}}\in \GL_2(\C)$ and $P\in \C_n[X]$, we set
\begin{displaymath}
  \Sym^n(r) P = (bX + d)^n P\paren[\Big]{\frac{aX + c}{bX + d}}.
\end{displaymath}
We also view an element of~$\C_n[X]$ as an $(n+1)$-dimensional complex
(vertical) vector in the basis~$(X^n,X^{n-1},\ldots,1)$. This allows us to also
consider~$\Sym^n(r)$ as an $(n+1)\times (n+1)$ matrix. In particular, we have
\begin{equation}
  \label{eq:sym2}
  \Sym^2(r) =
  \begin{pmatrix}
    a^2 & ab & b^2 \\
    2ac & ad+bc & 2bd \\
    c^2 & cd & d^2
  \end{pmatrix}.
\end{equation}
The modular function~$\mf{\chi}_{-2,6}$ takes values in~$\C_6[X]$. It has an expression in
terms of theta functions as follows:
\begin{equation}
  \label{eq:chi}
  \mf{\chi}_{-2,6} = \prod_{a^tb\,\text{even}} \frac{1}{\mf{\theta}_{a,b}^2}\ \cdot\prod_{a^tb\,\text{odd}}
    \paren[\Big]{\frac{\partial \mf{\theta}_{a,b}}{\partial z_1} x + \frac{\partial \mf{\theta}_{a,b}}{\partial z_2}}
\end{equation}
where the products run over all theta characteristics $(a,b)\in \{0,1\}^4$, and
$z_1$, $z_2$ denotes the two coordinates of~$\C^2$. Thus we see
that~$\mf{\chi}_{-2,6}(\tau)$ is well-defined precisely when~$\mf{h}_{10}(\tau)\neq 0$,
i.e.~when~$A(\tau)$ is the Jacobian of a genus~2 curve. Using~\eqref{eq:chi},
the modular function~$\mf{\chi}_{-2,6}$ can be evaluated in quasi-linear time in the
required precision at any point~$\tau\in \Fund_2$ where it is defined, but the
denominator prevents this complexity from being uniform.

Crucially, evaluating~$\mf{\chi}_{-2,6}$ at~$\tau\in \Half_2$ provides the equation
of a genus~$2$ curve attached to~$\tau$ with a precise control on differential
forms.

\begin{thm}[{\cite[§4]{cleryCovariantsBinarySextics2017}}]
  \label{thm:covariants}
  Let~$\tau\in \Half_2$ be such that~$\mf{h}_{10}(\tau)\neq 0$. Let~$\Crv_f$ be
  the curve with equation $y^2 = f(x)$ where
  $f = \mf{\chi}_{-2,6}(\tau)\in \C_6[X]$. Then~$\Jac(\Crv_f)$ admits~$\tau$ as
  a small period matrix. Moreover, there exists an isomophism
  $\Jac(\Crv_f) \simeq A(\tau)$ under which the basis $(x\, dx/y, dx/y)$
  of~$H^0\paren[\big]{\Jac(\Crv_f),\smash{\Omega^1_{\Jac(\Crv_f)}}}$
  corresponds to the basis of differential forms
  $(2\pi i\, dz_1, 2\pi i\, dz_2)$ on~$A(\tau)$.
\end{thm}

For a proof that the unidentified scalar factor
in~\cite{cleryCovariantsBinarySextics2017} is~$2\pi i$, see
\cite[§3]{kiefferComputingIsogeniesModular2025}.

\Cref{thm:covariants} provides a fruitful correspondence between vector-valued
Siegel modular forms on~$\Half_2$ and covariants of binary sextics.
For~$k,n \in \Z$ with~$n\geq 0$, let us define a \emph{fractional covariant} of
weight~$\det^k\otimes\Sym^n$ to be a rational map $G:\C_6[X]\to \C_n[X]$ (i.e.~a
polynomial in~$X$ whose coefficients are rational fractions in the coefficients
of $f\in \C_6[X]$) such that the following transformation rule holds: for every
$f\in \C_n[X]$ and $r\in \GL_2(\C)$,
\begin{displaymath}
  G\paren[\big]{\det\nolimits^{-2}\otimes\Sym^6(r) f} = \det\nolimits^k\otimes\Sym^n G(f).
\end{displaymath}
With this terminology, the map $G \mapsto G\circ \mf{\chi}_{-2,6}$ is a
bijection between fractional covariants of weight~$\det^k\otimes\Sym^n$ and
Siegel modular functions of weight~$\det^k\otimes\Sym^n$, for any pair~$(k,n)$.

We are specially interested in fractional covariants of weight $\Sym^2$, as
they relate to the derivatives of Igusa invariants. For each~$1\leq k\leq 3$,
we can define a Siegel modular function $\mf{dj}_k$ of weight~$\Sym^2$ by setting
\begin{equation}
  \mf{dj}_k = \frac{1}{2\pi i} \paren[\Big]{
    \frac{\partial \mf{j}_k}{\partial\tau_1}  x^2 + \frac{\partial \mf{j}_k}{\partial \tau_3} x + \frac{\partial \mf{j}_k}{\partial \tau_2}
  }.
\end{equation}
Up to reordering,~$\mf{dj}_k$ is one of the lines of the matrix~$\mf{dJ}$
defined in~\eqref{eq:dJ}.  The fractional covariants corresponding to these
modular forms are explicitly written down
in~\cite[§3]{kiefferComputingIsogeniesModular2025}. The algorithm can now be
stated as follows.

\begin{algo}
  \label{algo:differentials}
  \algoinput{A monic polynomial $f\in \C_6[X]$ encoding the genus~2
    curve~$\Crv_f:y^2 = f(x)$ equipped with the basis of differential forms
    $\omega_f = (x\, dx/y, dx/y)$; a period matrix~$\tau\in \Fund_2^\eps$
    of~$\Crv_f$ for~$\eps = 2^{-10}$; a working precision~$N$. We assume that
    $\mf{dJ}(\tau)$ is invertible.}  \algooutput{The big period
    matrix~$(\Omega_0 \ \Omega_1)$ of~$\Jac(\Crv_f)$ attached to the
    bases~$\phi$ and~$\omega_f$, where~$\phi$ denotes the symplectic basis
    of~$H_1(\Jac(\Crv_f),\Z)$ attached to~$\tau$; we have~$\tau = \Omega_1^{-1}\Omega_0$.}
  \begin{enumerate}
  \item \label{step:differentials-chi} Evaluate $g = \mf{\chi}_{-2,6}(\tau) \in \C_6[X]$.
  \item \label{step:differentials-sym2} Find a matrix~$r\in \GL_2(\C)$ such
    that $\det^{-2}\otimes\Sym^6(r) f = g$ as follows.
    \begin{enumerate}
    \item \label{step:differentials-cov} Let $G_k$ be the covariant attached
      to~$\mf{dj}_k$ for each $1\leq k\leq 3$.
    \item \label{step:differentials-sym2-mat} Compute~$\Sym^2(r)$ as
      $\paren[\big]{G_1(g) \ G_2(g)\ G_3(g)}\cdot \paren[\big]{G_1(f)\ G_2(f)\
        G_3(f)}^{-1}$ (these are~$3\times 3$ matrices.)
    \item \label{step:differentials-choose-a} Find which of the
      entries~$a,b,c,d$ of the $2\times 2$ matrix~$r$ has the largest magnitude
      by looking at the corners of the matrix~$\Sym^2(r)$. Without loss of
      generality, we assume that it is~$a$.
    \item \label{step:differentials-sqrt} Compute~$a$ by extracting an
      arbitrary square root of~$a^2$. Compute~$b$ and~$c$ as $ab/a$ and
      $2ac/2a$ respectively. Finally, compute~$d$ as $((ad + bc) - bc)/a$.
    \end{enumerate}
  \item \label{step:differentials-output} Output $\Omega_0 = 2\pi i r \tau$
    and~$\Omega_1 = 2\pi i r$.
  \end{enumerate}
\end{algo}

Recall the notation~$\Pi(f)$ from \cref{def:Pi}.

\begin{prop}
  \label{prop:differentials}
  \Cref{algo:differentials} is correct and runs in time is~$O(\M(N)\log N)$
  with a precision loss of $O(\Pi(f))$ bits. Its output satisfies
  $\max\{\abs{\Omega_1}, \abs{\Omega_1^{-1}}\} = O(\Pi(f))$.
\end{prop}

\begin{proof}
  First, step~\ref{step:differentials-sym2} correctly computes~$r$: in
  step~\ref{step:differentials-sym2-mat}, we have $\Sym^2(r) G_k(f) = G_k(g)$
  for every $1\leq k\leq 3$, and the following computations are valid
  by~\eqref{eq:sym2}.

  Since~$\det^{-2}\otimes\Sym^6(r) f = g$, there exists an
  isomorphism~$\psi_g: \Crv_f \to \Crv_g$ such that
  \begin{displaymath}
    \begin{pmatrix}
      \psi_g^*(x\,dx/y) \\ \psi_g^*(dx/y)
    \end{pmatrix}
    = r^{-t}
    \begin{pmatrix}
      x\, dx/y \\ dx/y
    \end{pmatrix}.
  \end{displaymath}
  After composing the isomorphism induced on Jacobians with the isomorphism
  given by \cref{thm:covariants}, we obtain an isomorphism
  $\psi_\tau: \Jac(\Crv_f) \to A(\tau)$ such that
  \begin{displaymath}
    \begin{pmatrix}
      \psi_\tau^*(2\pi i\, dz_1) \\ \psi_\tau^*(2\pi i \,dz_2)
    \end{pmatrix}
    = r^{-t}
    \begin{pmatrix}
      x\, dx/y \\ dx/y
    \end{pmatrix}.
  \end{displaymath}
  Since the big period matrix of~$\C^2/(\tau\Z^2 + \Z^2)$ in the
  bases~$\phi = \paren[\big]{\tau\tvec{1}{0}, \tau\tvec{0}{1}, \tvec{1}{0},
    \tvec{0}{1}}$ and~$\omega= (2\pi i\,dz_1, 2\pi i\,dz_2)$
  is~$(\tau\ I_2)$, the big period matrix of~$\Jac(\Crv_f)$ in the bases
  $\psi_\tau^{-1}(\phi)$ and~$\omega_f$ is indeed~$(2\pi i r\tau\ \ 2\pi i
  r)$. Thus \cref{algo:differentials} is correct.

  To begin the complexity analysis, we claim that
  $\log \max\{\abs{r}, \abs{r^{-1}}\} = O(\Pi(f))$. Indeed:
  \begin{itemize}
  \item We have $\logp \abs{G_k(f)} = O(\Pi(f))$ for every $1\leq k\leq 3$,
    because these covariants are polynomials in the coefficients of~$f$ divided
    by a power of the discriminant.
  \item We have $\logp \abs{G_k(g)} = O(\Pi(f))$ for every $1\leq k\leq 3$,
    because these covariants are the lines of the matrix~$\mf{dJ}(\tau)$ and we have
    \cref{lem:size-periods}.
  \item The determinant~$D(g) = \det \paren[\big]{G_1(g) \ G_2(g)\ G_3(g)}$ is
    precisely $\det \mf{dJ}(\tau)$, so it satisfies
    $\logp \abs{D(g)^{-1}} = O(\Pi(f))$ by definition of~$\Pi(f)$.
  \item The determinant of~$r$ satisfies
    $\log\max\{\abs{\det r}, \abs{\det r^{-1}}\} = O(\Pi(f))$. Indeed, we
    have
    $\log\max\{\abs{\mf{h}_{10}(\tau)}, \abs{\mf{h}_{10}(\tau)^{-1}}\} = O(\Pi(f))$
    by \cref{lem:size-periods}. On the other hand, the discriminant~$\Delta_f$ of~$f$
    satisfies $\log\max\{\Delta_f, \Delta_f^{-1}\} = O(\Pi(f))$ given its
    expression in terms of roots. Finally, we use that
    $\mf{h}_{10}(\tau) = \det(r)^{10} \Delta_f$.
  \item Finally, the analogously defined determinant~$D(f)$ satisfies
    $\logp\abs{D(f)^{-1}} = O(\Pi(f))$ as $D(g) = \det(r)^3 D(f)$.
  \end{itemize}
  This shows that
  $\log\max\{\abs{\Sym^2(r)}, \abs{\Sym^2(r)^{-1}}\} = O(\Pi(f))$, hence the claim.

  We now review the complexity of each step. Step~\ref{step:differentials-chi}
  has a running time of $O(\M(N)\log N)$ because~$\tau\in
  \Fund_2^\eps$. Step~\ref{step:differentials-cov} costs~$O(1)$. In
  step~\ref{step:differentials-sym2-mat}, evaluating the covariants can be done
  in time~$O(\M(N))$ with a precision loss of~$O(\Pi(f))$ bits, and the
  same estimates hold for the matrix inversion and
  multiplication. Step~\ref{step:differentials-choose-a} identifies an entry
  (say~$a$) such that $\abs{a^{-1}} = O(\Pi(f))$ in
  time~$O(\Pi(f))$. Then, step~\ref{step:differentials-sqrt} runs in
  time~$O(\M(N))$ and involves a further precision loss of~$\Pi(f)$
  bits. Finally, step~\ref{step:differentials-output} also runs in
  time~$O(\M(N))$ and involves a precision loss of~$O(\Pi(f))$ bits by
  \cref{lem:size-periods}.
\end{proof}

\begin{rem}
  \label{rem:rm-with-automorphisms}
  We could presumably adapt \cref{algo:differentials} to also work for period
  matrices~$\tau\in \Fund_2$ where~$\det(\mf{dJ}(\tau))$ is zero or vanishingly
  small: the only change would be in step~\ref{step:differentials-cov}, where
  one would consider other covariants of weight~$\Sym^2$, or possibly
  covariants of higher weights when~$\Crv_f$ has extra
  automorphisms. Generators for the algebra of covariants are listed
  in~\cite[p.\,296]{clebschTheorieBinaerenAlgebraischen1872}. The question
  would then be whether for any given~$\tau\in \Half_2$, one can find a triple
  of covariants that yields an invertible matrix.
\end{rem}

\subsection{The analytic representation of real endomorphisms}
\label{subsec:analytic-rep}

We now explain how to proceed to compute the analytic representation of real
endomorphisms (step~\ref{step:hilbert-rm} of~\cref{algo:numerical-hilbert})
using the results of the previous subsection. We focus on~$F = \Q(\sqrt{5})$,
but the idea applies to other quadratic fields equally.

We start with the following algorithm, which is based on writing the pullback
of~$\mf{\chi}_{-2,6}$ via the Hilbert embedding in terms of Hilbert modular
forms. We need the following definition, taken
from~\cite[§3]{kiefferComputingIsogeniesModular2025}. Consider a genus 2 curve
$\Crv_f: y^2 = f(x)$ whose Jacobian has~RM by~$\Z_F$, in other words we have an
embedding~$\iota:\Z_F\embed \End(\Jac(\Crv_f))$ such that every element
of~$\iota(\Z_F)$ is invariant under the Rosati involution. We say that the
equation~$f$ is \emph{Hilbert-normalized} if for every~$\lambda\in \Z_F$, the
matrix of~$\iota(\lambda)^*$ on~$H^0(\Jac(\Crv_f), \Omega^1_{\Jac(\Crv_f)})$, written in the
basis $(x \,dx/y, dx/y)$ is
\begin{displaymath}
  \mat{\lambda}{0}{0}{\conj{\lambda}}.
\end{displaymath}
Hilbert-normalized curve equations can be constructed as
follows~\cite[§7]{kiefferComputingIsogeniesModular2025}.

\begin{algo}
  \label{algo:hilbert-construction}
  \algoinput{A pair of Gundlach invariants $(g_1,g_2)$ over~$\C$ for
    $F = \Q(\sqrt{5})$; a working precision~$N$.}  \algooutput{A polynomial
    $f\in \C_6[X]$ with the following properties: $\Jac(\Crv_f)$ has RM
    by~$\Z_F$; the Gundlach invariants of~$\Jac(\Crv_f)$ are $(g_1,g_2)$;
    and~$f$ is Hilbert-normalized for a certain choice of
    embedding~$\iota:\Z_F\embed \End(\Jac(\Crv_f))$.}
  \begin{enumerate}
  \item Pick values~$G_2,F_6,F_{10}$ in~$\C$ that realize the Gundlach
    invariants~$g_1,g_2$. By~\eqref{eq:gundlach}, we can set~$G_2 = 1$,
    $F_{10} = 1/g_1$ and~$F_6 = g_2/g_1$.
  \item Let~$f_3$ be any square root of $4 F_{10} F_6^2$.
  \item Let~$f_1 = 1$ and~$f_5 = \dfrac{36}{25} F_{10} F_6^2 - \dfrac{4}{5} F_{10}^2 G_2$.
  \item Determine the quantities~$f_0,f_6$ using the quadratic relations
    \begin{align*}
      f_0 f_6 &= \dfrac{-4}{25} F_{10} F_6^2 + \dfrac{1}{5} F_{10}^2 G_2, \quad\text{and}\\
      f_3 \bigl(f_0^2 f_5^3 + f_1^3 f_6^2 \bigr) &= 123 F_{10}^3 F_6 -
      \dfrac{32}{25} F_{10}^2 F_6^2 G_2^2 + \dfrac{288}{125} F_{10}
      F_6^4 G_2 - \dfrac{3456}{3125} F_6^6.
    \end{align*}
  \item Output $f = \sum_{i=0}^6 f_i x^i$.
  \end{enumerate}
\end{algo}

We use these normalized curve equations as a key tool to compute the analytic
representation of real endomorphisms. Recall the element~$\xi\in \Z_F$ used
in \cref{algo:numerical-hilbert}.

\begin{algo}
  \label{algo:hilbert-rm}
  \algoinput{A pair of Gundlach invariants $(g_1,g_2)$ over~$\C$
    for~$F = \Q(\sqrt{5})$ corresponding to Igusa invariants~$(j_1,j_2,j_3)$; a
    period matrix~$\tau\in \Fund_2^\eps$ for~$\eps = 2^{-10}$ whose Igusa
    invariants are~$(j_1,j_2,j_3)$; a working precision~$N$.}  \algooutput{The
    analytic representation~$\rho_A(\xi)$ of~$\xi$ attached to the period
    matrix~$\tau$.}

  \begin{enumerate}
  \item \label{step:hilbert-construction} Let~$f$ be the output of
    \cref{algo:hilbert-construction} on the input of~$(g_1,g_2)$ and~$N$.
  \item \label{step:hilbert-monic} Multiply~$f$ by a scalar to make it monic of
    degree~6.
  \item \label{step:hilbert-r} Let~$r\in \GL_2(\C)$ be the matrix such that the
    output of \cref{algo:differentials} on the input of~$f,\tau$ and~$N$
    is~$(\Omega_0, \Omega_1)$ with~$\Omega_1 = 2\pi i r$.
  \item \label{step:hilbert-output} Output
    $\rho_A(\xi) = r \mat{\xi}{0}{0}{\conj{\xi}} r^{-1}$.
  \end{enumerate}
\end{algo}

\begin{prop}
  \label{prop:hilbert-rm}
  \Cref{algo:hilbert-rm} is correct. Moreover, there exist an explicitly
  computable polynomial $Q\in \Z[G_1,G_2]$ such that the following
  holds. For~$(g_1,g_2) \in \C$, write
  $ \Lambda(g_1,g_2) = \log \max\{\abs{g_1}, \abs{g_2}, \abs{Q(g_1,g_2)}^{-1}\}
  $.  Then, given $(g_1,g_2)$ to precision~$N$, the algorithm runs in time
  $O(\M(N)\log N)$ with a precision loss of $O(\Lambda(g_1,g_2))$ bits.
\end{prop}

\begin{proof}
  As in the proof of \cref{prop:differentials}, there exists an isomorphism
  $\psi_\tau:\Jac(\Crv_f) \to A(\tau)$ such that
  \begin{displaymath}
    \begin{pmatrix}
      \psi_\tau^*(2\pi i\, dz_1)\\ \psi_\tau^*(2\pi i\, dz_2)
    \end{pmatrix}
    = r^{-t}
    \begin{pmatrix}
      x \, dx/y \\ dx/y.
    \end{pmatrix}
  \end{displaymath}
  Step~\ref{step:hilbert-construction} of the algorithm ensures that the matrix
  of~$\xi^*$ (after choosing the correct RM embedding) in the basis
  $(x\, dx/y, dx/y)$ is the diagonal matrix with coefficients
  $\xi,\conj{\xi}$. Consequently, considering now~$\xi$ as an endomorphism
  of~$A(\tau)$, the matrix of~$\xi^*$ in the basis
  $(2\pi i\, dz_1, 2\pi i dz_2)$ is
  \begin{displaymath}
    r^{-t} \mat{\xi}{0}{0}{\conj{\xi}} r^{t},
  \end{displaymath}
  and~$\rho_A(\xi)$ is the transpose of this product. Thus
  \cref{algo:hilbert-rm} is correct.

  We now analyze its complexity. In step~\ref{step:hilbert-construction},
  \cref{algo:hilbert-construction} performs $O(1)$ elementary operations on
  complex numbers constructed from~$(g_1,g_2)$. For a good choice of~$Q$ which
  can be made explicit, these complex numbers are all nonzero. Then
  steps~\ref{step:hilbert-construction}--\ref{step:hilbert-monic} run in
  time~$O(\M(N))$ with a precision loss of~$O(\Lambda(g_1,g_2))$ bits. In
  step~\ref{step:hilbert-r}, we claim that $\Pi(f)$ (in the notation of
  \cref{algo:differentials}) is in $O(\Lambda(g_1,g_2))$ for a good choice
  of~$Q$: indeed, the polynomial~$Q$ only has to include the equation of the
  vanishing locus of the modular function~$\det(\mf{dJ})$ in terms of the Gundlach
  invariants. Therefore, steps~\ref{step:hilbert-r}--\ref{step:hilbert-output}
  run in time~$O(\M(N)\log N)$ with a precision loss of $O(\Lambda(g_1,g_2))$
  bits by \cref{prop:differentials}.
\end{proof}

\section{Complexity analysis of the numerical evaluation algorithms}
\label{sec:precision}

This section contains the analysis of complexity and precision losses in
\cref{algo:numerical-siegel,algo:numerical-hilbert}, building on the results
of~§\ref{sec:agm}. The remaining steps to study are Mestre's algorithm,
computing Weierstrass points, and computing the modular equations from the
individual values of modular forms in step~\ref{step:siegel-prodtree}
(resp.~step~\ref{step:hilbert-prodtree}) of these algorithms.

\subsection{Mestre's algorithm}
\label{subsec:roots}

Here, our complexity analysis relies on the fact that the input values of Igusa
invariants $j_1,j_2,j_3$ in~$\C$ arise from a number field. Recall the
definition of the (absolute logarithmic) height of (tuples of) number field
elements~\cite[§B.2]{hindryDiophantineGeometry2000}, denoted by~$h$, and the
notation~$\Pi(f)$ from \cref{def:Pi}.

\begin{thm}
  \label{thm:analytic-mestre}
  Let~$L$ be a number field of degree~$d_L$ over~$\Q$ and let
  $j_1,j_2,j_3\in L$ be such that the modular function~$\det(\mf{dJ})$ does not
  vanish at $(j_1,j_2,j_3)$. Let~$H = h(j_1,j_2,j_3)$. Then, for every complex
  embedding~$\mu$ of~$L$, on the input of $\mu(j_1),\mu(j_2),\mu(j_3)$ as
  complex numbers to precision~$N_\mu\in \Z_{\geq 0}$, one can compute a
  degree~6 polynomial $f_\mu\in \C[X]$ with simple roots such that the
  genus~$2$ curve $\Crv_{f_\mu}: y^2 = f_\mu(x)$ has Igusa invariants
  $(\mu(j_1),\mu(j_2),\mu(j_3))$.  Let~$B_\mu$ be the precision loss in this
  algorithm. Then we further have
  \begin{displaymath}
    \textstyle \sum_{\mu} \Pi(f_\mu) = O(d_LH) \quad\text{and}\quad \sum_\mu B_\mu = O(d_LH).
  \end{displaymath}
  The cost of this computation for each~$\mu$ is $O(\M(N_\mu))$ binary
  operations.
\end{thm}

Note that the hypothesis $\det(\mf{dJ})\neq 0$ guarantees that the genus~2
curves~$\Crv_{f_\mu}$ do not have extra automorphisms.

\begin{proof}
  When running Mestre's algorithm~\cite{mestreConstructionCourbesGenre1991}
  separately in each embedding~$\mu$, we are in fact in the following
  situation. There exist
  \begin{itemize}
  \item a quadratic extension $L'/L$,
  \item a degree~6 polynomial $f\in L'[X]$ such that the curve $y^2 = f(x)$ has
    Igusa invariants $(j_1,j_2,j_3)\in L^3$; and
  \item for each complex embedding~$\mu$ of~$L$, an extension~$\mu'$ of~$\mu$
    to~$L'$,
  \end{itemize}
  such that in each embedding~$\mu$ as above, Mestre's algorithm
  returns the polynomial $\mu'(f)$.

  If we were to run Mestre's algorithm in~$L'$, we would perform $O(1)$
  elementary operations (sums, products, and square roots) of elements in~$L'$
  of height $O(H)$. If~$x\in L'$ denotes any of the $O(1)$
  nonzero elements of~$L'$ encountered during Mestre's algorithm, we would then
  have by the definition of heights:
  \begin{displaymath}
    \textstyle \sum_{\mu} \log \max\{\abs{\mu'(x)}, \abs{\mu'(x)^{-1}}\} = O(d_LH).
  \end{displaymath}
  We deduce that $\sum_\mu \Pi(f_\mu) = O(d_LH)$ and $\sum_\mu B_\mu = O(d_LH)$ by
  \cref{lem:elementary}. The algorithm runs in time~$O(\M(N_\mu))$ by the same
  lemma.
\end{proof}

Precision losses of a similar shape to those in \cref{thm:analytic-mestre} will
repeatedly appear in the rest of the paper. Therefore, we adopt:

\begin{conv}
  \label{conv:Bmu}
  The notation $B_\mu$, where~$\mu$ runs through the complex embeddings of~$L$,
  always refers to a collection of nonnegative integers such that
  $\sum_\mu B_\mu = O(d_LH)$. However, the precise values of these
  integers~$B_\mu$ may vary from one occurrence to the next.
\end{conv}

In order to apply \cref{thm:agm-cost}, we need to compute the six complex roots
of the polynomials output by Mestre's algorithm. Specifying an efficient
root-finding algorithm in detail is unfortunately beyond the scope of this
paper, so we rely on the following black box.

\begin{thm}
  \label{thm:complex-roots}
  Let~$f\in \C[X]$ be a monic polynomial of degree~$n$ with simple roots, and
  let~$\delta > 0$ be a lower bound on the distance between any two distinct
  roots of~$f$. Then, given an approximation of~$f$ to precision~$N$, one can
  compute approximations of each complex roots of~$f$ with a precision loss of
  $\Otilde(n^2 + n\logp\abs{f} + \logp \delta^{-1})$ bits in
  time~$\Otilde(n N)$.
\end{thm}

\begin{proof}[Proof (sketch)]
  One can isolate the roots of~$f$ in time
  $\Otilde(n^3+n^2\logp\abs{f} + n\logp \delta^{-1})$ and with the required
  precision loss using the algorithm
  of~\cite{beckerNearoptimalSubdivisionAlgorithm2018}. After that, one can set
  up a Newton scheme around a multipoint evaluation algorithm (see
  e.g.~\cite[§10.1]{vonzurgathenModernComputerAlgebra2013} in an exact setting;
  the precision losses in such an algorithm can also be controlled in terms
  of~$n$ and~$\delta$) to refine the roots of~$f$ to precision~$N$ in
  time~$\Otilde(nN)$.
\end{proof}

\begin{cor}
  \label{cor:periods} Keep the hypotheses of \cref{thm:analytic-mestre}. Then,
  for every complex embedding~$\mu$ of~$L$, on the input
  of~$\mu(j_1),\mu(j_2),\mu(j_3)$ to precision~$N_\mu\in \Z_{\geq 0}$, one can
  compute~$\tau\in \Fund_2^\eps$ where $\eps = 2^{-10}$ whose Igusa invariants
  are $\mu(j_1),\mu(j_2),\mu(j_3)$ and such that $\abs{\tau} = B_\mu$, in time
  $O(\M(N)\log N)$ with a precision loss of~$B_\mu\log B_\mu + O(\log N)$ bits.
\end{cor}

\begin{proof}
  Combine \cref{thm:analytic-mestre},
  \cref{thm:complex-roots} with~$n=6$, and \cref{thm:agm-cost}.
\end{proof}

\subsection{Product trees}

In order to make products of polynomials of degree~1 as in
step~\ref{step:siegel-prodtree} of \cref{algo:numerical-siegel}, it is standard to
use product trees as in the following algorithm.

\begin{algo}
  \label{algo:prodtree}
  \algoinput{Complex numbers $x_k, y_k, z_k$ for $0\leq k \leq d-1$; a working
    precision~$N$.}  \algooutput{The polynomials
    $P = \prod_{k=0}^{d-1} (x_k X + y_k)$ and
    $Q = \sum_{k=0}^{d-1} z_k \prod_{j\neq k}(x_j X + y_j)$.}

  \begin{enumerate}
  \item Let~$n = \ceil{\log_2(d)}$. For~$d\leq k < 2^n$, we
    set~$x_k = 0, y_k = 1$ and~$z_k = 0$.
  \item Construct two lists $L_{P,0}$ and~$L_{Q,0}$ of length~$2^n$ such
    that~$L_{P,0}[k] = x_k X + y_k$ and $L_{Q,0}[k] = z_k$ for each~$k$.
  \item \label{step:prodtree-rec} For each~$1\leq m\leq n$, construct two lists~$L_{P,m}$ and~$L_{Q,m}$
    of length~$2^{n-m}$ as follows: for each $0\leq k < 2^{n-m}$, set
    \begin{align*}
      L_{P,m}[k] &= L_{P,m-1}[2k] \cdot L_{P,m-1}[2k+1],\\
      L_{Q,m}[k] &= L_{P,m-1}[2k+1] \cdot L_{Q,m-1}[2k] + L_{P,m-1}[2k] \cdot L_{Q,m-1}[2k+1].
    \end{align*}
  \item Output $P = L_{P,n}[0]$ and~$Q = L_{Q,n}[0]$.
  \end{enumerate}
\end{algo}

Of course, we can adapt \cref{algo:prodtree} to compute several polynomials of
type~$Q$.

\begin{prop}
  \label{prop:prodtree}
  \Cref{algo:prodtree} is correct. If~$B\leq 1$ and~$C\leq 1$ are real numbers such that
  \begin{displaymath}
    \logp\abs{x_k} \leq B, \quad\logp\abs{y_k}\leq B\quad \text{and} \quad \logp\abs{z_k}\leq C
    \quad \text{ for all } k,
  \end{displaymath}
  and the complex numbers~$x_k,y_k,z_k$ are given to precision~$N$, then the
  algorithm runs in time $\Otilde(dN+dC+d^2B)$ with a precision loss
  of $O(C+dB)$ bits.
\end{prop}

\begin{proof}
  If we let~$\mathcal{S}_{m,k} = \{2^m k, 2^m k + 1, \ldots, 2^mk+2^m-1\}$, then we have
  by induction
  \begin{displaymath}
    L_{P,m}[k] = \prod_{j \in \mathcal{S}_{m,k}} (x_j X + y_j)
    \quad \text{and} \quad
    L_{Q,m}[k] = \sum_{j \in \mathcal{S}_{m,k}} z_j \prod_{l\in \mathcal{S}_{m,k}\setminus\{j\}} (x_l X + y_l).
  \end{displaymath}
  In particular, $L_{P,m}[k]$ has degree at most~$2^m$ and satisfies
  $\logp\abs[\big]{L_{P,m}[k]} = O(2^m B)$. Similarly, $L_{Q,m}[k]$ has degree at most~$2^m - 1$ and satisfies
  $\logp\abs[\big]{L_{Q,m}[k]} = O(C + 2^m B)$.

  In step~\ref{step:prodtree-rec} for each~$1\leq m\leq n$, we compute one
  product between elements of~$L_{P,m-1}$. This
  costs~$\Otilde(2^m (N + 2^m B))$ binary operations and implies a precision
  loss of~$O(2^m B)$ bits by \cref{lem:elementary}. Simiarly, the products
  between elements of~$L_{P,m-1}$ and~$L_{Q,m-1}$ can be computed to precision
  $N - O(C + 2^m B)$ using a total of $\Otilde(2^mN+2^mC+2^{2m}B)$ binary
  operations. Therefore the total cost of this step for any each~$m$ is
  $\Otilde(dN + dC + d^2B)$. Summing over~$m$, the total running time of the
  algorithm is $\Otilde(dN+dC+d^2B)$.
\end{proof}

\subsection{Cost of numerical evaluation}

For the sake of brevity, we focus on the more involved Hilbert case
(\cref{algo:numerical-hilbert}), and simply point out the differences in the
Siegel case. We fix the quadratic field $F = \Q(\sqrt{5})$, and a totally
positive $\beta\in \Z_F$ of prime norm $\ell\in \Z$.  Recall
that~$P_{\beta,k}$ and~$Q_\beta$ denote the numerators and denominator of
Hilbert modular equations~$\Psi_{\beta,k}$ for~$1\leq k\leq 2$ as constructed
in~§\ref{subsec:hilbert-denom}, and recall \cref{conv:Bmu} for the notation $B_\mu$.

\begin{prop}
  \label{prop:cost-hilbert}
  There exists an explicitly computable polynomial $Q\in \Z[G_1,G_2]$ such that
  the following holds.  Let~$L$ be a number field of degree~$d_L$,
  and choose~$g_1,g_2\in L$ such that $Q(g_1,g_2) \neq 0$. Let~$H=
  h(g_1,g_2)$. Then for each complex embedding~$\mu$ of~$L$, given~$\mu(g_1)$
  and~$\mu(g_2)$ to precision~$N$, \cref{algo:numerical-hilbert} computes
  $Q_\beta(\mu(g_1),\mu(g_2))\in \C$ and $P_{\beta,k}(\mu(g_1),\mu(g_2), X)\in \C[X]$ for
  $1\leq k\leq 2$ in $\Otilde(\ell N)$ binary operations with a precision loss of
  $B_\mu\log B_\mu + O(\log N + \ell B_\mu + \ell \log \ell)$ bits.
\end{prop}

\begin{proof} The complexity and precision losses in each step of
  \cref{algo:numerical-hilbert} are as follows.
  \begin{enumerate}
  \item This step costs $\Otilde(N)$ binary operations and involves a precision
    loss of $O(B_\mu)$ bits, as in the proof of \cref{thm:analytic-mestre},
    because $(j_1,j_2,j_3)$ are elements of~$L$ of height $O(H)$.
  \item \label{step:proof-tau} We may assume $\det(\mf{dJ})$ does not vanish at
    $(j_1,j_2,j_3)$ if~$Q$ is well chosen. By \cref{cor:periods}, this step
    costs $\Otilde(N)$ binary operations with a precision loss of
    $B_\mu\log B_\mu + O(\log N)$ bits, and outputs $\tau\in \Fund_2^\eps$ such
    that $\abs{\tau}\leq B_\mu$.
  \item We may assume that~$Q$ is a multiple of the polynomial~$Q$ of
    \cref{prop:hilbert-rm}. We then have
    $\sum_\mu\Lambda(\mu(g_1),\mu(g_2)) = O(dH)$, so \cref{algo:hilbert-rm}
    costs $\Otilde(N)$ binary operations with a precision loss of $B_\mu$ bits.
  \item This steps involves only elementary operations. Given the upper bound
    on $\abs{\tau}$ from step~\ref{step:proof-tau} and \cref{lem:elementary},
    this steps costs $\Otilde(N)$ binary operations, involves a
    precision loss of $O(B_\mu)$ bits, and outputs integers $a,b,c,d,e$ of
    absolute value at most~$B_\mu$.
  \item By \cref{prop:birkenhake}, this step is quasi-linear in
    $\log\max\{\abs{a},\abs{b},\abs{c},\abs{d},\abs{e}\}$.
  \item \label{step:proof-loop}
    \begin{enumerate}
    \item Note that $\log\abs{\mathrm{Hilb}_F(\gamma)\eta} = \log\ell + B_\mu$. By
      \cref{prop:change-rep}, the cost of this step is quasi-linear in
      $\log\ell + B_\mu$.
    \item \label{step:proof-red} Since the lower left block of
      $\zeta_\gamma \mathrm{Hilb}_F(\gamma)\eta$ is zero, its lower right block is a
      $2\times 2$ integral matrix of determinant~$\ell$. Therefore, the matrix
      $\tau' = \zeta_\gamma \mathrm{Hilb}_F(\gamma)\eta\tau$ satisfies
      $\Xi(\tau') = O(\log\ell)$ and $\Lambda(\tau') = B_\mu + O(\log \ell)$,
      in the notation of~§\ref{subsec:fund}. By \cref{thm:reduction}, computing
      $\eta_\gamma$ costs $\Otilde(N\log\ell)$ binary operations with a
      precision loss of $B_\mu + O(\log\ell)$.
    \item By~\cite{kiefferCertifiedEwtonSchemes2022}, evaluating theta
      constants costs $\Otilde(N)$ binary operations with a precision loss of
      $O(1)$ bits. (Recall that in~§\ref{subsec:numerical-siegel}, we
      postulated that this algorithm also applies on~$\Fund_2^\eps$ for
      $\eps=2^{-10}$.) The same estimates hold for the evaluation
      of~$\mf{h}_4,\ldots,\mf{h}_{12}$ from theta constants. Given the estimates on
      $\Lambda(\tau')$ and $\abs{\eta_\gamma\zeta_\gamma}$ obtained in
      step~\ref{step:proof-red}, applying the transformation law for the matrix
      $(\eta_\gamma\zeta_\gamma)^{-1}\in \Sp_4(\Z)$ costs
      $\Otilde(N)$ binary operations and a precision loss
      of $B_\mu + O(\log \ell)$ bits. The final quantities $\mf{h}_4,\ldots,\mf{h}_{12}$
      we obtain after the transformation satisfy $\abs{\mf{h}_k} = B_\mu + O(\log\ell)$.
    \item This step costs $\Otilde(N)$ binary operations with a precision loss
      of $B_\mu + \log\ell$ bits by the previous estimates.
    \end{enumerate}
    \noindent
    Since the number of matrices~$\gamma\in C_\beta^{\mathrm{sym}}$ considered in this
    step is~$O(\ell)$, the total cost of step~\ref{step:proof-loop} is
    $\Otilde(\ell N)$ binary operations with a precision loss of
    $B_\mu + O(\log\ell)$ bits.
  \item As in step~\ref{step:proof-loop}, this step costs $\Otilde(B_\mu + N)$
    binary operations and involves a precision loss of $B_\mu$ bits.
  \item We apply \cref{prop:prodtree} with $B = C = B_\mu + O(\log\ell)$, and
    $d = O(\ell)$, so this steps runs in time $\Otilde(\ell N)$ with a precision
    loss of $\ell B_\mu + O(\ell\log\ell)$ bits.
  \item Using fast exponentiation, this step costs $\Otilde(N \log\ell)$ binary
    operations with a precision loss of $\ell B_\mu$ bits by
    \cref{lem:elementary}. We have
    $\log\max\{\abs{\lambda},\abs{\lambda}^{-1}\} = O(\ell B_\mu)$.
  \item Since we multiply $O(\ell)$ complex numbers by~$\lambda$, the total
    cost of this step is $\Otilde(\ell N)$ binary operations with a precision
    loss of $O(\ell B_\mu)$ bits.
  \end{enumerate}
  Summing the contributions of each step concludes the proof.
\end{proof}

In the Siegel case, we essentially omit
steps~\ref{step:hilbert-j},~\ref{step:hilbert-rm},~\ref{step:hilbert-abcde}
and~\ref{step:hilbert-eta} from \cref{algo:numerical-hilbert}, and we only
require that the modular function $\det(\mf{dJ})$ does not vanish at
$(j_1,j_2,j_3)$. On the other hand, we loop over more matrices~$\gamma$, as
$\# C_\ell = O(\ell^3)$. As a result, we obtain:

\begin{prop}
  \label{prop:cost-siegel}
  Let~$L$ be a number field of degree~$d_L$, and choose~$j_1,j_2,j_3\in L$ such
  that $\det(\mf{dJ})$ does not vanish at $(j_1,j_2,j_3)$. Then for each
  complex embedding~$\mu$ of~$L$, given $\mu(j_1),\mu(j_2),\mu(j_3)$ to
  precision~$N$, \cref{algo:numerical-siegel} computes
  $Q_\ell(\mu(j_1),\mu(j_2),\mu(j_3)) \in \C$ and
  $P_{\ell,k}(\mu(j_1),\mu(j_2),\mu(j_3), X)\in \C[X]$ for $1\leq k\leq 3$ in
  $\Otilde(\ell^3 N)$ binary operations with a precision loss of
  $B_\mu\log B_\mu + O(\log N + \ell^3 B_\mu + \ell^3\log\ell)$ bits.
\end{prop}

We further state, but do not prove, the following result on the numerical
evaluation of derivatives of modular equations following the approach
of~§\ref{subsec:numerical-siegel} and~§\ref{subsec:numerical-hilbert}.

\begin{prop}
  \label{prop:numerical-derivatives}
  \begin{enumerate}
\item In the setting of \cref{prop:cost-hilbert}, one can further compute
  \begin{displaymath}
    Q_\beta(\mu(g_1),\mu(g_2))^2 \,\frac{\partial \Psi_{\beta,k}}{\partial G_n}(\mu(g_1),\mu(g_2),X) \in \C[X]
  \end{displaymath}
  for $1\leq k, n\leq 2$ with the same time complexity and precision losses (up
  to increasing the implied constants in the~$O$ estimates.)
\item In the setting of \cref{prop:cost-siegel}, one can further compute
  \begin{displaymath}
    Q_\ell(\mu(j_1),\mu(j_2),\mu(j_3))^2 \,\frac{\partial \Psi_{\ell,k}}{\partial J_n}(\mu(j_1),\mu(j_2),\mu(j_3),X) \in \C[X]
  \end{displaymath}
  for $1\leq k, n\leq 3$ with the same time complexity and precision losses.
  \end{enumerate}
\end{prop}

\section{Complexity analysis of the main algorithms}
\label{sec:proofs}

The complexity analysis of the numerical evaluation algorithms
in~§\ref{sec:precision} was agnostic of how the number field~$L$ is presented:
we only used that the Igusa invariants $(j_1,j_2,j_3)$ were algebraic numbers
of bounded height and degree over~$\Q$. However, in the main
algorithm~\ref{algo:skeleton}, actual number field elements are manipulated
twice: when evaluating $\mu(j_1),\mu(j_2),\mu(j_3)$ for all embeddings~$\mu$ in
step~\ref{step:skeleton-embed}, and when recovering the coefficients of modular
equations as number field elements in step~\ref{step:skeleton-reconstruct}. As
explained in~§\ref{subsec:io}, we consider two possible settings:
\begin{enumerate}
\item \label{case:proof-ff} In the first description $L=\Q(\alpha)$ where~$\alpha$ is a root of a
  certain monic polynomial $P_L\in \Z[X]$, and the input invariants belong to
  $\Z[\alpha]$. Elements in this order as represented as polynomials
  in~$\alpha$ of degree at most~$d_L-1$.
\item \label{case:proof-nf} In the second description, we assume that a
  $\Z$-basis $(\alpha_1,\ldots,\alpha_{d_L})$ of~$\Z_L$ is known, and the input
  invariants are presented as quotients of elements in~$\Z_L$. Elements
  of~$\Z_L$ are represented as linear combinations of $\alpha_1,\ldots,\alpha_{d_L}$.
\end{enumerate}

In case~\ref{case:proof-ff}, we obtain complexity estimates in terms of
$\abs{P_L}$ and the largest coefficients of $j_1,j_2,j_3$ seen as elements
in~$\Z[\alpha]$. We then specialize the results to the case of lifts from
finite fields, thereby yielding \cref{thm:intro} and its Hilbert analogue.
While case~\ref{case:proof-nf} is more general and perhaps more natural, more
information is needed to obtain meaningful complexity estimates: we need to
guarantee that the basis $(\alpha_1,\ldots,\alpha_{d_L})$ is reasonably reduced
and that we can evaluate their complex embeddings.

This section is structured as follows. First, we detail the additional
assumptions we make in case~~\ref{case:proof-nf} (§\ref{subsec:setting-2}). We
then focus on the complexity of embedding number field elements in~$\C$, and
conversely recognizing them from complex approximations, in each of the two
cases (§\ref{subsec:embed-1} and~§\ref{subsec:embed-2}). Finally, we state and
prove our main theorems over both finite fields and number fields
(§\ref{subsec:proof-ff} and~§\ref{subsec:proof-nf}).

\subsection{Further assumptions in case~2}
\label{subsec:setting-2}

We endow~$\Z_L$ with the euclidean metric induced by the map~$\Z_L\to \C^{d_L}$
given by the~$d_L$ complex embeddings~$\mu_1,\ldots,\mu_{d_L}$ of~$L$. Then~$\Z_L$
becomes a lattice of volume~$\Delta_L$ in the Euclidean
space~$\Z_L\otimes_\Z\R$. Denote by~$1\leq \nu_1 \leq \cdots\leq \nu_{d_L}$
the successive minima of~$\Z_L$. They satisfy the following
inequality~\cite[Chap.~2, Thm.~5]{nguyenLLLAlgorithmSurvey2009}:
\begin{equation}
  \label{eq:lattice-minima}
  \prod_{k=1}^{d_L} \nu_k \leq \kappa_{d_L}^{d_L/2} \Delta_L,
\end{equation}
where~$\kappa_{d_L}\leq 1+\frac {d_L}{4}$ denotes Hermite's constant~\cite[Chap.~2,
Cor.~3]{nguyenLLLAlgorithmSurvey2009}. We will always assume that
$(\alpha_1,\ldots,\alpha_{d_L})$ is HKZ-reduced, which by \cite[Chap.~2,
Thm.~6]{nguyenLLLAlgorithmSurvey2009} implies in particular that for each
$1\leq k\leq d$, we have
\begin{equation}
  \label{eq:hkz}
  \frac{4}{k+3}\leq \paren[\Big]{\frac{\norm{\alpha_k}}{\nu_k}}^2\leq \frac{k+3}{4}.
\end{equation}
Computing HKZ-reduced bases is exponential in~$d_L$ in general, but this
assumption is perhaps justified by the fact that computing an integral basis
of~$\Z_L$ is already an expensive operation which requires factoring the
discriminant~$\Delta_L$ of~$L$. One could alternatively work with an
LLL-reduced basis at the cost of slightly larger complexity estimates.

We now consider the matrix
\begin{displaymath}
  M_L=\paren[\big]{\mu_i(\alpha_j)}_{1\leq i,j\leq d_L}.
\end{displaymath}

\begin{lem}
  \label{lem:ML}
  If the basis $(\alpha_1,\ldots,\alpha_d)$ of~$\Z_L$ is HKZ-reduced, then
  $\logp\abs{M_L}$ and $\logp\abs{M_L^{-1}}$ are in  $O(d_L\log d_L + \log \Delta_L)$.
\end{lem}

\begin{proof}
  By~\eqref{eq:lattice-minima} and~\eqref{eq:hkz}, we have for each
  $1\leq j\leq d$:
  \begin{displaymath}
    \logp\norm{\alpha_j} = O(d_L\log d_L + \log \Delta_L),
  \end{displaymath}
  so $\logp\abs{M_L} = O(d_L\log d_L + \log\Delta_L)$. Note that
  $\abs{\det M_L} = \abs{\Delta_L}\geq 1$.  Each coefficient
  of~$(\det M_L) M_L^{-1}$ is the determinant of a submatrix of~$M_L$, so
  Hadamard's inequality yields
  \begin{displaymath}
    \abs{M_L^{-1}} \leq \prod_{k=1}^{d_L} \norm{\alpha_k} \leq (d_L\kappa_{d_L})^{d_L/2} \Delta_L,
  \end{displaymath}
  and hence $\logp\abs{M_L}^{-1} = O(d_L\log d_L + \log \Delta_L)$.
\end{proof}

We then make the following assumption.

\begin{assumption}
  \label{assumption}
  On the input of a precision~$N\leq 1$, one can
  compute $M_L$ and~$M_L^{-1}$ to precision~$N$ in time
  $\Otilde\paren[\big]{d_L^3 (N + d_L + \log \Delta_L)}$.
\end{assumption}

This complexity is almost quasi-linear in the total bit size of these matrices
by \cref{lem:ML} except that we allow matrix operations to have cubic
complexity in~$d_L$. This is a reasonable assumption, as one can simply
precompute this matrix to a precision~$N_0$ that is sufficient for Newton
iterations to start converging.

\subsection{Embedding and recognizing number field elements: case~1}
\label{subsec:embed-1}

In this case, embedding and reconstructing number field elements relies on
computing the complex roots of~$P_L$, denoted by~$\mu_k(\alpha)$ for
$1\leq k\leq d_L$. (This numbers the complex embeddings of~$L$.)
Since~$P_L\in \Z[X]$, as a particular case of \cref{thm:complex-roots}, one can
compute these roots to precision~$N$ in time $\Otilde(d_L N)$ with a precision
loss of $\Otilde(d_L^2 + d_L\log \abs{P_L})$ bits.

\begin{lem}
  \label{lem:height-alpha}
  For every~$1\leq i\leq d$, we have $\abs{\alpha_i}\leq 2\abs{P_L}$. In
  particular $h(\alpha)\leq \log(2\abs{P_L})$.
\end{lem}

\begin{proof}
  This is classical.  Write $P_L = \sum_{j=0}^{d_L} p_j X^j$, and assume that
  $\abs{\mu_k(\alpha)}> 2\abs{P_L}$ for some~$k$. In particular
  $\abs{\mu_k(\alpha)}>2$. Since $p_{d_L}=1$, we get the contradiction
  \begin{displaymath}
    \abs{\mu_k(\alpha)}^{d_L} \leq \sum_{j=0}^{d_L-1} \abs{p_j}\abs{\mu_k(\alpha)}^j \leq
    \abs{P_L}\frac{\abs{\mu_k(\alpha)}^{d_L} - 1}{\abs{\mu_k(\alpha)}-1}
    \leq \abs{P_L}\cdot 2\abs{\mu_k(\alpha)}^{d_L-1}. \qedhere
  \end{displaymath}
\end{proof}

\begin{lem}
  \label{lem:Pprime}
  For each $1\leq k\leq d_L$, we have
  $\logp\abs{P_L'(\mu_k(\alpha))}^{-1} = O(d_L\log\abs{P_L} + d_L\log d_L)$. We can
  evaluate $P_L'(\mu_k(\alpha))^{-1}$ for every $1\leq k\leq d_L$ to precision~$N$ in
  time $\Otilde(d_L^2\log\abs{P_L})$ with a precision loss of
  $O(d_L \log\abs{P_L} + d_L\log d_L)$ bits.
\end{lem}

\begin{proof}[Proof (sketch)]
  We have $\log\abs{P_L'}\leq \log\abs{P_L}+\log d_L$. The discriminant~$\Delta_{P_L}\in \Z$
  of~$P_L$ is the resultant of~$P_L$ and~$P_L'$. Hence we can write
  \begin{displaymath}
    UP_L + VP'_L = \Delta_{P_L}
  \end{displaymath}
  with $U,V\in\Z[X]$; the coefficients of~$U,V$ have expressions as
  determinants of size~$O(d_L)$ involving the coefficients of~$P_L$
  and~$P_L'$. By Hadamard's inequality, we have
  \begin{displaymath}
    \log \abs{V} = O(d_L\log\abs{P_L} + d_L\log d_L).
  \end{displaymath}
  Moreover, $V$ and $\Delta_{P_L}$ can be computed in time
  $\Otilde(d_L^2\log\abs{P_L})$
  \cite[§11.2]{vonzurgathenModernComputerAlgebra2013}.

  For every $1\leq k\leq d_L$, we have
  $1/P'(\mu_k(\alpha)) = V(\mu_k(\alpha))/\Delta_P$, so it is enough to
  evaluate $V(\mu_k(\alpha))$ for every $1\leq i\leq d$. This can be done with
  a fast multipoint evaluation algorithm: one considers dyadic numbers
  sufficiently close to the roots $\mu_k(\alpha)$ for $1\leq k\leq d_L$, runs
  the multipoint evaluation algorithm on this exact data, then propagates error
  bounds using \cref{lem:height-alpha}. We claim that this computation runs in
  time $\Otilde(d_L^2\log\abs{P_L})$ with a precision loss of
  $O(d_L\log\abs{P_L} + d_L\log d_L)$ bits.
\end{proof}

\begin{lem}
  \label{lem:embed-1}
  Let~$x\in \Z[\alpha]$ be presented as a polynomial $P_x$ of degree at most
  $d_L-1$.  Given the roots $\mu_k(\alpha)$ for $1\leq k\leq d_L$ of~$P_L$ to
  precision~$N$, one can evaluate $\mu_k(x)$ for all~$k$ in time
  $\Otilde(d_L N)$ with a precision loss
  of~$\Otilde(\logp\abs{P_x} + d_L\log\abs{P_L})$ bits.
\end{lem}

\begin{proof}[Proof (sketch)]
  Our task is to evaluate $P_x$ at~$\mu_k(\alpha)$ for $1\leq k\leq d_L$. This
  can be done with the stated complexity and precision losses using a fast
  multipoint evaluation algorithm as in \cref{lem:Pprime}.
\end{proof}

Conversely, computing an element~$P_x\in \Z[\alpha]$ from its evaluations at
the complex roots of~$P_L$ is a Lagrange interpolation problem which we solve
as follows.

\begin{lem}
  \label{lem:recognize-1}
  Let~$x\in \Z[\alpha]$ be represented as a polynomial $P_x$ of degree at most
  $d_L-1$. Given the values $\mu_k(\alpha), 1/P_L'(\mu_k(\alpha))$ as well as
  $\mu_k(x)$ for $1\leq k\leq d_L$ to precision~$N$, one
  can recover~$P_x\in \C[X]$ in time
  $\Otilde(d_L N + d_L\log\abs{P_x} +d_L^2\log\abs{P_L})$ with a precision loss
  of $O(\logp\abs{P_x} +d_L\log\abs{P_L} + d_L\log d_L)$ bits. In particular,
  one can recover $P_x\in \Z[X]$ as soon as
  $N\geq K(\logp\abs{P_x} +d_L\log\abs{P_L} +d_L\log d_L)$ for some absolute
  constant~$K$.
\end{lem}

\begin{proof}
  For every~$1\leq k\leq d_L$, we have by \cref{lem:height-alpha}
  \begin{displaymath}
    \abs{\mu_k(x)} = \abs{P_x(\mu_k(\alpha))} \leq \sum_{j=0}^d \abs{P_x} \abs{\mu_k(\alpha)}^j
    \leq \abs{P_x} \cdot 4\abs{P_L}^{d_L}.
  \end{displaymath}
  The polynomial~$P_x$ interpolates the points
  $(\mu_k(\alpha), \mu_k(x))$ for $1\leq k\leq d_L$, and thus
  \begin{displaymath}
    P_x = \sum_{k=1}^{d_L} \frac{\mu_k(x)}{P'(\mu_k(\alpha))} \prod_{j\neq k} (X - \mu_j(\alpha)).
  \end{displaymath}
  We apply \cref{prop:prodtree} with $B = O(\log \abs{P_L})$ and
  $C = O(d_L \log \abs{P_L} + \logp\abs{P_x} + d_L\log d_L)$ by
  \cref{lem:height-alpha,lem:Pprime}, and the result follows.
\end{proof}

\subsection{Embedding and recognizing integers in number fields: case~2}
\label{subsec:embed-2}

Recall that in~§\ref{subsec:setting-2}, we assumed that the basis
$(\alpha_1,\ldots,\alpha_{d_L})$ of~$\Z_L$ is HKZ-reduced, and than we have
access to an oracle which computes the matrices~$M_L$ and~$M_L^{-1}$
(\cref{assumption}). For every~$x = \sum_{k=1}^{d_L} x_k \alpha_k \in \Z_L$, we
have by definition of~$M_L$
\begin{equation}
  \label{eq:linear-embedding}
  \left(
    \begin{matrix}
      \mu_1(x)\\\vdots\\\mu_{d_L}(x)
    \end{matrix}
  \right)
  = M_L
  \left(
    \begin{matrix}
      x_1\\\vdots\\x_{d_L}
    \end{matrix}
  \right) \quad \text{and consequently}\quad
  \left(
    \begin{matrix}
      x_1\\\vdots\\x_{d_L}
    \end{matrix}
  \right)
  = M_L^{-1}
  \left(
    \begin{matrix}
      \mu_1(x)\\\vdots\\\mu_{d_L}(x)
    \end{matrix}
  \right).
\end{equation}
By \cref{lem:ML}, we have
$\logp\abs{x_k} = O(d_Lh(x) + d_L \log d_L + \log \Delta_L)$ for all
$1\leq k\leq d_L$.

\begin{lem}
  \label{lem:embed-2}
  Let~$x = \sum_{k=1}^{d_L} x_k\alpha_k\in \Z_L$. Then, given the coefficients
  $x_k$ for $1\leq k\leq d_L$ and given the matrix $M_L$ to precision~$N$, one
  can compute $\mu_k(x)$ for every $1\leq k\leq d_L$ in time $\Otilde(d_L^2 N)$
  with a precision loss of $O(d_L h(x) + d_L \log d_L + \log \Delta_L)$ bits.
\end{lem}

\begin{proof}
  We perform the matrix-vector product of $M_L$ by $(x_1,\ldots, x_{d_L})$ by
  making $d_L^2$ products, then summing. The complexity estimates directly
  follow from \cref{lem:elementary} and the above bounds on $\abs{M_L}$ and
  $\logp\abs{x_k}$ for $1\leq k\leq d_L$.
\end{proof}

\begin{lem}
  \label{lem:recognize-2}
  There exists an absolute constant~$K$ such that the following holds. Let
  $x = \sum_{k=1}^{d_L} x_k\alpha_k\in \Z_L$. Then, given $\mu_k(x)$ for
  $1\leq k\leq d_L$ to precision~$N$ and given $M_L^{-1}$ to
  precision~$N + K d_L h(x)$, one can compute the coefficients
  $x_1,\ldots, x_k$ as complex numbers in time $\Otilde(d_L^2 N + d_L^2 h(x))$
  with a precision loss of $O(d_L \log d_L + \log \Delta_L)$ bits, and in
  particular recover $x\in \Z_L$ as soon as
  $N\geq K(d_L \log d_L + \log \Delta_L)$.
\end{lem}

\begin{proof}
  We consider the second matrix-vector product
  of~\eqref{eq:linear-embedding}. By definition of~$h$, we have
  $\sum_{k=1}^{d_L}\logp \abs{\mu_k(x)} = d_L h(x)$. Thus, we can compute the
  dot product of $(\mu_1(x),\ldots,\mu_k(x))$ with each line of~$M_L^{-1}$ in
  time $\Otilde(d_L h(x) + d_L^2 \log d_L + d_L\log \Delta_L)$.
\end{proof}

\subsection{Main theorems over finite fields}
\label{subsec:proof-ff}

In the case of a finite field, we are given a prime power $q=p^d$, and a monic
polynomial $P\in \Z[X]$ of degree~$d$, irreducible modulo~$p$, and we represent
elements of~$\F_q$ as elements of $\F_p[X]/(P)$. We have the following analogue
of \cref{thm:intro} in the Hilbert case.

\begin{thm}
  \label{thm:hilbert-ff}
  Let~$F = \Q(\sqrt{5})$ or $F = \Q(\sqrt{8})$. Then there exists an explicitly
  computable polynomial $Q\in \Z[G_1,G_2]$ such that the following
  holds. Let~$\beta\in \Z_F$ be a totally positive element, prime to the
  discriminant of~$F$, of prime norm $\ell\in \Z$. Given $g_1,g_2\in \F_q$ such
  that $Q(g_1,g_2)\neq 0$, \cref{algo:skeleton,algo:numerical-hilbert} together
  with \cref{lem:embed-1,lem:recognize-1} allow us to evaluate
  $\Psi_{\beta,k}(g_1,g_2,X)$ and $\partial_{G_n} \Psi_{\beta,k}(g_1,g_2,X)$
  for $1\leq k,n\leq 2$ as elements of~$\F_q[X]$ in time
  $\Otilde(\ell^2 \log q + \ell^2 d^2 \log\abs{P} +\ell d \log q + \ell
  d^3\log\abs{P})$.
\end{thm}

If $d\log\abs{P} = O(\log p)$, then the complexity estimate simplifies to
$\Otilde(\ell^2\log q + \ell d \log q)$ binary operations. In particular, the
algorithm runs in quasi-linear time in the output size when~$d$ is fixed, for
instance when $q=p$ is prime.

\begin{proof}
  We focus on the evaluation of $\Psi_{\beta,k}$ for $1\leq k\leq 2$; their
  partial derivatives can be treated similarly.  Let~$L$ be the number
  field~$\Q[X]/(P)$, and let~$\alpha$ be a root of~$P$ in~$L$. We lift~$g_1$
  and~$g_2$ to elements of~$\Z[\alpha]$ in such a way that the height of their
  coefficients is bounded above by~$\log p$. Then
  \begin{displaymath}
    \max\{h(g_1),h(g_2)\}
    \leq \log(p) + d h(\alpha) + \log(d) = O(d\log\abs{P} + \log p).
  \end{displaymath}

  Since~$Q_\beta$ and the coefficients of~$Q_\beta P_{\beta,k}$ for
  $1\leq k\leq 2$ are polynomials in $\Z[g_1,g_2]$ of degree~$O(\ell)$ and
  height~$O(\ell\log \ell)$, the algebraic numbers we have to recognize are all
  elements~$x\in\Z[\alpha]$ represented by polynomials
  $\widetilde{P}_x\in \Z[\alpha]$ of degree $O(\ell d)$ with coefficients of
  size $O(\ell \log \ell + \ell \log p + \log(d\ell))$. Applying the Euclidean
  algorithm, we see that the reminder~$P_x$ of~$\widetilde{P}_x$ modulo~$P$
  satisfies $\logp\abs{P_x} = O(\ell \log\ell + \ell\log p + d\ell \log\abs{P})$.

  We first backtrack to assess the necessary precision at the start of
  \cref{algo:numerical-hilbert}. To apply \cref{lem:recognize-1}, we need to
  compute $P_{\beta,k}(\mu(g_1),\mu(g_2),X)$ and $Q_{\beta}(\mu(g_1),\mu(g_2))$
  to precision
  $N = O(\ell\log \ell + \ell \log p + d\ell\log\abs{P} + d\log d)$ for all
  complex embeddings~$\mu$ of~$L$. By \cref{prop:cost-hilbert}, this can be
  done if $\mu(g_1)$ and $\mu(g_2)$ are first computed to precision
  $N' = \Otilde(d^2\log\abs{P} + d\log p + \ell \log p + d\ell \log\abs{P})$.

  We now analyze the cost of the whole algorithm. First, we apply
  \cref{lem:embed-1}: we need to compute the roots of~$P$ to precision
  $N' + O(\log p + d\log\abs{P})$, which can be done in~$\Otilde(d N')$ binary
  operations. Computing $\mu(g_1)$ and $\mu(g_2)$ to precision~$N'$ can be done
  in time $\Otilde(d N')$. After that, \cref{algo:numerical-hilbert} runs in
  time $\Otilde(d \ell N')$ by \cref{prop:cost-hilbert}. Finally, to apply
  \cref{lem:recognize-1}, we need to compute $1/P'(\mu_k(\alpha))$ for
  $1\leq k\leq d$ to precision~$N'$, which can be done in time $\Otilde(d N')$
  by \cref{lem:Pprime}, and then recovering $P_x$ costs $\Otilde(\ell d N')$ binary
  operations. Summing up the cost of each step yields the result.
\end{proof}

We note that the precision losses throughout \cref{algo:numerical-hilbert} are
dominated by the necessary precision to recognize integers at the very end, and
that steps~\ref{step:hilbert-loop} and~\ref{step:hilbert-prodtree} of
\cref{algo:numerical-hilbert} dominate the running time. This is also true in
practice.

In the Siegel case, we similarly have the following theorem, which implies
\cref{thm:intro}.

\begin{thm}
  \label{thm:siegel-ff}
  Let~$\ell$ be a prime number. Then, given Igusa invariants
  $j_1,j_2,j_3\in \F_q$ where the modular functions $\det(\mf{dJ})$ and
  $Q_\ell(\mf{j}_1,\mf{j}_2,\mf{j}_3)$ do not vanish,
  \cref{algo:skeleton,algo:numerical-siegel} together with
  \cref{lem:embed-1,lem:recognize-1} allow us to evaluate
  $\Psi_{\ell,k}(j_1,j_2,j_3, X)$ and
  $\partial_{J_n}\Psi_{\ell,k}(j_1,j_2,j_3,X)$ for $1\leq k,n\leq 3$ in time
  $\Otilde(\ell^6 \log q + \ell^6 d^2\log \abs{P} + \ell^3 d\log q +
  \ell^3d^3\log\abs{P})$.
\end{thm}

\subsection{Main theorems over number fields}
\label{subsec:proof-nf}

We now are in the setting of~§\ref{subsec:setting-2}: $L$ is a number field, we
have computed an HKZ-reduced basis of~$\Z_L$, and have access to an algorithm
which computes the matrices $M_L$ and~$M_L^{-1}$ to any desired precision.

\begin{thm}
  \label{thm:hilbert-nf}
  Let $F = \Q(\sqrt{5})$ or $F = \Q(\sqrt{8})$. Then there exists an explicitly
  computable polynomial $Q\in \Z[G_1,G_2]$ such that the following
  holds. Let~$\beta\in \Z_F$ be a totally positive prime, prime to the
  discriminant of~$F$, of prime norm $\ell\in \Z$. Let $H\geq 1$, and
  let~$g_1,g_2\in L$ be presented as quotients of integers of height at
  most~$H$ such that $Q(g_1,g_2)\neq 0$. Then
  \cref{algo:skeleton,algo:numerical-hilbert} together with
  \cref{lem:embed-2,lem:recognize-2} allow us to evaluate a certain nonzero
  denominator $D\in \Z_L$, as well as $D\Psi_{\beta,k}(g_1,g_2,X)$ and
  $D \partial_{G_n}\Psi_{\beta,k}(g_1,g_2,X)$ for $1\leq k,n\leq 2$ as elements
  of~$\Z_L[X]$, in time $\Otilde(\ell^2 d_L^2 H + \ell d_L^2 \log \Delta_L + \ell d_L^4 H + d_L^3\log \Delta_L)$.
\end{thm}

If~$d_L$ is fixed, then the complexity estimate in this theorem is simply
$\Otilde(\ell^2 H + \ell^2 \log \Delta_L)$ binary operations. In particular,
when $L=\Q$, we evaluate the Hilbert modular equations of level~$\beta$ at
rational numbers of height at most~$H$ in quasi-linear time
$\Otilde(\ell^2 H)$. We also note that the exponents of~$d_L$ in the complexity
estimate could be decreased if we rely on fast matrix algorithms in
\cref{assumption} and \cref{lem:embed-2,lem:recognize-2}, or if we were able to
compute columns of $M_L$ and $M_L^{-1}$ independently at different precisions.

\begin{proof}
  As above, we focus on the evaluation of~$\Psi_{\beta,k}$ for $k=1,2$: their
  partial derivatives can be treated similarly. Let~$a_1,a_2,b\in \Z_L$ be
  elements of height at most~$H$ such that $g_k = a_k/b$ for $k = 1,2$. We know
  that $Q_\beta$ and $P_{\beta,k}$ for $1\leq k\leq 2$ have total degree at
  most $\floor{10(\ell+1)/3}$ in~$G_1,G_2$ and height $O(\ell\log\ell)$, so we
  can take~$D = b^{\floor{10(\ell+1)/3}} Q_\beta(g_1,g_2)$. Then~$D$ and all
  coefficients of
  $D \Psi_{\beta,k}(g_1,g_2,X) = b^{\floor{10(\ell+1)/3}}
  P_{\beta,k}(g_1,g_2,X)$ for $1\leq k\leq 2$ are elements of~$\Z_L$ of height~$\Otilde(\ell H)$.

  As in \cref{thm:hilbert-ff}, we first backtrack to assess the necessary
  numerical precision at the start of \cref{algo:numerical-hilbert}.
  Let~$x\in \Z_L$ be one of the coefficients we wish to compute. To apply
  \cref{lem:recognize-2}, we need to compute $\mu_k(x)$ to precision
  $O(d_L \log d_L + \log \Delta_L)$ for each $1\leq k\leq d_L$. By
  \cref{prop:cost-hilbert}, this can be done if $\mu_k(g_1)$ and $\mu_k(g_2)$
  are computed to precision $N' = \Otilde(\ell d_L H + \log \Delta_L)$. Up to
  increasing the implied constants, it is sufficient to compute
  $\mu_k(a_1),\mu_k(a_2)$ and~$\mu_k(b)$ to precision~$N'$.

  We finally analyze the cost of the whole algorithm. First, to apply
  \cref{lem:embed-2}, we need to compute $M_L$ to precision~$N'$. This can be
  done in time $\Otilde(d_L^3 N')$ by \cref{assumption}. Then,
  computing~$\mu(g_1)$ and~$\mu(g_2)$ can be done in time $\Otilde(d_L^2
  N')$. After that, \cref{algo:numerical-hilbert} runs in time
  $\Otilde(\ell N')$ in each of the $d_L$ embeddings of~$L$. Finally, we apply
  \cref{lem:recognize-2} with $N = O(d_L \log d_L + \log \Delta_L)$: we compute
  $M_L^{-1}$ to precision~$N'$ in time $\Otilde(d_L^3 N')$, and then recovering
  each of the $O(\ell)$ coefficients costs~$\Otilde(d_L^2 N + \ell d_L^2 H)$
  binary operations. Summing up the cost of each step yields the result.
\end{proof}

In the Siegel case, we similarly obtain:

\begin{thm}
  \label{thm:siegel-nf}
  Let~$\ell$ be a prime number and $H\geq 1$. Then, given
  Igusa invariants $j_1,j_2,j_3\in L$ presented as quotients of integers of
  height at most~$H$ and where the modular functions $\det(\mf{dJ})$ and
  $Q_\ell(\mf{j}_1,\mf{j}_2,\mf{j}_3)$ do not vanish,
  \cref{algo:skeleton,algo:numerical-siegel} together with
  \cref{lem:embed-2,lem:recognize-2} allow us to evaluate a certain nonzero
  denominator $D\in \Z_L$ as well as $D\Psi_{\ell,k}(j_1,j_2,j_3,X)$ and
  $D \partial_{J_n}\Psi_{\ell,k}(j_1,j_2,j_3,X)$ for $1\leq k,n\leq 3$ as
  elements of~$\Z_L[X]$, in time
  $\Otilde(\ell^6 d_L^2 H + \ell^3 d_L^2 \log \Delta_L + \ell^3 d_L^4 H +
  d_L^3\log \Delta_L)$.
\end{thm}

\printbibliography

\end{document}